\documentclass[10pt]{article}
\usepackage{delarray} % LF: I have added this package.
\usepackage{amscd}
\usepackage{amssymb}
\usepackage{amsmath}
\usepackage{amsthm,amsxtra}
\usepackage{latexsym,amsmath,mathrsfs}
\usepackage[bookmarks,colorlinks]{hyperref}
\usepackage[all]{xy}
\usepackage{amsfonts}
\usepackage{color}
\usepackage{fancybox}
\usepackage{colordvi}
\usepackage{multicol}
\usepackage{textcomp}
\usepackage{stmaryrd}
\usepackage[dvips]{graphicx}

\usepackage{enumerate,xspace}
\usepackage{geometry}
\usepackage{tocbibind}
 \newcommand{\rmd}{{{\rm d}}}

\newcommand{\Tr}{{{\rm Tr}}}
\newcommand{\id}{{{\rm id}}}

\def\Robba{\mathscr{R}}
\def\sat{\mathrm{sat}}

\def\SI{\mathscr{I}}
\def\SS{\mathscr{S}}

\def\an{\mathrm{an}}

\def\cyc{\mathrm{cyc}}

\def\FX{\mathfrak{X}}
\def\Proj{\mathrm{Proj}}

\def\rd{\mathrm{d}}
\def\res{\mathrm{res}}

\newcommand*{\Z}{\mathbb{Z}}

\newcommand*{\C}{\mathbb{C}}
\newcommand*{\Gmul}{{\mathbb{G}_\textrm{m}}}

    \newcommand{\BC}{{\mathbb {C}}}

     \newcommand{\BN}{{\mathbb {N}}}
     
    \newcommand{\BQ}{{\mathbb {Q}}} \newcommand{\BR}{{\mathbb {R}}}

     \newcommand{\BZ}{{\mathbb {Z}}}

     \newcommand{\CF}{{\mathcal {F}}}

    \newcommand{\CO}{{\mathcal {O}}}

     \newcommand{\RB}{{\mathrm {B}}}
     \newcommand{\RD}{{\mathrm {D}}}
    \newcommand{\RE}{{\mathrm {E}}}

    \newcommand{\RM}{{\mathrm {M}}} 
     \newcommand{\RP}{{\mathrm {P}}}

     \newcommand{\RV}{{\mathrm {V}}}
    \newcommand{\RW}{{\mathrm {W}}}

    \newcommand{\End}{{\mathrm{End}}} 
    \newcommand{\Gal}{{\mathrm{Gal}}} \newcommand{\GL}{{\mathrm{GL}}}
    \newcommand{\Hom}{{\mathrm{Hom}}} \renewcommand{\Im}{{\mathrm{Im}}}

    \newcommand{\Ker}{{\mathrm{Ker}}}   \newcommand{\Lie}{{\mathrm{Lie}}}

             \newcommand{\im}{{\mathrm{im}}}
    
    \newcommand{\ord}{{\mathrm{ord}}} 
     
      \newcommand{\Res}{{\mathrm{Res}}}

     \newcommand{\Ext}{{\mathrm{Ext}}}

    \renewcommand{\mod}{\hskip 6pt \mathrm{mod} \hskip 3pt}

   \newcommand{\et}{\acute{\mathrm{e}}\mathrm{t}}

\newcommand{\cris}{\mathrm{cris}}

\newcommand{\unr}{\mathrm{unr}} \def\witt{\mathrm{W}}

\def\Rep{\mathrm{Rep}}

\def\SD{\mathscr{D}}
\def\SE{\mathscr{E}}
\def\SI{\mathscr{I}}
\def\SR{\mathscr{R}}

\def\ng{\mathrm{ng}}
\def\ncl{\mathrm{ncl}}

     \newcommand{\SL}{{\mathscr{L}}}

     \newcommand{\tr}{{\mathrm{tr}}}

    \newcommand{\st}{{\mathrm{st}}}
    
    \newcommand{\Mod}{{\mathrm{Mod}}}

    \theoremstyle{plain}

    \newtheorem{thm}{Theorem}[section] \newtheorem{cor}[thm]{Corollary}
    \newtheorem{lem}[thm]{Lemma} \newtheorem{sublem}[thm]{Sublemma}
    \newtheorem{prop}[thm]{Proposition}

    \theoremstyle{definition}

    \newtheorem{defn}[thm]{Definition}

    \theoremstyle{remark}
    \newtheorem {rem}[thm]{Remark}

    \numberwithin{equation}{section}

 \newcommand{\binc}[2]{ \bigg [\!\! \begin{array}{c} #1\\
    #2 \end{array}\!\! \bigg ]}

\begin{document}

\title{Triangulable $\CO_F$-analytic $(\varphi_q,\Gamma)$-modules of rank 2}
\author{
Lionel Fourquaux \\
Universit\'e Rennes~1, France \\
lionel.fourquaux+cohomlt2012@normalesup.org
\\
Bingyong Xie
\\
Department of Mathematics, East China Normal University,
\\
Shanghai, 200241, P. R. China\\
byxie@math.ecnu.edu.cn}
\date{}\maketitle

\begin{abstract} The theory of
$(\varphi_q,\Gamma)$-modules is a generalization of Fontaine's
theory of $(\varphi,\Gamma)$-modules, which classifies
$G_F$-representations on $\CO_F$-modules and $F$-vector spaces for
any finite extension $F$ of $\BQ_p$. In this paper following
Colmez's method we classify triangulable $\CO_F$-analytic
$(\varphi_q,\Gamma)$-modules of rank $2$. In this process we
establish two kinds of cohomology theories for $\CO_F$-analytic
$(\varphi_q,\Gamma)$-modules. Using them we show that, if $D$ is an
$\CO_F$-analytic $(\varphi_q,\Gamma)$-module such that
$D^{\varphi_q=1,\Gamma=1}=0$ i.e. $V^{G_F}=0$ where $V$ is the
Galois representation attached to $D$, then any overconvergent
extension of the trivial representation of $G_F$ by $V$ is
$\CO_F$-analytic. In particular, contrarily to the case of
$F=\BQ_p$, there are representations of $G_F$ that are not
overconvergent.
\end{abstract}

\tableofcontents

\section*{Introduction}

The present paper heavily depends on the theory of
$(\varphi,\Gamma)$-modules for Lubin-Tate extensions, a
generalization of Fontaine's theory of $(\varphi,\Gamma)$-modules.
The existence of this generalization was more or less implicit in
\cite{Fontaine91, Col-2002}. See also \cite{Fou} and \cite[Remark
2.3.1]{Sch}. In \cite{Kisin-Ren}, Kisin and Ren provided details,
where $(\varphi,\Gamma)$-modules for Lubin-Tate extensions are
called $(\varphi_q,\Gamma)$-modules.

To recall this theory, let $F$ be a finite extension of $\BQ_p$,
$\CO_F$ the ring of integers in $F$ and $\pi$ a unifomizer of
$\CO_F$. Fix an algebraic closure of $F$ denoted by $\bar{F}$, and
put $G_F=\Gal(\bar{F}/F)$. Let $k_F$ be the residue field of $F$,
$q=\# k_F$. Let $\witt=\witt(k_F)$ be the ring of Witt vectors over
$k_F$, $F_0=\witt[1/p]$. Then $F_0$ is the maximal absolutely
unramified subfield of $F$. Let $\CF$ be a Lubin-Tate group over $F$
corresponding to the uniformizer $\pi$. Then $\CF$ is a formal
$\CO_F$-module. Let $X$ be a local coordinate on $\CF$. Then the
formal Hopf algebra $\CO_\CF$ may be identified with $\CO_F[[X]]$.
For any $a\in \CO_F$, let $[a]_\CF\in \CO_F[[X]]$ be the power
series giving the endomorphism $a$ of $\CF$. If $n\geq 1$, let
$F_n\subset \bar{F}$ be the subfield generated by the
$\pi^n$-torsion points of $\CF$. Write $F_\infty=\cup_n F_n$,
$\Gamma = \Gal(F_\infty/F)$ and
$G_{F_\infty}=\Gal(\bar{F}/F_\infty)$. For any integer $n\geq 0$,
let $\Gamma_n\subset \Gamma$ be the subgroup $\Gal(F_\infty/F_n)$.
Let $T\CF$ be the Tate module of $\CF$. It is a free $\CO_F$-module
of rank $1$. The action of $G_F$ on $T\CF$ factors through $\Gamma$
and induces an isomorphism $\chi_\CF: \Gamma\rightarrow
\CO_F^\times$. For any $a\in \CO_F^\times$ we write
$\sigma_a:=\chi_\CF^{-1}(a)$. Using the periods of $T\CF$, one can
construct a ring $\CO_\SE$ with actions of
$\varphi_q=\varphi^{\log_pq}$ and $\Gamma$. We will recall the
construction in Section \ref{sec:KR}. Kisin and Ren \cite{Kisin-Ren}
defined \'etale $(\varphi_q,\Gamma)$-modules over $\CO_\SE$ and
classified $G_F$-representations on $\CO_F$-modules in terms of
these modules.

In this paper we are interested in triangulable $\CO_F$-analytic
$(\varphi_q,\Gamma)$-modules over a Robba ring $\SR_L$, where $L$ is
a finite extension of $F$. A {\it triangulable
$(\varphi_q,\Gamma)$-module} over $\SR_L$ means a
$(\varphi_q,\Gamma)$-module $D$ that has a filtration consisting of
$(\varphi_q,\Gamma)$-submodules $0=D_0\subset D_1\subset \cdots
\subset D_d=D$ such that $D_i/D_{i-1}$ is free of rank $1$ over
$\SR_L$.

In the spirit of Colmez's work \cite{tri} on the classification of
triangulable $(\varphi,\Gamma)$-modules of rank $2$, in the present
paper we will classify triangulable $\CO_F$-analytic
$(\varphi_q,\Gamma)$-modules over $\SR_L$ of rank 2. One motivation
for doing this, is that the authors believe that under the
hypothetical $p$-adic local Langlands correspondence these
$(\varphi_q,\Gamma)$-modules should correspond to certain unitary
principal series of $\GL_2(F)$. Colmez \cite{Col-an} and
Liu--Xie--Zhang \cite{LXZ} respectively determined the spaces of
locally analytic vectors of the unitary principal series of
$\GL_2(\BQ_p)$ based on this kind of $(\varphi,\Gamma)$-modules. Our
computations of dimensions of $\Ext^1_\an$ match those of Kohlhaase
on extensions of locally analytic representations \cite{Koh}.
Nakamura \cite{Nakamura} gave a generalization of Colmez's work in
another direction. But we think that Nakamura's point of view is
probably not the best one for applications to the $p$-adic local
Langlands correspondence.

For our purpose we consider two kinds of cohomology theories for
$\CO_F$-analytic $(\varphi_q,\Gamma)$-modules.

For a $(\varphi_q,\Gamma)$-module $D$ over $\SR_L$, we define
$H^\bullet(D)$ by the cohomology of the semigroup
$\varphi_q^{\BN}\times \Gamma$ as in \cite{Col-an}. Then the first
cohomology group $H^1(D)$ is isomorphic to $\Ext(\SR_L, D)$, the
$L$-vector space of extensions of $\SR_L$ by $D$ in the category of
$(\varphi_q,\Gamma)$-modules.

If $D$ is $\CO_F$-analytic, we consider the following complex
$$C^\bullet_{\varphi_q,\nabla}(D): \hskip 10pt
\xymatrix{ 0\ar[r] & D \ar[r]^{f_1} & D \oplus D \ar[r]^{f_2} & D
\ar[r] & 0 } ,
$$ where $f_1: D\rightarrow D \oplus D$ is the map defined as $ m \mapsto ((\varphi_q-1)m,
\nabla m)$ and   $f_2: D \oplus D\rightarrow D$ is $(m, n)\mapsto
\nabla m-(\varphi_q-1)n$. The operator $\nabla$ is defined in
Section \ref{ss:an}. Put
$H^i_{\varphi_q,\nabla}(D):=H^i(C^\bullet_{\varphi_q,\nabla}(D))$,
$i=0,1,2$. Each of these modules admits a $\Gamma$-action. We set
$H_\an^i(D)=H^i_{\varphi_q,\nabla}(D)^\Gamma.$

\begin{thm}\label{thm:intro-coh}
Let $D$ be an $\CO_F$-analytic $(\varphi_q,\Gamma)$-module over $\SR_L$.
Then there is a natural isomorphism $\Ext_\an (\SR_L, D)\rightarrow
H_\an^1(D)$, where $\Ext_\an(\SR_L,D)$ is the $L$-vector space that
consists of extensions of $\SR_L$ by $D$ in the category of
$\CO_F$-analytic $(\varphi_q,\Gamma)$-modules.
\end{thm}

The proof is given in Section \ref{sec:coh}, which is due to the
referee and much simpler than that in our original version.

\begin{thm}\label{thm:ext}
Let $D$ be an $\CO_F$-analytic $(\varphi_q,\Gamma)$-module over $\SR_L$.
Then $\Ext_\an(\SR_L, D)$ is of codimension $([F:\BQ_p]-1)\dim_L
D^{\varphi_q=1,\Gamma=1}$ in $\Ext(\SR_L, D)$. In particular, if
$D^{\varphi_q=1,\Gamma=1}=0$, then $\Ext_\an(\SR_L,D) = \Ext(\SR_L,
D)$.
\end{thm}

To prove Theorem \ref{thm:ext}, we will construct a (non canonical)
projection from $\Ext(\SR_L,D)$ onto $\Ext_\an(\SR_L,D)$ whose
kernel is of dimension $([F:\BQ_p]-1)\dim_L
D^{\varphi_q=1,\Gamma=1}$.

If $V$ is an overconvergent $L$-representation of $G_F$ (in the
sense of Definition \ref{defn:over}), $\Delta$ is the
$(\varphi_q,\Gamma)$-module over $\SE^\dagger_L$ attached to $V$,
and $D=\SR_L\otimes_{\SE^\dagger_L}\Delta$, then $\Ext(\SR_L,D)$
measures the set of extensions of the trivial representation by $V$
that are overconvergent (cf. Proposition \ref{prop:faithful} and
Proposition \ref{th:ked}). Theorem \ref{thm:ext} tells us that, if
$V^{G_F}=D^{\varphi_q=1,\Gamma=1}=0$, then any such extension is
$\CO_F$-analytic.

Let $\SI(L)$ (resp. $\SI_\an(L)$) be the set of continuous (resp.
locally $F$-analytic) characters $\delta: F^\times \rightarrow
L^\times$. Let $\delta_\unr$ denote the character of $F^\times$ such
that $\delta_\unr(\pi)=q^{-1}$ and $\delta_\unr|_{\CO^\times_F}=1$.
Then $\delta_\unr$ is a locally $F$-analytic character. If
$\delta\in\SI(L)$, let $\SR_L(\delta)$ be the
$(\varphi_q,\Gamma)$-module over $\SR_L$ of rank $1$ that has a
basis $e_\delta$ such that $\varphi_q(e_\delta)=\delta(\pi)e_\delta$
and $\sigma_a(e_\delta)=\delta(a)e_\delta$. If
$\delta\in\SI_\an(L)$, then $\SR_L(\delta)$ is $\CO_F$-analytic.

For locally $F$-analytic characters we have the following

\begin{thm}\label{thm:intro-dim} For any $\delta\in \SI_\an(L)$, we have
$$ \dim_L H_\an^1(\SR_L(\delta))=\left\{\begin{array}{ll} 2 &
\text{ if } \delta=x^{-i}, i\in \BN \text{ or } x^i\delta_\unr , i \in \BZ_+  \\
 1 & \text{ otherwise,} \end{array}\right. $$ and
$$ \dim_L H^1(\SR_L(\delta))=\left\{\begin{array}{ll} [F:\BQ_p]+1 & \text{ if } \delta=x^{-i}, i\in \BN \\ 2 &
\text{ if } \delta = x^i\delta_\unr , i \in \BZ_+  \\
 1 & \text{ otherwise.} \end{array}\right. $$
\end{thm}

For the proof of Theorem \ref{thm:intro-dim} we follow Colmez's
method. In his paper \cite{tri} Colmez used the theory of $p$-adic
Fourier transform for $\BZ_p$. For our case we use the $p$-adic
Fourier transform for $\CO_F$ developed by Schneider and Teitelbaum
\cite{ST1} instead. But this transform can not be applied to our
situation directly because, except for the case of $F=\BQ_p$, it is
defined over $\BC_p$ and can not be defined over any finite
extension $L$ of $F$. We overcome this difficulty by applying it to
$\SR_{\BC_p}$ and then descending certain results to $\SR_L$. As a
result, we obtain that, if $\delta_1$ and $\delta_2$ are in
$\SI_\an(L)$, then $\SR_L(\delta_1)^{\psi=0}$ and
$\SR_L(\delta_2)^{\psi=0}$ are isomorphic to each other as
$L[\Gamma]$-modules. This is exactly what we need. In fact, we will
show that $S_\delta := (\SR_L e_\delta/\SR^+_L e_\delta)^{\psi=0
,\Gamma=1}$ is $1$-dimensional over $L$ for any $\delta\in
\SI_\an(L)$, and that $ H_\an^1(\SR_L(\delta)) $ is isomorphic to
$S_\delta$ when $v_\pi(\delta(\pi))<1-v_\pi(q)$ and $\delta$ is not
of the form $x^i$.

For characters that are not locally $F$-analytic we have the
following

\begin{thm}\label{thm:intro-nonanal} For any $\delta\in \SI(L)\backslash
\SI_\an(I)$ we have $H^1(\SR_L(\delta))=0$. Consequently every
extension of $\SR_L$ by $\SR_L(\delta)$ splits.
\end{thm}

To state our result on the classification, we need some parameter
spaces. These parameter spaces are analogues of Colmez's parameter
spaces \cite{tri}. Let $\SS$ be the analytic variety over
$\SI_\an(L)\times \SI_\an(L)$ whose fiber over $(\delta_1,\delta_2)$
is isomorphic to $\Proj(H^1(\delta_1\delta_2^{-1}))$,  $\SS_\an$ the
analytic variety over $\SI_\an(L)\times \SI_\an(L)$ whose fiber over
$(\delta_1,\delta_2)$ is isomorphic to
$\Proj(H^1_\an(\delta_1\delta_2^{-1}))$. There is a natural
inclusion $ \SS_\an \hookrightarrow \SS $.  Let $\SS_+, \SS^\an_+,
\SS_+^\ng, \SS_+^\cris, \SS_+^\st, \SS_+^\ord$ and $ \SS_+^\ncl$ be
the subsets of $\SS$ defined in Section \ref{sec:tri}. We can assign
to any $s\in \SS$ (resp. $s\in \SS_\an$) a triangulable (resp.
triangulable and $\CO_F$-analytic) $(\varphi_q,\Gamma)$-module
$D(s)$.

\begin{thm} \label{thm:cl}
\begin{enumerate}
\item \label{it:cl-1}
For $s\in \SS$, $D(s)$ is of slope zero if and only if $s$ is in $
\SS_+ -\SS_+^\ncl$; $D(s)$ is of slope zero and the Galois
representation attached to $D(s)$ is irreducible if and only if $s$
is in $\SS_* - (\SS_\ast^\ord\cup \SS_\ast^\ncl)$; $D(s)$ is of
slope zero and $\CO_F$-analytic if and only if $s$ is in $\SS_+^\an
- \SS_+^\ncl$.
\item \label{it:cl-2}
Let $s=(\delta_1,\delta_2,\SL)$ and $s'=(\delta'_1,\delta'_2,\SL')$
be in $\SS_+-\SS_+^\ncl$. If $\delta_1=\delta'_1$, then $D(s)\cong
D(s')$ if and only if $s=s'$. If $\delta_1\neq \delta'_1$, then
$D(s)\cong D(s')$ if and only if $s,s'\in \SS_+^\cris\cup
\SS_+^\ord$ with $\delta'_1=x^{w(s)}\delta_2$,
$\delta'_2=x^{-w(s)}\delta_1$.
\end{enumerate}
\end{thm}

In the case when $F=\BQ_p$, this becomes Colmez's result \cite{tri}.
The proof of Theorem \ref{thm:cl} will be given at the end of
Section \ref{sec:tri}. \vskip 10pt

We give another application of Theorem \ref{thm:intro-dim}. In the
case of $F=\BQ_p$, i.e. the cyclotomic extension case, Cherbonnier
and Colmez \cite{Col-Ch} showed that all representations of
$G_{\BQ_p}$ are overconvergent. But our following result shows that
this is not the case when $[F:\BQ_p]\geq 2$.

\begin{thm} \label{prop:not-over-0}
Suppose that $[F:\BQ_p]\geq 2$. Then there exist $2$-dimensional
$L$-representations of $G_F$ that are not overconvergent (in the
sense of Definition \ref{defn:over}).
\end{thm}

By Kedlaya's Theorem \cite{Kedlaya}, any $(\varphi_q,\Gamma)$-module
of slope zero $D(s)$ in Theorem \ref{thm:cl} (\ref{it:cl-1}) comes
from a $2$-dimensional $L$-representation of $G_F$ that is
overconvergent.

\vskip 5pt

We outline the structure of this paper. We recall Fontaine's rings,
the theory of $(\varphi_q,\Gamma)$-modules and the relation between
$(\varphi_q,\Gamma)$-modules and Galois representations in Section
\ref{sec:fontaine-ring} and Section \ref{ss:galois}, and then define
$\CO_F$-analytic $(\varphi_q,\Gamma)$-modules over the Robba ring
$\SR_L$ in Section \ref{ss:an}.  We define $\psi$ in Section
\ref{ss:psi}, and study the properties of $\partial$ and $\Res$ in
Section \ref{ss:partial}. In Section \ref{ss:psi-ext} we extend
$\psi$ to $\SR_{\BC_p}$, in Section \ref{ss:m} we define operators
$m_\alpha$ on $\SR_{\BC_p}$, and then in Section \ref{ss:psi0} we
study the $\Gamma$-action on $\SR_L(\delta)^{\psi=0}$ for all
$\delta\in\SI_\an(L)$. The cohomology theories for $\CO_F$-analytic
$(\varphi_q,\Gamma)$-modules are given in Section \ref{sec:coh}. In
Section \ref{sec:comp} we compute $H_\an^1(\SR_L(\delta))$ and
$H^1(\SR_L(\delta))$ for all $\delta\in \SI_\an(L)$. After providing
preliminary lemmas in Section \ref{ss:pre}, we compute $H^0(\delta)$
for all $\delta\in\SI(L)$ and $H^1_\an(\delta)$ for $\delta\in
\SI_\an(L)$ satisfying $v_\pi(\delta(\pi))<1-v_\pi(q)$ respectively
in Section \ref{ss:H0} and Section \ref{ss:H1an}. For the purpose of
computing $H^1_\an(\delta)$ for all $\delta\in\SI_\an(L)$, we
construct a transition map $\partial:
H^1_\an(x^{-1}\delta)\rightarrow H^1_\an(\delta)$, which is done in
Section \ref{ss:H1an-par}. The computation of $H^1_\an(\delta)$ is
given in Section \ref{ss:H1}. In section \ref{ss:iota-k} we define
two maps $\iota_k$ and $\iota_{k,\an}$. Applying results in Section
\ref{sec:comp} we classify triangulable $\CO_F$-analytic
$(\varphi_q,\Gamma)$-modules in Section \ref{sec:tri}.

\section{$(\varphi_q,\Gamma)$-modules and $\CO_F$-analytic
$(\varphi_q,\Gamma)$-modules} \label{sec:KR}

In this section we recall the theory of $(\varphi_q,\Gamma)$-modules
built in \cite{Col-2002, Fou, Kisin-Ren}. We keep using notations in
the introduction.

\subsection{The rings of formal series} \label{sec:fontaine-ring}

Put
$\widetilde{\RE}^+=\lim\limits_{\overleftarrow{\;\;\;\;\;\;}}\CO_{\bar{F}}/p$
with the transition maps given by Frobenius, and let
$\widetilde{\RE}$ be the fractional field of $\widetilde{\RE}^+$. We
may also identify $\widetilde{\RE}^+$ with
$\lim\limits_{\overleftarrow{\;\;\;\;\;\;}}\CO_{\bar{F}}/\pi$ with
the transition maps given by the $q$-Frobenius
$\varphi_q=\varphi^{\log_pq}$. Evaluation of $X$ at $\pi$-torsion
points induces a map $\iota: T\CF\rightarrow \widetilde{\RE}^+$.
Precisely, if $v=(v_n)_{n\geq 0}\in T\CF$ with $v_n
\in\CF[\pi^n](\CO_{\bar{F}})$ and $\pi\cdot v_{n+1}=v_n$, then
$\iota(v)=(v_n^*(X)+\pi\CO_{\bar{F}})_{n\geq 0}$.

Let $\{\cdot\}$ be the unique lifting map $\widetilde{\RE}^+
\rightarrow \RW(\widetilde{\RE}^+)_F:=
\witt(\widetilde{\RE}^+)\otimes_{\CO_{F_0}}\CO_F$ such that
$\varphi_q\{x\}=[\pi]_\CF(\{x\})$ (see \cite[Lemma 9.3]{Col-2002}).
When $\CF$ is the cyclotomic Lubin-Tate group $\Gmul$, we have
$\{x\}=[1+x]-1$, where $[1+x]$ is the Teichm\"uller lifting of
$1+x$. This map respects the action of $G_F$. If $v\in T\CF$ is an
$\CO_F$-generator, there is an embedding
$\CO_F[[u_\CF]]\hookrightarrow \witt(\widetilde{\RE}^+)_F$ sending
$u_\CF$ to $\{\iota(v)\}$ which identifies $\CO_F[[u_\CF]]$ with a
$G_F$-stable and $\varphi_q$-stable subring of
$\witt(\widetilde{\RE}^+)_F$. The $G_F$-action on $\CO_F[[u_\CF]]$
factors through $\Gamma$. By \cite[Lemma 9.3]{Col-2002} we have
$$ \varphi_q(u_\CF)=[\pi]_\CF(u_\CF), \hskip 10pt \sigma_a(u_\CF)=[a]_\CF(u_\CF).
$$ In the case of $\CF=\Gmul$, $u_\CF$ is denoted by $T$ in
\cite{tri}. Here $T$ is used to denote the Tate module of a
Lubin-Tate group.

Let $\CO_\SE$ be the $\pi$-adic completion of
$\CO_F[[u_\CF]][1/u_\CF]$. Then $\CO_\SE$ is a complete discrete
valuation ring with uniformizer $\pi$ and residue field
$k_F((u_\CF))$. The topology induced by this valuation is called the
{\it strong topology}. Usually we consider the {\it weak topology}
on $\CO_\SE$, i.e. the topology with $\{\pi^i\CO_\SE +
u_\CF^j\CO_F[[u_\CF]]: \ i,j\in\mathbb{N} \}$ as a fundamental
system of open neighborhoods of $0$. Let $\SE$ be the field of
fractions of $\CO_\SE$. Let $\SE^+$ be the subring
$F\otimes_{\CO_F}\CO_F[[u_\CF]]$ of $\SE$.

For any $r\in \BR_+\cup\{+\infty\}$, let $\SE^{]0,r]}$ be the ring
of Laurent series $f=\sum_{i\in\BZ}a_iu_\CF^i$ with coefficients in
$F$ that are convergent on the annulus $0< v_p(u_\CF)\leq r$. For
any $0<s\leq r$ we define the valuation $v^{\{s\}}$ on $\SE^{]0,r]}$
by
$$v^{\{s\}}(f)=\inf_{i\in\BZ} (v_p(a_i)+is) \in
\BR\cup\{\pm \infty\}.$$ We equip $\SE^{]0,r]}$ with the Fr\'echet
topology defined by the family of valuations $\{v^{\{s\}}: 0<s\leq
r\}$. Then $\SE^{]0,r]}$ is complete. We equip the Robba ring
$\SR:=\cup_{r>0}\SE^{]0,r]}$ with the inductive limit topology. The
subring of $\SR$ consisting of Laurent series of the form
$\sum_{i\geq 0}a_iu_\CF^i$ is denoted by $\SR^+$.

Put $ \SE^\dagger:=\{ \sum_{i\in\BZ}a_i u_\CF^i\in \SR \ |\ a_i
\text{ are bounded when }i\rightarrow +\infty \}.$ This is a field
contained in both $\SE$ and $\SR$. Put $\SE^{(0,r]}=\SE^\dagger\cap
\SE^{]0,r]}$. Let $v^{[0,r]}$ be the valuation defined by
$v^{[0,r]}(f)=\min_{0\leq s\leq r}v^{\{s\}}(f)$. Let
$\CO_{\SE^{(0,r]}}$ be the ring of integers in $\SE^{(0,r]}$ for the
valuation $v^{[0,r]}$. We equip $\CO_{\SE^{(0,r]}}[1/u_\CF]$ with
the topology induced by the valuation $v^{\{r\}}$ and then equip
$\SE^{(0,r]}=\cup_{m\in \BN} \pi^{-m} \SE^{(0,r]}[1/u_\CF]$ with the
inductive limit topology. The resulting topology on $\SE^{(0,r]}$ is
called the {\it weak topology} \cite{Col-prin}. Note that the
restriction of the weak topology to the subset $\{ f(u_\CF) = \sum
_{i\in\BZ} a_i u_\CF^i \in\SE^{(0,r]}: a_i=0 \text{ if } i\geq 0 \}$
coincides with the topology defined by the valuation $v^{\{r\}}$ and
its restriction to $\SE^+$ coincides with the weak topology on
$\SE^+$. Then we equip $\SE^\dagger=\cup_{r> 0}\SE^{(0,r]}$ with the
inductive limit topology.

We extend the actions of $\varphi_q$ and $\Gamma$ on
$\CO_F[[u_\CF]]$ to $\SE^+$, $\CO_\SE$, $\SE$, $\SE^\dagger$ and
$\SR$ continuously.

Put $t_\CF=\log_\CF(u_\CF)$, where $\log_\CF$ is the logarithmic of
$\CF$. Then $t_\CF$ is in $\SR$ but not in $\SE^\dagger$.  When
$\CF=\Gmul$, $t_\CF$ coincides with the usual $t$ in \cite{tri}.
Note that $\varphi_q(t_\CF)=\pi t_\CF$ and $\sigma_a(t_\CF)=a t_\CF$
for any $a\in \CO_F^\times$. Put
$Q=Q(u_\CF)=[\pi]_\CF(u_\CF)/u_\CF$.

We have the following analogue of \cite[Lemma I.3.2]{berger04}.

\begin{lem} \label{lem:Berger}
If $I$ is a $\Gamma$-stable principal ideal of $\SR^+$, then $I$ is
generated by an element of the form
$u_\CF^{j_0}\prod\limits_{n=0}^{+\infty}
\Big(\varphi_q^{n}(Q(u_\CF)/Q(0))\Big)^{j_{n+1}}$. Furthermore the
following hold:
\begin{enumerate}
\item If $ \SR^+\cdot\varphi_q(I)\subseteq  I$, then the sequence $\{j_n\}_{n\geq 0}$
is decreasing.
\item If $ \SR^+\cdot\varphi_q(I)\supseteq  I$, then the sequence $\{j_n\}_{n\geq 0}$
is increasing.
\end{enumerate}
\end{lem}
\begin{proof} The argument is similar to the proof of
\cite[Lemma I.3.2]{berger04}. Let $f(u_\CF)$ be a generator of $I$.
For any $\rho \in (0,1)$ put $V_\rho(I)=\{ z \in \BC_p : f(z)=0 ,
0\leq |z| \leq \rho \}$. If $I$ is stable by $\Gamma$, then
$V_\rho(I)$ is stable by $[a]_\CF$ for any $a \in \CO^\times_F$. As
$V_\rho(I)$ is finite, for any $z\in V_\rho(I)$, there must be some
element $a\in \CO^\times_F$, $a\neq 1$ such that $[a]_\CF(z)=z$.
Note that $[\pi]_\CF(z)$ satisfies
$[a]_\CF([\pi]_\CF(z))=[\pi]_\CF(z)$ if $[a]_\CF(z)=z$. But the
cardinal number of the set $\{ z \in \BC_p : [a]_\CF(z)=z, |z|\leq
\rho\}$ is finite. Thus for any $z\in V_I(\rho)$ there exists a
positive integer $m=m(\rho)$ such that $[\pi^m]_\CF(z)=0$. Therefore
$I$ is generated by an element of the form
$u_\CF^{j_0}\prod\limits_{n=0}^{+\infty}
(\varphi_q^n(Q(u_\CF)/Q(0)))^{j_{n+1}}$. The other two assertions
are easy to prove.
\end{proof}

\begin{cor} We have
\begin{equation} \label{eq:t-decom}
(t_\CF) = \Big( u_\CF \prod_{n\geq 0}
\varphi_q^{n}(Q(u_\CF)/Q(0))\Big)
\end{equation}
in the ring $\SR^+$.
\end{cor}
\begin{proof}  Because the
ideal $(t_\CF)$ is $\Gamma$-invariant and $\SR^+ \cdot
\varphi_q(t_\CF) =(t_\CF)$, by Lemma \ref{lem:Berger} there exists
$j\in \BN$ such that $(t_\CF) = \Big(u_\CF^j \prod\limits_{n\geq 0}
\varphi_q^n\big(Q(u_\CF)/Q(0)\big)^j\Big)$. From the fact
$(t_\CF/u_\CF)\equiv 1\mod u_\CF\SR^+$ we obtain $j=1$.
\end{proof}

If $\CF'$ is another Lubin-Tate group over $F$ corresponding to
$\pi$, by the theory of Lubin-Tate groups there exists a unique
continuous ring isomorphism $\eta_{\CF,\CF'}:
\CO_{\SE_{\CF}}^+\rightarrow \CO_{\SE_{\CF'}}^+$ with
$$ \eta_{\CF,\CF'}(u_\CF) = u_{\CF'}+ \text{ higher degree terms in } \CO_F [[u_{\CF'}]]$$ such
that $\eta_{\CF,\CF'}\circ [a]_\CF=[a]_{\CF'}\circ \eta_{\CF,\CF'}$
for all $a\in \CO_F $. We extend $\eta_{\CF,\CF'}$ to isomorphisms
$$\CO_{\SE_{\CF}}
\xrightarrow{\sim}\CO_{\SE_{\CF'}}, \ \ %
\SE^+_{\CF}
\xrightarrow{\sim}\SE^+_{\CF'}, \ \ %
\SE_{\CF}  \xrightarrow{\sim} \SE_{\CF'} , \ \
\SE^\dagger_{\CF}\rightarrow \SE^\dagger_{\CF'},  \ \
\SR_{\CF}\rightarrow \SR_{\CF'}.$$ By abuse of notations these
isomorphisms are again denoted by $\eta_{\CF,\CF'}$.

Let $\ell_u=\log u_\CF$ be a variable over $\SR[1/t_\CF]$. We extend
the $\varphi_q,\Gamma$-actions to $\SR[1/t_\CF,\ell_u]$ by
$$ \varphi_q(\ell_u) = q\ell_u +\log\frac{[\pi]_\CF(u_\CF)}{u_\CF^q}, \hskip 10pt
\sigma_a (\ell_u) = \ell_u +\log \frac{[a]_\CF(u_\CF)}{u_\CF} . $$

\subsection{Galois representations and $(\varphi_q,\Gamma)$-modules}
\label{ss:galois}

Let $L$ be a finite extension of $F$. Let $\Rep_{L} G_F$ be the
category of finite dimensional $L$-vector spaces $V$ equipped with a
linear action of $G_F$.

If $A$ is any of $\SE^+$, $\SE$, $\SE^\dagger$, $\SR$, we put
$A_L=A\otimes_{F}L$. Then we extend the $\varphi_q$,
$\Gamma$-actions on $A$ to $A_L$ by $L$-linearity. Let $R$ denote
any of $\SE_L$, $\SE^\dagger_L$ and $\SR_L$. For a
$(\varphi_q,\Gamma)$-module over $R$, we mean a free $R$-module $D$
of finite rank together with continuous semilinear actions of
$\varphi_q$ and $\Gamma$ commuting with each other such that
$\varphi_q$ sends a basis of $D$ to a basis of $D$. When $R=\SE_L$,
we say that $D$ is {\it \'etale} if $D$ has a $\varphi_q$-stable
$\CO_{\SE_L}$-lattice $M$ such that the linear map
$\varphi_q^*M\rightarrow M$ is an isomorphism. When
$R=\SE^\dagger_L$, we say that $D$ is {\it \'etale} if
$\SE_L\otimes_{\SE_L^\dagger}D$ is \'etale. When $R=\SR_L$, we say
that $D$ is {\it \'etale} or {\it of slope $0$} if there exists an
\'etale $(\varphi_q,\Gamma)$-module $\Delta$ over $\SE_L^\dagger$
such that $D=\SR_L\otimes_{\SE_L^\dagger}\Delta$. Let
$\Mod^{\varphi_q,\Gamma, \et}_{/R}$ be the category of \'etale
$(\varphi_q,\Gamma)$-modules over $R$.

Put $\widetilde{\RB}= \witt( \widetilde{\RE} )_F [1/\pi]$. Let $\RB$
be the completion of the maximal unramified extension of $\SE$ in
$\widetilde{\RB}$ for the $\pi$-adic topology. Both
$\widetilde{\RB}$ and $\RB$ admit actions of $\varphi_q$ and $G_F$.
We have $\RB^{G_{F_\infty}}=\SE$.

For any $V\in\Rep_L G_F $, put $\RD_{\SE}(V)=( \RB \otimes_F
V)^{G_{F_\infty}}.$ For any
$D\in\Mod^{\varphi_q,\Gamma,\et}_{/\SE_L}$, put $\RV(D)=( \RB
\otimes_{\SE}D)^{\varphi_q=1}.$

\begin{thm} \label{th:kisin} $($Kisin-Ren \cite[Theorem 1.6]{Kisin-Ren}$)$ The functors $\RV$ and $\RD_{\SE}$ are
quasi-inverse equivalences of categories between  $\Mod^{\varphi_q,
\Gamma, \et}_{/\SE_L}$ and  $\Rep_L G_F $.
\end{thm}

As usual, let $\widetilde{\RB}^\dagger$ be the subring of
$\widetilde{\RB}$ consisting of overconvergent elements, and put
$\RB^\dagger=\RB\cap \widetilde{\RB}^\dagger$. Then
$(\RB^\dagger)^{G_{F_\infty}}=\SE^\dagger$.

\begin{defn}\label{defn:over} If $V$ is an $L$-representation of $G_F $, we say that
$V$ is {\it overconvergent} if
$\RD_{\SE^\dagger}(V):=(\RB^\dagger\otimes_F  V)^{G_{F_\infty}}$
contains a basis of $\RD_\SE(V)$.
\end{defn}

When $F=\BQ_p$, according to Cherbonnier-Colmez theorem
\cite{Col-Ch}, all $L$-representations are overconvergent. But in
general this is not true. For details see Remark \ref{rem:non-over}.

% However we have the following

\begin{prop} \label{prop:faithful}
\begin{enumerate}
\item \label{it:over-0} If $\Delta$ is an \'etale $(\varphi_q,\Gamma)$-module
over $\SE^\dagger_L$, then
$\RV(\SE_L\otimes_{\SE^\dagger_L}\Delta)=(\RB^\dagger\otimes_{\SE^\dagger}\Delta)^{\varphi_q=1}$.
\item \label{it:over-1} The functor $\Delta\mapsto \SE_L\otimes_{\SE_L^\dagger}\Delta$ is
a fully faithful functor from the category
$\Mod^{\varphi_q,\Gamma,\et}_{/\SE_L^\dagger}$ to the category
$\Mod^{\varphi_q,\Gamma,\et}_{/\SE_L}$.
\item \label{it:over-2} The functor $\RD_{\SE^\dagger}$ is an equivalence of
categories between the category of overconvergent
$L$-representations of $G_F $ and
$\Mod^{\varphi_q,\Gamma,\et}_{/\SE_L^\dagger}$.
\end{enumerate}
\end{prop}
\begin{proof} Without loss of generality we may assume that $L=F$.
Put $\widetilde{\RB}_{\BQ_p}=\witt(\widetilde{\RE})[1/p]$ and
$\widetilde{\RB}_{\BQ_p}^\dagger=\widetilde{\RB}_{\BQ_p}\cap
\widetilde{\RB}^\dagger$. The technics of almost \'etale descent as
in Berger-Colmez \cite{BerCol} allows us to show that the functor
$\Delta\mapsto
\widetilde{\RB}_{\BQ_p}\otimes_{\widetilde{\RB}^\dagger_{\BQ_p}}\Delta$
from the category of \'etale $(\varphi, G_F)$-modules over
$\widetilde{\RB}^\dagger_{\BQ_p}$ to the category of \'etale
$(\varphi, G_F)$-modules over $\widetilde{\RB}_{\BQ_p}$ is an
equivalence.  For any $(\varphi_q,G_F)$-module $D$ over
$\widetilde{\RB}^\dagger$ (resp. $\widetilde{\RB}$), we can attach a
$(\varphi,G_F)$-module $\bar{D}$ over
$\widetilde{\RB}^\dagger_{\BQ_p}$ (resp. $\widetilde{\RB}_{\BQ_p}$)
to $D$ by letting $\bar{D}=\oplus_{i=0}^{f-1}\varphi^{i\ast}(D)$
with the map
$$\varphi^\ast(\bar{D})=\oplus_{i=1}^{f}\varphi^{i\ast}(D)\rightarrow
\oplus_{i=0}^{f-1}\varphi^{i\ast}(D)=\bar{D}$$ that sends
$\varphi^{i\ast}(D)$ identically to $\varphi^{i\ast}(D)$ for
$i=1,\cdots,f-1$, and sends $\varphi^{f\ast}(D)=\varphi_q^\ast(D)$
to $D$ using $\varphi_q$. Here $f=\log_pq$. Thus the functor
$\alpha: \Delta\mapsto
\widetilde{\RB}\otimes_{\widetilde{\RB}^\dagger}\Delta$ from the
category of \'etale $(\varphi_q, G_F)$-modules over
$\widetilde{\RB}^\dagger$ to the category of \'etale $(\varphi_q,
G_F)$-modules over $\widetilde{\RB}$ is an equivalence. Now let
$\Delta$ be an \'etale $(\varphi_q,\Gamma)$-module over
$\SE^\dagger$, and put $V=\RV(\SE\otimes_{\SE^\dagger}\Delta)$. As
$\alpha(\widetilde{\RB}^\dagger\otimes_F V )=
\widetilde{\RB}\otimes_F V =
\widetilde{\RB}\otimes_{\SE^\dagger}\Delta
=\alpha(\widetilde{\RB}^\dagger\otimes_{\SE^\dagger}\Delta)$, we
have $\widetilde{\RB}^\dagger\otimes_F V  =
\widetilde{\RB}^\dagger\otimes_{\SE^\dagger}\Delta$. Thus $V$ is
contained in $\widetilde{\RB}^\dagger\otimes_{\SE^\dagger}\Delta\cap
\RB\otimes_{\SE^\dagger}\Delta=
\RB^\dagger\otimes_{\SE^\dagger}\Delta$, and
$V=(\RB^\dagger\otimes_{\SE^\dagger}\Delta)^{\varphi_q=1}$. This
proves (\ref{it:over-0}).

Next we prove (\ref{it:over-1}). Let $\Delta_1$ and $\Delta_2$ be
two objects in $\Mod^{\varphi_q,\Gamma,\et}_{/\SE^\dagger}$. What we
have to show is that the natural map
$$\Hom_{\Mod^{\varphi_q,\Gamma,\et}_{/\SE^\dagger}}(\Delta_1,\Delta_2)\rightarrow
\Hom_{\Mod^{\varphi_q,\Gamma,\et}_{/\SE}}(\SE\otimes_{\SE^\dagger}\Delta_1,
\SE\otimes_{\SE^\dagger}\Delta_2)$$ is an isomorphism. For this we
reduce to show that
$$\Big(\check{\Delta}_1\otimes_{\SE^\dagger}
\Delta_2\Big)^{\varphi_q=1,\Gamma=1}\rightarrow \Big(
\SE\otimes_{\SE^\dagger} (\check{\Delta}_1\otimes_{\SE^\dagger}
\Delta_2)\Big)^{\varphi_q=1,\Gamma=1}$$ is an isomorphism. Here
$\check{\Delta}_1$ is the $\SE^\dagger$-module of
$\SE^\dagger$-linear maps from $\Delta_1$ to $\SE^\dagger$, which is
equipped with a natural \'etale $(\varphi_q,\Gamma)$-module
structure. We have {\allowdisplaybreaks \begin{eqnarray*} \Big(
\SE\otimes_{\SE^\dagger} (\check{\Delta}_1\otimes_{\SE^\dagger}
\Delta_2)\Big)^{\varphi_q=1,\Gamma=1} &=& \Big(
\RB\otimes_{\SE^\dagger} (\check{\Delta}_1\otimes_{\SE^\dagger}
\Delta_2)\Big)^{\varphi_q=1,G_F=1}  \\ &=& \RV(
\SE\otimes_{\SE^\dagger}
(\check{\Delta}_1\otimes_{\SE^\dagger} \Delta_2))^{G_F=1} \\
& = & \Big( \RB^\dagger\otimes_{\SE^\dagger}
(\check{\Delta}_1\otimes_{\SE^\dagger}
\Delta_2)\Big)^{\varphi_q=1,G_F=1} \\ &=&
(\check{\Delta}_1\otimes_{\SE^\dagger}
\Delta_2)^{\varphi_q=1,\Gamma=1}. \end{eqnarray*}}

Finally, (\ref{it:over-2}) follows from (\ref{it:over-0}),
(\ref{it:over-1}) and Theorem \ref{th:kisin}.
\end{proof}

\begin{prop}\label{th:ked} The functor $\Delta\mapsto \SR_L\otimes_{\SE_L^\dagger}\Delta$ is
an equivalence of categories between
$\Mod^{\varphi_q,\Gamma,\et}_{/\SE_L^\dagger}$ and
$\Mod^{\varphi_q,\Gamma,\et}_{/\SR_L}$.
\end{prop}
\begin{proof} Let $D$ be an \'etale $(\varphi_q,\Gamma)$-module over
$\SR_L$. By Kedlaya's slope filtration theorem \cite{Kedlaya}, there
exists a unique $\varphi_q$-stable $\SE^\dagger_L$-submodule
$\Delta$ of $D$ that is \'etale as a $\varphi_q$-module such that
$D=\SR_L\otimes_{\SE_L^\dagger}\Delta$. For any $\gamma\in \Gamma$,
$\gamma(\Delta)$ also has this property. Thus, by uniqueness of
$\Delta$, we have $\gamma(\Delta)=\Delta$. This means that $\Delta$
is $\Gamma$-invariant.
\end{proof}

\subsection{$\CO_F $-analytic $(\varphi_q,\Gamma)$-modules}
\label{ss:an}

For any $r\geq s>0$, let $v^{[s,r]}$ be the valuation defined by
$v^{[s,r]}(f)=\inf_{r'\in [s,r]}v^{\{r'\}}(f)$. Note that
$v^{[s,r]}(f)=\inf\limits_{z\in \BC_p, s\leq v_p(z)\leq
r}v_p(f(z))$.

\begin{lem}\label{lem:gamma-conv} For any $r>s>0$, there exists a
sufficiently large integer $n=n(s,r)$ such that, if $\gamma\in
\Gamma_n$, then we have $v^{[s,r]}((1-\gamma)z)\geq v^{[s,r]}(z)+1$
for all $z\in\SE_L^{]0,r]}$.
\end{lem}
\begin{proof} It suffices to consider $z=u_\CF^k$, $k\in\BZ$.
If $k\geq 0$, then
$$ \gamma(u_\CF^k)-u_\CF^k =u_\CF^k(\frac{\gamma(u_\CF)}{u_\CF}-1)
(\frac{\gamma(u_\CF^{k-1})}{u_\CF^{k-1}}+\cdots +1) $$ and
$$\gamma(u_\CF^{-k})-u_\CF^{-k}=u_\CF^{-k} (\frac{u_\CF}{\gamma(u_\CF)}-1)
(\frac{u_\CF^{k-1}}{\gamma(u_\CF^{k-1})}+\cdots +1).$$ As
$v^{[s,r]}(yz)\geq v^{[s,r]}(y)+v^{[s,r]}(z)$, the lemma follows
from the fact that $\frac{\gamma(u_\CF)}{u_\CF}-1 \rightarrow 0$
when $\gamma\rightarrow 1$.
\end{proof}

Let $D$ be an object in $ \Mod^{\varphi_q,\Gamma,\et}_{/\SR_L}$. We
choose a basis $\{e_1,\cdots, e_d\}$ of $D$ and write
$D^{]0,r]}=\oplus_{i=1}^d \SE_L^{]0,r]} \cdot e_i$. Note that our
definition of $D^{]0,r]}$ depends on the choice of $\{e_1,\cdots,
e_d\}$. However, if $\{e'_1,\cdots, e'_d\}$ is another basis, then
$\oplus_{i=1}^d \SE_L^{]0,r]} \cdot e_i=\oplus_{i=1}^d \SE_L^{]0,r]}
\cdot e'_i$ for sufficiently small $r>0$. When $r>0$ is sufficiently
small, $D^{]0,r]}$ is stable under $\Gamma$. By Lemma
\ref{lem:gamma-conv} and the continuity of the $\Gamma$-action on
$D^{]0,r]}$, the series
$$\log\gamma = \sum_{i=1}^{\infty}(\gamma-1)^i(-1)^{i-1}/i$$
converges on $D^{]0,r]}$ when $\gamma\rightarrow 1$. It follows that
the map
$$ \rd\Gamma:\Lie \Gamma \rightarrow \End_L D^{]0,r]},
\hskip 10pt \beta\mapsto \log(\exp \beta)$$ is well defined for
sufficiently small $\beta$, and we extend it to all of $\Lie \Gamma$
by $\BZ_p$-linearity. As a result, we obtain a $\BZ_p$-linear map
$\rd\Gamma_D: \Lie  \Gamma \rightarrow \End_{L} D$. For any
$\beta\in \Lie \Gamma$, $\rd \Gamma_{\SR_L}(\beta)$ is a derivation
of $\SR_L$ and $\rd \Gamma_D (\beta)$ is a differential operator
over $\rd\Gamma_{\SR_L}(\beta)$, which means that for any $a\in
\SR_L$, $m\in D$ and $\beta\in \Lie \Gamma$ we have
\begin{equation} \label{eq:diff}
\rd\Gamma_D(\beta)(am) = \rd \Gamma_{\SR_L}(\beta)(a)m + a \cdot
\rd\Gamma_D(\beta)(m) .
\end{equation}

The isomorphism $\chi_\CF:\Gamma\rightarrow \CO_F^\times$ induces an
$\CO_F$-linear isomorphism $\Lie\Gamma\rightarrow \CO_F$. We will
identify $\Lie\Gamma$ with $\CO_F$ via this isomorphism.

We say that $D$ is {\it $\CO_F $-analytic} if the map $\rd\Gamma_D$
is not only $\BZ_p$-linear, but also $\CO_F $-linear. If $D$ is
$\CO_F $-analytic, the operator $\rd\Gamma_D(\beta)/\beta$, $\beta
\in \CO _F , \ \beta \neq 0$, does not depend on the choice of
$\beta$. The resulting operator is denoted by $\nabla_D$ or just
$\nabla$ if there is no confusion.
Note that the $\Gamma$-action on $\SR_L$ is $\CO_F $-analytic and by
\cite[Lemma 2.1.4]{Kisin-Ren}
\begin{equation} \label{eqn-partial-0}
\nabla =  t_\CF \cdot \frac{\partial F_\CF}{\partial Y}(u_\CF,0)
\cdot \rd/\rd u_\CF,
\end{equation} where $F_\CF(X,Y)$ is the formal group law of $\CF$. Put
$\partial=\frac{\partial F_\CF}{\partial Y}(u_\CF,0) \cdot \rd/\rd
u_\CF$. From the relation $\sigma_a(t_\CF)=a  t_\CF$ we obtain $
\nabla t_\CF = t_\CF $ and $ \partial \:  t_\CF =1$. When
$\CF=\Gmul$, $\nabla$ and $\partial$ are already defined in
\cite{berger02}. In this case $F_\CF(X,Y)=X+Y+XY$ and so
$\partial=(1+u_\CF)\rd/\rd u_\CF$.

We end this section by classification of
$(\varphi_q,\Gamma)$-modules over $\SR_L$ of rank $1$.

Let $\SI(L)$ be the set of continuous characters $\delta:
F^\times\rightarrow L^\times$, $\SI_\an(L)$ the subset of locally
$F$-analytic characters. If $\delta$ is in $\SI_\an(L)$, then
$\frac{\log \delta(a)}{\log (a)}$, $a\in \CO^\times_F $, which makes
sense when $\log(a)\neq 0$, does not depend on $a$. This number,
denoted by $w_\delta$, is called the {\it weight} of $\delta$.
Clearly $w_\delta=0$ if and only if $\delta$ is locally constant;
$w_\delta$ is in $\BZ$ if and only if $\delta$ is locally algebraic.

If $\delta\in\SI(L)$, let $\SR_L(\delta)$ be the
$(\varphi_q,\Gamma)$-module over $\SR_L$ (of rank $1$) that has a
basis $e_\delta$ such that $\varphi_q(e_\delta)=\delta(\pi)e_\delta$
and $\sigma_a(e_\delta)=\delta(a)e_\delta$. It is easy to check
that, if $\delta\in \SI_\an(L)$, then $\SR_L(\delta)$ is $\CO_F
$-analytic and $\nabla_\delta=\nabla_{\SR_L(\delta)} = t_\CF
\partial +w_\delta$ (more precisely $\nabla_\delta(z
e_\delta)=(t_\CF
\partial z +w_\delta z)e_\delta$). If $\SR_L(\delta)$ is \'etale, i.e. $v_p(\delta(\pi))=0$, we will
use $L(\delta)$ to denote the Galois representation attached to
$\SR_L(\delta)$.

\begin{rem} \label{rem:1-over} All of $1$-dimensional $L$-representations of $G_F$ are
overconvergent. In fact, such a representation comes from a
character of $F^\times$ and thus is of the form $L(\delta)$.
\end{rem}

\begin{prop}\label{prop:rank-one} Let $D$ be a $(\varphi_q,\Gamma)$-module over $\SR_L$ of rank $1$.
Then there exists a character $\delta\in \SI(L)$ such that $D$ is
isomorphic to $\SR_L(\delta)$. Furthermore $D$ is $\CO_F$-analytic
if and only if $\delta\in \SI_\an(L)$.
\end{prop}
\begin{proof}  The argument is similar to the proof of \cite[Proposition
3.1]{tri}. We first reduce to the case that $D$ is \'etale. Then by
Proposition \ref{th:ked} there exists an \'etale
$(\varphi_q,\Gamma)$-module $\Delta$ over $\SE_L^\dagger$ such that
$D=\SR_L\otimes_{\SE^\dagger_L}\Delta$. Now the first assertion
follows from Proposition \ref{prop:faithful} and Remark
\ref{rem:1-over}. The second assertion is obvious.
\end{proof}

\section{The operators $\psi$ and $\partial$} \label{sec:two-op}

\subsection{The operator $\psi$} \label{ss:psi}

We define an operator $\psi$ and study its properties.

Note that  $\{u_\CF^i\}_{0\leq i\leq q-1}$ is a basis of $\SE_L$
over $\varphi_q(\SE_L)$. So $\SE_L$ is a field extension of
$\varphi_q(\SE_L)$ of degree $q$. Put
$\tr=\tr_{\SE_L/\varphi_q(\SE_L)}$.

\begin{lem} \label{lem:psi} ~
\begin{enumerate}
\item \label{it:psi-1} There is a unique operator $\psi: \SE_L \rightarrow \SE_L$ such that
                       $\varphi_q\circ\psi=q^{-1}\tr$.
\item \label{it:psi-3} For any $a, b\in \SE_L$ we have
                       $\psi(\varphi_q(a)b)=a\psi(b)$. In
                       particular, $\psi\circ \varphi_q=\id$.
\item \label{it:psi-4} $\psi$ commutes with $\Gamma$.
\end{enumerate}
\end{lem}
\begin{proof}
Assertion (\ref{it:psi-1}) follows from the fact that $\varphi_q$ is
injective. % For any $a\in \SE_L$, $\varphi_q\circ \psi\circ % \varphi_q(a)=\tr/q(\varphi_q(a))=\varphi_q(a)$.
Assertion (\ref{it:psi-3}) follows from the relation
$$ \varphi_q(\psi(\varphi_q(a)b)) =\tr(\varphi_q(a)b)/q=
\varphi_q(a) \tr(b)/q = \varphi_q(a)
\varphi_q(\psi(b))=\varphi_q(a\psi(b))   $$ and the injectivity of
$\varphi_q$. As $\varphi_q$ commutes with $\Gamma$,
$\varphi_q(\SE_L)$ is stable under $\Gamma$. Thus $\gamma\circ \tr
\circ \gamma ^{-1}=\tr$ for all $\gamma\in \Gamma$. This ensures
that $\psi$ commutes with $\Gamma$. Assertion (\ref{it:psi-4})
follows.
\end{proof}

We first compute $\psi$ in the case of the special Lubin-Tate group.

\begin{prop} \label{prop:psi-sp} Suppose that $\CF$ is the special Lubin-Tate group.
\begin{enumerate}
\item \label{it:psi-sp-pos}
If $\ell\geq 0$, then $
\psi(u_\CF^\ell)=\sum_{i=0}^{[\ell/q]}a_{\ell,i}u_\CF^i$ with
$v_\pi(a_{\ell,i})\geq  [\ell/q]+1-i-v_\pi(q). $
\item \label{it:psi-sp-neg}
If $\ell<0$, then
$\psi(u_\CF^{\ell})=\sum_{i=\ell}^{[\ell/q]}b_{\ell,i}u_\CF^i$ with
$v_\pi(b_{\ell,i})\geq [\ell/q]+1-i-v_\pi(q)$.
\end{enumerate}
\end{prop}
\begin{proof}
First we prove (\ref{it:psi-sp-pos}) by induction on $\ell$. As the
minimal polynomial of $u_\CF$ is $X^q +\pi X - (u_\CF^q+ \pi
u_\CF)$, by Newton formula we have
$$\tr(u_\CF^i)=\left\{\begin{array}{ll} 0 & \text{ if } 1\leq i\leq q-2 , \\
(1-q)\pi & \text{ if } i=q-1. \end{array}\right.
$$ It follows that
$$ \psi(u_\CF^i) =\left\{\begin{array}{ll} 0 & \text{ if } 1\leq i\leq q-2 , \\
(1-q)\pi/q  & \text{ if } i=q-1. \end{array} \right. $$ Thus the
assertion holds when $0\leq \ell\leq q-1$. Now we assume that
$\ell=j\geq q$ and the assertion holds when $0\leq \ell \leq j-1$.
We have {\allowdisplaybreaks
\begin{eqnarray*}\psi(u_\CF^\ell) &=& \psi((u_\CF^q+\pi
u_\CF)u_\CF^{\ell-q})-\psi(\pi u_\CF^{\ell-q+1})  =
u_\CF\psi(u_\CF^{\ell-q})-\pi\psi(u_\CF^{\ell-q+1})
\\ &=&
\sum_{i=1}^{[\ell/q]}a_{\ell-q, i-1}u_\CF^i
-\sum_{i=0}^{[(\ell+1)/q]-1}\pi  a_{\ell-q+1,i}u_\CF^i
.\end{eqnarray*}} Thus $a_{\ell,i}=a_{\ell-q,i-1}-\pi
a_{\ell-q+1,i}$. By the inductive assumption we have
\begin{eqnarray*}
v_\pi(a_{\ell-q,i-1})  \geq   [(\ell-q)/q]+1-(i-1) -v_\pi(q)  =
[\ell/q]+1-i-v_\pi(q)
\end{eqnarray*} and
\begin{eqnarray*} v_\pi(a_{\ell-q+1,i}) \geq  [(\ell-q+1)/q]+1-i
-v_\pi(q) \geq  [\ell/q]-i -v_\pi(q) \end{eqnarray*} It follows that
$v_\pi(a_{\ell,i})\geq [\ell/q]+1-i-v_\pi(q)$.

Next we prove (\ref{it:psi-sp-neg}). We have {\allowdisplaybreaks
\begin{eqnarray*} \psi(u_\CF^\ell) &=&
\psi\Big(\frac{(u_\CF^{q-1}+\pi)^{-\ell}}{\varphi_q(u_\CF)^{-\ell}}\Big)
= \frac{\psi\Big(\sum_{j=0}^{-\ell} \binc{-\ell}{j}
u_\CF^{j(q-1)}\pi^{-\ell-j}  \Big)}{u_\CF^{-\ell}} \\
&=& \sum_{i=0}^{[-\ell(q-1)/q]} \sum_{j=0}^{-\ell}
\binc{-\ell}{j}\pi^{-\ell-j} a_{j(q-1),i} \cdot u_\CF^{i+\ell} =
\sum_{i=\ell}^{[\ell/q]} \sum_{j=0}^{-\ell}
\binc{-\ell}{j}\pi^{-\ell-j} a_{j(q-1),i-\ell} \cdot u_\CF^{i}
\end{eqnarray*}} Here, $\binc{-\ell}{j}=\frac{(-\ell)!}{j!(-\ell-j)!}$. Thus $ b_{\ell,i}=  \sum\limits_{j=0}^{-\ell}
\binc{-\ell}{j}\pi^{-\ell-j} a_{j(q-1),i-\ell} $. As
{\allowdisplaybreaks
\begin{eqnarray*}v_\pi(\pi^{-\ell-j}
a_{j(q-1),i-\ell}) & \geq & -\ell-j
+([\frac{j(q-1)}{q}]+1-(i-\ell)-v_\pi(q))
\\ &=& [-j/q]+1-i -v_\pi(q) \geq  [\ell/q]+1-i-v_\pi(q),
\end{eqnarray*}} we obtain $v_\pi(b_{\ell,i})\geq
[\ell/q]+1-i-v_\pi(q)$.
\end{proof}

Let $\SE_L^-$ be the subset of $\SE_L$ consisting of elements of the
form $\sum\limits_{i\leq -1} a_i u_\CF^i$.

\begin{cor}\label{cor:psi-sp} Suppose that $\CF$ is the special Lubin-Tate group.
Then $\psi(\SE_L^-)\subset \SE_L^-$.
\end{cor}
\begin{proof} This follows directly from Proposition
\ref{prop:psi-sp}.
\end{proof}

% From the proof we see that $\psi(1/u)=\pi_L/(qu)$.

\begin{prop} \label{prop:psi-con}
\begin{enumerate}
\item \label{it:psi-con-1}
We have $\psi(\SE_L^+)=\SE_L^+$, $\psi(\CO_{\SE_L^+})\subset
\frac{\pi}{q}\CO_{\SE_L^+}$ and $\psi(\CO_{\SE_L})\subset
\frac{\pi}{q}\CO_{\SE_L}$.
\item \label{it:psi-con-2}
$\psi$ is continuous for the weak topology on $\SE_L$.
\item \label{it:psi-con-3}
$\SE_L^\dagger$ is stable under $\psi$, and the restriction of
$\psi$ on $\SE_L^\dagger$ is continuous for the weak topology of
$\SE_L^\dagger$.
\item \label{it:psi-bon} If $f\in\SE_L^{(0,r]}$, then the sequence
$(\frac{q}{\pi}\psi)^n(f)$, $n\in \BN$, is bounded in
$\SE_L^{(0,r]}$ for the weak topology.
\end{enumerate}
\end{prop}
\begin{proof}
Let $\CF_0$ be the special Lubin-Tate group over $F$ corresponding
to $\pi$. Observe that
$\psi_{\CF}=\eta_{\CF_0,\CF}^{-1}\psi_{\CF_0}\eta_{\CF_0,\CF}$. As
$\eta_{\CF_0,\CF}(u_{\CF_0}) = u_\CF \times $ a unit in
$\CO_F[[u_\CF]]$, for any $r>0$ we have that
$\eta_{\CF_0,\CF}(\CO_{\SE_{\CF_0,L}}^{(0,r]}[1/u_{\CF_0}])=\CO_{\SE_{\CF,L}}^{(0,r]}[1/u_{\CF}]$
and that $\eta_{\CF_0,\CF}$ respects the valuation $v^{[0,r]}$. Thus
$\eta_{\CF_0,\CF}: \SE_{\CF_0,L}^{(0,r]}\rightarrow
\SE_{\CF,L}^{(0,r]}$ is a topological isomorphism. It follows that
$\SE^\dagger_{\CF_0,L}\rightarrow \SE^\dagger_{\CF,L}$ and its
inverse are continuous for the weak topology. Similarly
$\eta_{\CF_0,\CF}: \SE_{\CF_0,L}\rightarrow \SE_{\CF,L}$ and its
inverse are continuous for the weak topology. Hence we only need to
consider the case of the special Lubin-Tate group. Assertions
(\ref{it:psi-con-1}) and (\ref{it:psi-con-2}) follow from
Proposition \ref{prop:psi-sp}. For (\ref{it:psi-con-3}) we only need
to show that, for any $r>0$ we have $\psi(\SE_L^{(0,r]})\subset
\SE_L^{(0,r]}$ and the restriction $\psi:
\SE_L^{(0,r]}\rightarrow\SE_L^{(0,r]}$ is continuous. By
(\ref{it:psi-con-2}) the restriction of $\psi$ to $\SE_L^+$ is
continuous. By Proposition \ref{prop:psi-sp} (\ref{it:psi-sp-neg})
and Corollary \ref{cor:psi-sp}, if $f$ is in $\SE_L^-\cap
\SE_L^{(0,r]}$, then $\psi(f)$ is in $\SE_L^-$ and
$v^{\{r\}}(\psi(f))\geq v^{\{r\}}(f)+v_p(\pi/q)$. Thus $\psi:
\SE_L^-\cap \SE_L^{(0,r]}\rightarrow \SE_L^-\cap \SE_L^{(0,r]}$ is
continuous, which proves  (\ref{it:psi-con-3}). As
$\frac{q}{\pi}\psi(\CO_{\SE_L^+})\subset \CO_{\SE_L^+}$ and
$v^{\{r\}}(\frac{q}{\pi}\psi(f))\geq v^{\{r\}}(f)$ for any $f\in
\SE_L^-\cap \SE_L^{(0,r]}$,  (\ref{it:psi-bon}) follows.
\end{proof}

Next we extend $\psi$ to $\SR_L$.

\begin{prop}\label{prop:tr} We can extend $\tr$ continuously to
$\SR_L$. The resulting operator $\tr$ satisfies
$\tr|_{\varphi_q(\SR_L)}=q\cdot \id$ and
$\tr(\SR_L)=\varphi_q(\SR_L)$.
\end{prop}
\begin{proof} Let $\SE_{L}^{\gg -\infty}$ denote the subset of $\SE_L$
consisting of $f\in \SE_L$ of the form $\sum_{n\gg -\infty} a_n
u_\CF^n$. If $f\in \SE_{L}^{\gg-\infty}$, then
$$\tr(f)=\sum_{\eta \in \ker [\pi]_\CF} f(u_\CF +_\CF \eta).$$
If $\eta$ is in $\ker[\pi]_\CF$, then $v_p(\eta)\geq
\frac{1}{(q-1)e_F}$ where $e_F=[F:F_0]$.  Thus, if $r$ and
$s\in\BR_+$ satisfy $\frac{1}{(q-1) e_F}> r\geq s$, the morphisms
$u_\CF\mapsto u_\CF+_\CF \eta$ ($\eta\in \ker[\pi]_\CF$) keep the
annulus $ \{ z\in \BC_p : p^{-r} \leq |z| \leq p^{-s}\} $ stable. So
for any $f\in \SE_L^{\gg -\infty}$ we have $ v^{ \left[ s, r \right]
} (f(u_\CF+_\CF \eta)) = v^{\left[s, r\right]}(f)$ and $v^{ \left[
s, r \right] } (\tr(f))\geq v^{\left[ s, r \right]}(f)$. Hence there
exists a unique continuous operator $\Tr: \SR_L\rightarrow \SR_L$
such that $\Tr(f)=\tr(f)$ for any $f\in \SE^{\gg -\infty}_L$. (For
any $f\in \SR_L$, choosing a positive real number $r$ such that
$f\in\SE^{]0,r]}_L$, we can find a sequence $\{f_i\}_{i\geq 1}$ in
$\SE^{\gg -\infty}_L$ such that $f_i\rightarrow f$ in $\SE^{]0,r]}$;
then $\{\tr(f_i)\}_{i\geq 1}$ is a Cauchy sequence in
$\SE^{[s,r]}_L$ for any $s$ satisfying $0<s\leq r$, and we let
$\Tr(f)$ be their limit in $\SE^{]0,r]}$; it is easy to show that
$\Tr(f)$ does not depend on any choice.) From the continuity of
$\Tr$ we obtain that $\Tr|_{\SE^\dagger_L}=\tr$ and
$\Tr|_{\varphi_q(\SR_L)}=q\cdot \id$. By Lemma
\ref{lemme-continuite-phi-gamma} below, $\varphi_q:\SR_L\rightarrow
\SR_L$ is strict and thus has a closed image. Since $\SE^\dagger_L$
is dense in $\SR_L$ and
$\Tr(\SE^\dagger_L)=\varphi_q(\SE^\dagger_L)\subset
\varphi_q(\SR_L)$, we have $\Tr(\SR_L)\subseteq \varphi_q(\SR_L)$.
\end{proof}

\begin{lem}\label{lemme-continuite-phi-gamma}
If \( \frac{q}{(q-1) e_F}> r \geq s >0\) and $ f \in \SE_L^{]0,r]}$,
then we have
\begin{itemize}
\item \(v^{\left[s, r\right]}(\gamma(f)) = v^{\left[s, r\right]}(f)\)
for all \(\gamma \in \Gamma\);
\item \(v^{\left[s,r\right]}(\varphi_q(f)) = v^{\left[qs, qr\right]}(f)\)
if \(r<  \frac{1}{(q-1) e_F}\).
\end{itemize}
\end{lem}
\begin{proof}
Since \([\chi_\CF(\gamma)]_\CF(u_\CF) \in u_\CF \CO_F[[u_\CF]]\), we
have \(v_p([\chi_\CF(\gamma)]_\CF(z) ) \geqslant v_p(z)\) for all
\(z \in \BC_p\) such that \(v_p(z) > 0\). By the same reason we have
\(v_p([\chi_\CF(\gamma^{-1})]_\CF(z) ) \geqslant v_p(z)\) and thus
\(v_p([\chi_\CF(\gamma)]_\CF(z) ) \leqslant v_p(z)\). So
\(v_p([\chi_\CF(\gamma)]_\CF(z)) = v_p(z)\).

If $z\in \BC_p$ satisfies \(p^{-\frac{1}{(q-1) e_F}} < p^{-r}
\leqslant \left\lvert z \right\rvert \leqslant p^{-s} < 1\), then
\(v_p([\pi]_\CF(z)) = q v_p(z)\). Thus, the image by \(z \mapsto
[\pi]_\CF(z)\) of the annulus \( \{ z\in \BC_p: p^{-r} \leqslant
\left\lvert z \right\rvert \leqslant p^{-s} \} \) is inside the
annulus \( \{ z\in\BC_p: p^{-qr} \leqslant \left\lvert z
\right\rvert \leqslant p^{-qs} \} \). Conversely, if \(w \in \BC_p\)
is such that \( p^{-qr} \leqslant \left\lvert w \right\rvert
\leqslant p^{-qs}\), then \(v_p(w) < \frac{q}{(q-1) e_F}\). The
Newton polygon of the polynomial \(-w + [\pi]_\CF(u_\CF)\) shows
that this polynomial has \(q\) roots of valuation \(\frac{1}{q}
v_p(w)\). If \(z \in \BC_p\) is such a root, we have \( p^{-r}
\leqslant \left\lvert z \right\rvert \leqslant p^{-s}\). Thus, the
image of the annulus \( p^{-r} \leqslant \left\lvert z \right\rvert
\leqslant p^{-s}\) is the annulus \( p^{-qr} \leqslant \left\lvert z
\right\rvert \leqslant p^{-qs}\).
\end{proof}

We define $\psi:\SR_L\rightarrow \SR_L$ by $\psi=\frac{1}{q}\
\varphi_q^{-1}\circ \tr$.

\begin{lem}
\label{lemme-continuite-psi} If \( \frac{q}{(q-1) e_F}> r\geq s>0 \)
and \(f \in \SE_L^{]0,r]}\), then \(v^{\left[s,r\right]}(\psi(f))
\geqslant
    v^{\left[s/q, r/q\right]}(f) - v_p(q)\).
\end{lem}
\begin{proof} By
Lemma~\ref{lemme-continuite-phi-gamma} it suffices to show that
$$
v^{\left[s/q,r/q\right]}(\varphi_q(\psi(f)))=v^{\left[s/q,r/q\right]}(q^{-1}\tr(f))\geq
    v^{\left[s/q, r/q\right]}(f) - v_p(q) . $$ But this follows from
Proposition \ref{prop:tr} and its proof.
\end{proof}

As a consequence, $\psi:\SR_L\rightarrow \SR_L$ is continuous.

\begin{cor}\label{cor:dagger}
\begin{enumerate}
\item\label{it:dagger} $\{u_\CF^i\}_{0\leq i\leq q-1}$ is a basis of $\SE^\dagger_L$
over $\varphi_q (\SE^\dagger_L)$, and
$\tr|_{\SE^\dagger_L}=\tr_{\SE^\dagger_L/\varphi_q(\SE^\dagger_L)}$.
\item $\{u_\CF^i\}_{0\leq i\leq q-1}$ is a basis of $\SR_L$
over $\varphi_q (\SR_L)$.
\end{enumerate}
\end{cor}
\begin{proof}
Let $\{b_i\}_{0\leq i\leq q-1}$ be the dual basis of
$\{u_\CF^i\}_{0\leq i\leq q-1}$ relative to
$\tr_{\SE_L/\varphi_q(\SE_L)}$. Let $B$ be the inverse of the matrix
$(\tr(u_\CF^{i+j}))_{i,j}$. Then $B\in \GL_{q}(\SE_L^\dagger)$ and
$(b_0,b_1,\cdots, b_{q-1})^t=B(1, u_\CF, \cdots, u_\CF^{q-1})^t$. So
$b_0,b_1,\cdots, b_{q-1}$ are in $\SE_L^\dagger$. Then
$f=\sum_{i=0}^{q-1}u_\CF^i \psi(b_i f)$ for any $f\in \SE_L$,
$\SE^\dagger_L$ or $\SR_L$. (For the former two cases, this follows
from the definition of $\{b_i\}_{0\leq i\leq q-1}$; for the last
case, we apply the continuity of $\psi$.) Thus $\{u_\CF^i\}_{0\leq
i\leq q-1}$ generate $\SE_L^\dagger$ (resp. $\SR_L$) over
$\varphi_q(\SE^\dagger_L)$ (resp. $\varphi_q(\SR_L)$). In either
case, to prove the independence of $\{u_\CF^i\}_{0\leq i\leq q-1}$,
we only need to use the fact $\psi(b_iu_\CF^j)=\delta_{ij}$ $(i,j\in
\{0,1,\cdots, q-1\})$, where $\delta_{ij}$ is the Kronecker sign.
Finally we note that the second assertion of (\ref{it:dagger})
follows from the first one.
\end{proof}

We apply the above to $(\varphi_q,\Gamma)$-modules.

\begin{prop} If $D$ is a $(\varphi_q,\Gamma)$-module over $R$ where
$R=\SE_L$, $\SE_L^\dagger$ or $\SR_L$, then there is a unique
operator $\psi: D\rightarrow D$ such that \begin{equation}
\label{eq:psi-D} \psi(a\varphi_q(x))=\psi(a)x \text{ and }
\psi(\varphi_q(a)x)=a\psi(x) \end{equation} for any $a\in R$ and
$x\in D$. Moreover $\psi$ commutes with $\Gamma$.
\end{prop}
\begin{proof}  Let $\{e_1, e_2, \cdots, e_d\}$ be a basis of $D$ over
$R$. By the definition of $(\varphi_q,\Gamma)$-modules,
$\{\varphi_q(e_1), \varphi_q(e_2), \cdots, \varphi_q(e_d)\}$ is also
a basis of $D$. For any $m\in D$ writing $ m = a_1\varphi_q(e_1) +
a_2\varphi_q(e_2) + \cdots + a_d\varphi_q(e_d)$, we put
$\psi(m)=\psi(a_1)e_1 + \psi(a_2)e_2 + \cdots + \psi(a_d)e_d$. Then
$\psi$ satisfies (\ref{eq:psi-D}). It is easy to prove the
uniqueness of $\psi$. Observe that for any $\gamma\in\Gamma$,
$\gamma\psi\gamma^{-1}$ also satisfies (\ref{eq:psi-D}). Thus
$\gamma\psi\gamma^{-1}=\psi$ by uniqueness of $\psi$. This means
that $\psi$ commutes with $\Gamma$.
\end{proof}

\subsection{The operator $\partial$ and the map $\Res$}
\label{ss:partial}

Recall that $\partial=\frac{\partial F_\CF}{\partial Y}(u_\CF,0)
\cdot \rd/\rd u_\CF$. So $\rmd t_\CF= \frac{\partial F_\CF}{\partial
Y}(u_\CF,0) \rmd u_\CF$ and $\frac{\rmd t_\CF}{\rmd
u_\CF}=(\frac{\partial F_\CF}{\partial Y}(u_\CF,0))^{-1}$.

\begin{lem}\label{lemme-continuite-partial}
If $r\geq s>0$ and \(f \in \Robba_L^{]0,r]}\), then
\(v^{[s,r]}(\partial f) \geqslant v^{[s,r]}(f) - r\).
\end{lem}
\begin{proof}
Observe that \(v_p(\frac{\partial F_\CF}{\partial Y}(z,0) ) = 0\)
for all \(z\) in the disk \(\left\lvert z\right\rvert < 1\). Thus
\(v^{[s,r]}(\partial f) = v^{[s,r]}\left(\frac{\rd f}{\rd u_\CF}
\right)\). Write \( f = \sum_{n \in \BZ} a_n u_\CF^n \). Then we
have
\begin{align*}
v^{[s,r]}\left(\frac{\rd f}{\rd u_\CF} \right)
    &= \inf_{\substack{ r \geqslant v_p(z) \geqslant s \\
        n \in \BZ}} v_p\left(n a_n z^{n-1}\right) \\
    &\geqslant \inf_{\substack{ r \geqslant v_p(z) \geqslant s \\
        n \in \BZ}} \left(v_p(a_n) + n v_p(z) - v_p(z)\right) \\
    &\geqslant \inf_{\substack{ r \geqslant v_p(z) \geqslant s \\
        n \in \BZ}} \left(v_p(a_n) + n v_p(z)\right) -r \\
    &\geqslant v^{[s,r]}(f) -r,
\end{align*} as desired.
\end{proof}

\begin{lem} \label{lemma-comm-1} We have
$$ \partial\cdot \sigma_a=a\sigma_a \cdot
\partial, \hskip 10pt  \partial \cdot
\varphi_q=\pi  \varphi_q \cdot
\partial , \hskip 10pt \partial \circ \psi=\pi ^{-1} \psi\circ \partial. $$
\end{lem}
\begin{proof}
From the definition of $\nabla$ we see that $\nabla=t_\CF\partial$
commutes with $\Gamma$, $\varphi_q$ and $\psi$. So the equalities
$$\sigma_a(t_\CF)=at_\CF, \ \varphi_q(t_\CF)=\pi t_\CF,
\ \psi(t_\CF)=\psi(\pi^{-1}\varphi_q(t_\CF))=\pi^{-1}t_\CF $$ imply
the lemma.
% the equality $\log_\CF([a](u))=a \log_\CF (u)$ for $a\in \CO_L$ yields the result.
\end{proof}

Let $\res:\SR_L \rd u_\CF\rightarrow L$ be the residue map
$\res(\sum\limits_{i\in \BZ} a_i u_\CF^i \rd u_\CF)=a_{-1}$,
and let $\Res:\SR_L\rightarrow L$ be the map defined by
$\Res(f)=\res(f \rd t_\CF)$.

\begin{prop}\label{prop:res-partial} We have the following exact sequence
$$ \xymatrix{ 0 \ar[r] & L \ar[r] & \SR_L \ar[r]^{\partial}
& \SR_L \ar[r]^{\Res} & L \ar[r] & 0   } $$ where  $L\rightarrow
\SR_L$ is the inclusion map.
\end{prop}
\begin{proof} The kernel of $\partial$ is just the kernel of $\rd/\rd
u_\CF$ and thus is $L$. For any $a\in L$ we have $\Res(
\frac{a}{u_\CF} \cdot (\frac{\rd t_\CF}{\rd u_\CF})^{-1})=a$, which
implies that $\Res$ is surjective. If $f=\partial g$, then $f\rmd
t_\CF=\mathrm{d}g$ and so $\Res(f)=\res(\rd g)=0$. It follows that
$\Res\circ \partial=0$. Conversely, if $f\in \SR_L$ satisfies
$\Res(f)=0$, then $f$ can be written as $f=(\frac{\rmd t_\CF}{\rmd
u_\CF})^{-1} \cdot \sum_{i\neq -1}a_i u_\CF^i$. Put $g=\sum_{i\neq
-1}\frac{a_i}{i+1}u_\CF^{i+1}$. Then $f=\partial g$.
\end{proof}
\

\begin{prop} ~ \label{prop-tate-trace}
\begin{enumerate}
\item \label{it:res-1}
$\Res\circ \sigma_a =a^{-1}\Res $.
\item \label{it:res-2}
$\Res\circ \varphi_q =\frac{q}{\pi }\Res $ and $\Res\circ \psi
=\frac{\pi }{q}\Res $.
\end{enumerate}
\end{prop}
\begin{proof}
First we prove (\ref{it:res-1}). Let $g$ be in $\SR_L$ and put
$f=\partial g$. By Lemma \ref{lemma-comm-1}  we have
$$\sigma_a(f)=\sigma_a\circ\partial (g) = a^{-1}\partial(\sigma_a(g)), \hskip 10pt \psi(f)=\psi\circ\partial (g) = \pi \partial(\psi(g)).$$
Thus by Proposition \ref{prop:res-partial} we have $\Res\circ
\sigma_a=a^{-1}\Res=0$ and $\Res\circ \psi=\frac{\pi }{q}\Res=0$ on
$\partial \SR_L$. From
$$\sigma_a(1/u_\CF)=\frac{1}{[a]_\CF(u_\CF)}\equiv \frac{1}{au_\CF}\mod \SR_L^+,$$ we see
that $\Res\circ
\sigma_a(\frac{1}{u_\CF})=a^{-1}\Res(\frac{1}{u_\CF})$. Assertion
(\ref{it:res-1}) follows.

To prove $ \Res\circ \psi =\frac{\pi }{q}\Res $, without loss of
generality we suppose that $\CF$ is the special Lubin-Tate group. In
this case $\psi(\frac{1}{u_\CF})=\frac{\pi}{q u_\CF}$, and so
$\Res(\psi(1/u_\CF)) =\frac{\pi }{q}\Res(1/u_\CF).$ It follows that
$\Res\circ \psi=\frac{\pi }{q} \Res$. Finally we have $
\Res(\varphi_q(z)) =\frac{q}{\pi }\Res(\psi(\varphi_q(z)))
=\frac{q}{\pi }\Res(z) $ for any $z\in \SR_L$. In other words,
$\Res\circ \varphi_q = \frac{q}{\pi }\Res$.
\end{proof}

Using $\Res$ we can define a pairing $\{\cdot, \cdot\}: \SR_L\times
\SR_L\rightarrow L$ by $\{f, g\}=\Res(fg)$.

\begin{prop} \label{prop:dual} \begin{enumerate}
\item \label{it:pairing-perfect}
The pairing $\{\cdot, \cdot\}$ is perfect and induces a continuous
isomorphism from $\SR_L$ to its dual.
\item \label{it:pairing-property}
We have $$\{\sigma_a (f), \sigma_a(g)\}=a^{-1}\{ f, g\}, \hskip 10pt
\{ \varphi_q(f), \varphi_q(g)\}=\frac{q}{\pi}\{f, g\}, \hskip 10pt
\{\psi (f), \psi(g)\}=\frac{\pi}{q}\{f, g\}. % \hskip 10pt \{\psi (f),g\}=\frac{\pi}{q}\{f, \varphi_q(g)\}.
$$ \end{enumerate}
\end{prop}
\begin{proof} Assertion (\ref{it:pairing-perfect}) follows from \cite[Remark
I.1.5]{Col-lang}. Assertion (\ref{it:pairing-property}) follows from
Proposition \ref{prop-tate-trace}.
\end{proof}

\section{Operators on $\SR_{\BC_p}$}
\label{sec:act}

\subsection{The operator $\psi$ on $\SR_{\BC_p}$} \label{ss:psi-ext}

First we define $\SR_{\BC_p}$. For any $r\geq 0$, let
$\SE^{]0,r]}_{\BC_p}:=\SE^{]0,r]}\widehat{\otimes}_F\BC_p$ be the
topological tensor product, i.e. the Hausdorff completion of the
projective tensor product $\SE^{]0,r]}\otimes_F \BC_p$
(cf.~\cite{Schn}). Then $\SE^{]0,r]}_{\BC_p}$ is the ring of Laurent
series $f=\sum_{i\in\BZ}a_iu_\CF^i$ with coefficients in $\BC_p$
that are convergent on the annulus $0< v_p(u_\CF)\leq r$. We also
write $\SR_{\BC_p}^+$ for $\SE^{]0,+\infty]}_{\BC_p}$. Then we
define $\SR_{\BC_p}$ to be the inductive limit $\lim\limits_{
r\rightarrow 0 } \SE^{]0,r]}_{\BC_p}$.

The $p$-adic Fourier theory of Schneider and Teitelbaum \cite{ST1}
shows that $\SR^+_{\BC_p}$ is isomorphic to the ring $\SD(\CO_F,
\BC_p)$ of $\BC_p$-valued locally $F$-analytic distributions on
$\CO_F$. We recall this below.

By \cite{ST1} there exists a rigid analytic group variety $\FX$ such
that $\FX(L)$, for any extension $L\subseteq \BC_p$ of $F$, is the
set of $L$-valued locally $F$-analytic characters. For $\lambda\in
\SD(\CO_F, L)$, put $F_\lambda(\chi)=\lambda(\chi)$, $\chi\in
\FX(L)$. Then $F_\lambda$ is in $\CO(\FX/L)$, and the map
$\SD(\CO_F, L)\rightarrow \CO(\FX/L)$, $\lambda\mapsto F_\lambda$,
is an isomorphism of $L$-Fr\'echet algebras.

Let $\CF'$ be the $p$-divisible group dual to $\CF$, $T\CF'$ the
Tate module of $\CF'$. Then $T\CF'$ is a free $\CO_F$-module of rank
$1$; the Galois action on $T\CF'$ is given by the continuous
character $\tau := \chi_{\cyc}\cdot\chi_\CF^{-1}$, where
$\chi_{\cyc}$ is the cyclotomic character. By Cartier duality, we
obtain a Galois equivariant % and $\CO_F$-invariant
pairing $ \langle\ , \ \rangle: \CF(\BC_p)\otimes_{\CO_F} T\CF
'\rightarrow \RB_1(\BC_p)$, where $\RB_1(\BC_p)$ is the
multiplicative group $\{z\in \BC_p : |z-1|<1\}$. Fixing a generator
$t'$ of $T\CF'$, we obtain a map $\CF(\BC_p)\rightarrow
\RB_1(\BC_p)$. As a formal series, this morphism can be written as
\(\beta_\CF(X):=\exp(\Omega \log_\CF(X))\) for some \(\Omega \in
\C_p\), and it lies in \(1 + X \mathscr{O}_{\C_p}[[X]]\). Moreover,
we have \(v_p(\Omega) = \frac{1}{p-1} - \frac{1}{(q-1) e_F}\)
(cf.~the appendix of~\cite{ST1} or
\cite{colmez-varietes-abeliennes}) and \(\sigma(\Omega) =
\tau(\sigma) \Omega\) for all \(\sigma \in G_F\). Using
$\langle\cdot,\cdot\rangle$ we obtain an isomorphism of rigid
analytic group varieties
$$
\kappa: \CF(\BC_p)\xrightarrow{\sim} \FX(\BC_p), \hskip 10pt
z\mapsto \kappa_z(i) := \langle t', [i]_\CF(z) \rangle =
\beta_\CF([i]_\CF(z)). $$ Passing to global sections, we obtain the
desired isomorphism $\SD(\CO_F,\BC_p)\cong \CO(\FX/\BC_p) \cong
\SR^+_{\BC_p}$.

We extend $\varphi_q$, $\psi$ and the $\Gamma$-action $\BC_p$-linear
and continuously to $\SR_{\BC_p}$. By continuity we have
$\psi(\varphi_q(f)g)=f\psi(g)$ for any $f,g\in \SR_{\BC_p}$. All of
these actions keep $\SR^+_{\BC_p}$ invariant.

\begin{lem}\label{lem:beta} We have
\begin{eqnarray*}
\sigma_a(\beta_\CF([i]_\CF)) &=& \beta_\CF([ai]_\CF), \\
\varphi_q(\beta_\CF([i]_\CF)) &=& \beta_\CF([\pi i]_\CF), \\
\psi(\beta_\CF([i]_\CF)) &=& \left\{\begin{array}{ll} 0 & \text{ if
}i\notin \pi\CO_F \\
\beta_\CF([i/\pi]_\CF) & \text{ if } i\in
\pi\CO_F,\end{array}\right.\\
\partial (\beta_\CF([i]_\CF)) &=& i \Omega  \beta_\CF([i]_\CF) .
\end{eqnarray*}
\end{lem}
\begin{proof} The formulae for $\sigma_a$ and $\varphi_q$ are
obvious. The formula for $\partial$ follows from that
$$ \partial \exp(i\Omega \log_\CF(u_\CF)) = \exp(i\Omega \log_\CF(u_\CF))  \cdot \partial (i\Omega   t_\CF)= i\Omega \exp(i\Omega \log_\CF(u_\CF)). $$

If $i\in \pi\CO_F$, then $\psi(\beta_\CF([i]_\CF)
)=\psi\circ\varphi_q(\beta_\CF([i/\pi]_\CF))=\beta_\CF([i/\pi]_\CF)$.
For any $i\notin\pi\CO_F$, we have
\begin{equation} \label{eq:psi0}\begin{aligned}
\psi (\beta_\CF([i]_\CF)  ) &= \frac{1}{q} \varphi_q^{-1}\left(%
        \sum_{\eta \in \ker [\pi]_\CF}
            \beta_\CF([i]_\CF(u_\CF +_\CF \eta))
    \right)  = \frac{1}{q} \varphi_q^{-1}\left(%
        \beta_\CF([i]_\CF)
        \sum_{\eta \in \ker [\pi]_\CF}
            \beta_\CF([i]_\CF(\eta))
        \right) = 0
\end{aligned}\end{equation} because \(\{\beta_\CF([i]_\CF(\eta)): \eta\in
\ker[\pi]_\CF\}=\{\beta_\CF(\eta): \eta\in \ker[\pi]_\CF \}\) take
values in the set of \(p\)-th roots of unity and each of these
\(p\)-th roots of unity appears \(q/p\) times.
\end{proof}

The isomorphism $\SR^+_{\BC_p}\cong \SD(\CO_F,\BC_p)$ transfers the
actions of $\varphi_q$, $\psi$ and $\Gamma$ to $\SD(\CO_F,\BC_p)$.

\begin{lem} For any $\mu\in \SD(\CO_F,\BC_p)$, we have
$$\sigma_a(\mu)(f)=\mu(f(a\cdot)), \hskip 10pt \varphi_q(\mu)(f)=\mu(f(\pi\cdot)).$$
\end{lem}
\begin{proof} Note that the action of $\varphi_q$ and $\Gamma$ on $\SR^+_{\BC_p}$ comes, by passing to global
sections, from the $(\varphi_q,\Gamma)$-action on $\CF(\BC_p)$ with
$\varphi_q=[\pi]_\CF$ and $\sigma_a=[a]_\CF$. The isomorphism
$\kappa$ transfers the action to $\FX(\BC_p)$:
$\varphi_q(\chi)(x)=\chi(\pi x)$ and $\sigma_a(\chi)(x)=\chi(ax)$.
Passing to global sections yields what we want.
\end{proof}

\begin{lem}
\label{lemme-exp-Omega-log-base} The family
\(\Big(\beta_\CF([i]_\CF)
    \Big)_{\overline{i} \in \CO_F/\pi}\)
is a basis of \(\Robba_{\C_p}\) over \(\varphi_q(\Robba_{\C_p})\).
Moreover, if
\[
f  = \sum_{\overline{i} \in \CO_F/\pi}
    \beta_\CF([i]_\CF)\varphi_q(f_i ),
\]
then the terms of the sum do not depend on the choice of the
liftings~\(i\), and we have
\[
f_i = \psi\Big(\beta_\CF([-i]_\CF) f \Big).
\]
\end{lem}
\begin{proof}
What we need to show is that
\begin{equation}\label{eq:decomp}
f= \sum_{\bar{i}\in \CO_F/\pi} \beta_\CF([i]_\CF) \cdot
\varphi_q\circ \psi (\beta_\CF([-i]_\CF)f) \end{equation} for all
$f\in \SR_{\BC_p}$. Indeed, (\ref{eq:decomp}) implies that
$\{\beta_\CF([i]_\CF)\}_{\bar{i}\in \CO_F/\pi}$ generate
$\SR_{\BC_p}$ over $\varphi_q(\SR_{\BC_p})$. On the other hand, if
$f=\sum\limits_{\bar{i}\in \CO_F/\pi} \beta_\CF([i]_\CF) \varphi_q
(f_i)$, using (\ref{eq:psi0}) we obtain
$f_i=\psi(\beta_\CF([-i]_\CF)f)$, which implies the linear
independence of $\{\beta_\CF([i]_\CF)\}_{\bar{i}\in \CO_F/\pi}$ over
$\varphi_q(\SR_{\BC_p})$. As the map $f\mapsto \sum_{\bar{i}\in
\CO_F/\pi} \beta_\CF([i]_\CF) \cdot \varphi_q\circ \psi
(\beta_\CF([-i]_\CF)f)$ is $\varphi_q(\SR_{\BC_p})$-linear and
continuous, we only need to prove (\ref{eq:decomp}) for a subset
that topologically generates $\SR_{\BC_p}$ over
$\varphi_q(\SR_{\BC_p})$. For example, $\{u^i_{\CF}\}_{0\leq i\leq
q-1}$ is such a subset. So it is sufficient to prove
(\ref{eq:decomp}) for $f\in \SR^+_{\BC_p}$. For any $i\in \CO_F$,
let $\delta_i$ be the Dirac distribution such that
$\delta_i(f)=f(i)$. Then $\kappa^*(\delta_i)=\beta_\CF([i]_\CF)$.
Indeed, we have
$$\kappa^*(\delta_i)(z)=\delta_i(z)=\kappa_z(i)=\beta_\CF([i]_\CF(z)).$$ It is
easy to see that $(\delta_i)_{\bar{i}\in \CO_F/\pi}$ is a basis of
$\SD(\CO_F,\BC_p)$ over $\varphi_q(\SD(\CO_F,\BC_p))$. Thus every
$f\in \SR^+_{\BC_p}$ can be written uniquely in the form
$f=\sum_{\bar{i}\in \CO_F/\pi} \beta_\CF([i]_\CF)\varphi_q(f_i)$
with $f_i\in \SR^+_{\BC_p}$. As is observed above, from
(\ref{eq:psi0}) we deduce that $f_i=\psi(\beta_\CF([-i]_\CF) f)$.
\end{proof}

Next we define operators $\Res_U$. These are analogues of the
operators defined in \cite{Col-lang}.

For any \(f \in \Robba_{\C_p}\), \(i \in \CO_F\) and integer \(m
\geqslant 0\), put
\[
\Res_{i + \pi^m \CO_F}(f) =
    \beta_\CF([i]_\CF)
    (\varphi_q^m \circ \psi^m)\Big(%
        \beta_\CF([-i]_\CF) f %
    \Big).
\]
Lemma \ref{lemme-exp-Omega-log-base} says that
\[
f = \sum_{\overline{i} \in \CO_F/\pi}
    \Res_{i + \pi \CO_F}(f),
\]
This implies that the operators \(\Res_{i + \pi^m \CO_F}\) are well
defined (i.e. \(\Res_{i + \pi^m \CO_F}\) does not depend on the
choice of \(i\) in the ball \(i + \pi^m \CO_F\)). Applying Lemma
\ref{lemme-exp-Omega-log-base} recursively we get
\[
f = \sum_{\overline{i} \in \CO_F/\pi^m}
    \Res_{i + \pi^m \CO_F}(f).
\]
Finally, if \(U\) is a compact open subset of \(\CO_F\), it is a
finite disjoint union of balls \(i_k + \pi^{m_k} \CO_F\). Define
\(\Res_U = \sum_k \Res_{i_k + \pi^{m_k} \CO_F}\). The map
\(\Res_U\colon \Robba_{\BC_p} \rightarrow \Robba_{\BC_p}\) does not
depend on the choice of these balls, and we have \(\Res_{\CO_F} =
1\), \(\Res_\emptyset = 0\) and \(\Res_{U \cup U'} + \Res_{U \cap
U'} = \Res_U + \Res_{U'}\).

\subsection{The operator $m_\alpha$} \label{ss:m}

Let \(\alpha\colon \CO_F \rightarrow \C_p\) be a locally
($F$-)analytic function. In this subsection, we define an operator
\(m_\alpha\colon \Robba_{\C_p} \rightarrow \Robba_{\C_p}\) similar
to the one defined in \cite[V.2]{colmez-mirabolique}.

Since \(\alpha\) is a locally analytic function on \(\CO_F\), there
is an integer \(m \geqslant 0\) such that
\[
\alpha(x) = \sum_{n = 0}^{+\infty} a_{i,n} (x-i)^n \qquad\text{for
all \(x \in i + \pi^m \CO_F\),}
\]
with \(a_{i,n} = \frac{1}{n!} \left.\frac{\rd^n}{\rd x^n} \alpha(x)
\right|_{x=i}\). Let \(\ell \geqslant m\) be an integer. Define
\[
m_\alpha(f) =
    \sum_{\overline{i} \in \CO_F / \pi^\ell}
    \beta_\CF([i]_\CF)
    \left(\varphi_q^\ell \circ
        \left(\sum_{n=0}^{+\infty} a_{i,n} \pi^{\ell n} \Omega^{-n} \partial^n\right)
        \circ \psi^\ell\right)\Big(
    \beta_\CF([-i]_\CF) \cdot f \Big).
\] (Formally, this definition can be seen as ``\(m_\alpha =
\alpha(\Omega^{-1} \partial)\)''). According to
Lemmas~\ref{lemme-continuite-phi-gamma}, \ref{lemme-continuite-psi}
and \ref{lemme-continuite-partial}, if \(r<\frac{1}{q^{\ell-1} (q-1)
e_F }\) then we have
\[
v^{[s,r]}\left((\varphi_q^\ell \circ \Omega^{-n}\partial^n \circ
\psi^\ell)(g)\right)
    \geqslant  - n q^\ell r - n v_p(\Omega)
        + v^{[s,r]}(g) - \ell v_p(q),
\]
and thus \(\sum_{n=0}^{+\infty} a_{n,i} \pi^{\ell n}
        (\varphi_q^\ell \circ \Omega^{-n}\partial^n \circ \psi^\ell)(g)\)
converges when $\ell$ and $r$ satisfy
\[
\frac{\ell}{e_F} - q^\ell r - \frac{1}{p-1} + \frac{1}{(q-1) e_F}
    \geqslant \frac{m}{e_F}.
\]
If we choose \(\ell > m + \frac{e_F}{p-1} - \frac{1}{q-1}\) and
\(r\) close enough to~\(0\), then this condition is satisfied.
Hence, we have indeed defined a continuous operator \(m_\alpha\colon
\Robba_{\C_p} \rightarrow \Robba_{\C_p}\).

Now, let us prove that \(m_\alpha(f)\) neither depend on the choice
of~\(\ell\), nor on that of the liftings~\(i\) for \(\overline{i}
\in \CO_F / \pi^\ell\). By linearity and continuity, we may assume
that $f=1_{i+\pi^m\CO_F}(x-i)^k$. Remark that we have
\[
a_{i+\pi^m v,n} =
    \binc{k}{n} \pi^{(k-n)m} v^{k-n}.
\]
It suffices to show that,
\[\begin{aligned}
\sum_{\overline{v} \in \CO_F / \pi^{\ell-m}}
    \beta_\CF([\pi^m v]_\CF)
    \left(\varphi_q^{\ell} \circ \left(%
        \sum_{n=0}^{k} a_{i+\pi^m v,n} \pi^{\ell n} \Omega^{-n} \partial^n
    \right) \circ \psi^{\ell}\right)\Big(%
        \beta_\CF([-\pi^m v]_\CF)\cdot f\Big)  \\
    = \left(\varphi_q^m \circ \left(%
       \pi^{mk} \Omega^{-k}
        \partial^k
        \right) \circ \psi^m\right) f .
\end{aligned}\] and for this it is sufficient to prove that
\[
\sum_{\overline{v} \in \CO_F / \pi^{\ell-m}}
    \beta_\CF([v]_\CF)
    \left(\varphi_q^{\ell-m} \circ \left(%
        \sum_{n=0}^{k} a_{i+\pi^m v,n} \pi^{\ell n} \Omega^{-n} \partial^n
    \right) \circ \psi^{\ell-m}\right)\Big(%
        \beta_\CF([-v]_\CF) \cdot f\Big)
    = \pi^{mk} \Omega^{-k} \partial^k f.
\]
 As
\[
\sum_{n=0}^{k} a_{i+\pi^m v,n} \pi^{\ell n} \Omega^{-n}
\partial^n = \sum_{n=0}^k \binc{k}{n} \pi^{(k-n)m}v^{k-n}\cdot \pi^{\ell n} \Omega^{-n}
\partial^n=
     \pi^{m k}
        \left(\pi^{\ell-m} \Omega^{-1} \partial + v\right)^k,
\]
it suffices to prove that
\[
\Omega^{-k}\partial^k f = \sum_{\overline{v} \in \CO_F /
\pi^{\ell-m}}
    \beta_\CF([v]_\CF)
    \left(\varphi_q^{\ell-m} \circ \left(%
        \pi^{\ell-m} \Omega^{-1} \partial + v
    \right)^k \circ \psi^{\ell-m}\right)\Big(%
        \beta_\CF([-v]_\CF) f\Big).
\]
Since \(\left(\pi^{\ell-m} \Omega^{-1} \partial + v\right)^k \circ
\psi^{\ell-m}
        = \psi^{\ell-m} \circ \left(\Omega^{-1} \partial + v\right)^k\)
and
\[
\left(\Omega^{-1} \partial + v\right)\Big(%
        \beta_\CF([-v]_\CF) f\Big)
    = \beta_\CF([-v]_\CF)
        \Omega^{-1} \partial f
\] (which follows from Lemma \ref{lem:beta}),
the problem reduces to proving
\[
f = \sum_{\overline{v} \in \CO_F / \pi^{\ell-m}}
    \beta_\CF([v]_\CF)
    \left(\varphi^{\ell-m} \circ \psi^{\ell-m}\right)\Big(%
        \beta_\CF([-v]_\CF) f\Big).
\]
But this can be deduced from Lemma \ref{lem:beta} and Lemma
\ref{lemme-exp-Omega-log-base}.

\begin{lem}
\label{lemme-composition-m-alpha} If \(\alpha, \beta\colon \CO_F
\rightarrow \C_p\) are locally analytic functions, then \(m_\alpha
\circ m_\beta = m_{\alpha\beta}\).
\end{lem}
\begin{proof}
We can choose \(\ell\) sufficiently large, so that the same value
can be used to define \(m_\alpha(f)\) and \(m_\beta(f)\). Since
\(\psi^\ell \circ \varphi_q^\ell = 1\), the equality in the lemma
reduces to the expression of the product of two power series.
\end{proof}

\begin{lem}
\label{lemme-prop-elem-m-alpha} We have:
\begin{itemize}
\item \(m_1 = \operatorname{id}\)
\item If \(U\) is a compact open subset of \(\CO_F\), then
\(\Res_U = m_{\mathbf{1}_U}\).
\item If \(\lambda \in \C_p\), then
\(m_{\lambda \alpha} = \lambda m_\alpha\).
\item \(\varphi_q \circ m_\alpha =
    m_{x \mapsto \mathbf{1}_{\pi \CO_F}(x) \alpha(\pi^{-1} x)}
    \circ \varphi_q\)
\item \(\psi \circ m_\alpha = m_{x \mapsto \alpha(\pi x)} \circ \psi\)
\item For any $a\in \CO_F^\times$, we have
\(\sigma_a \circ m_\alpha =
    m_{x \mapsto \alpha(a^{-1} x)}
    \circ \sigma_a\)
\item \(\Robba^+_{\BC_p}\) is stable under \(m_\alpha\).
\end{itemize}
\end{lem}
\begin{proof}
These are easy consequences of the definition of \(m_\alpha\).
\end{proof}

\begin{rem} The notation $m_\alpha$ stands for ``multiply by
$\alpha$'': for any $\mu\in \SD(\CO_F,\BC_p)$ we have $m_\alpha
\kappa^*(F_\mu) = \kappa^*( F_{\alpha\mu} )$, where $\alpha\mu$ is
the distribution such that $(\alpha \mu )(f)= \mu(\alpha f)$ for any
locally $F$-analytic function $f$.
\end{rem}

The operator \(m_\alpha\) has been defined over \(\Robba_{\C_p}\),
using a period \(\Omega \in \C_p\) that is transcendental
over~\(F\). However, in some cases, it is possible to construct
related operators over \(\Robba_L\), for \(L\) smaller
than~\(\C_p\). This is done using the following lemma.

\begin{lem}
\label{lemme-action-coeffs-m-alpha} Let $\sigma$ be in $G_L$.
Consider the action of \(\sigma\) over~\(\Robba_{\C_p}\) given by
\[
f^\sigma(u_\CF) = \sum_{n \in \Z} \sigma(a_n) u_\CF^n
\qquad\text{if}\qquad f(u_\CF) = \sum_{n \in \Z} a_n u_\CF^n \in
\Robba_{\C_p}.
\]
Then, we have \(m_\alpha(f)^\sigma = m_\beta(f^\sigma)\), for
\(\beta(x) = \sigma\left(\alpha\left(%
    \frac{\chi_\CF(\sigma)}{\chi_\Gmul(\sigma)} x\right)\right)\).
\end{lem}

\begin{proof}
This can be deduced easily from the definition of \(m_\alpha\) and
the action of \(\sigma\) on~\(\Omega\).
\end{proof}

\subsection{The $L[\Gamma]$-module $\SR_L(\delta)^{\psi=0}$}
\label{ss:psi0}

Let \(\delta\colon F^\times \rightarrow L^\times\) be a locally
\(F\)-analytic character. Then the map \(x \mapsto
\mathbf{1}_{\CO_F^\times}(x) \delta(x)\) is locally analytic on
\(\CO_F\). Thus, we have an operator
\(m_{\mathbf{1}_{\CO_F^\times}\delta}\) on \(\Robba_{\C_p}\).

\begin{lem}
Let \(f \) be in \( \Robba_L\). If
\(m_{\mathbf{1}_{\CO_F^\times}\delta}(f)
    = \sum_{n \in \Z} a_n u_\CF^n \in \Robba_{\C_p}\), then the coefficients
\(a_n\) are all on the same line of the \(L\)-vector space~\(\C_p\).
Moreover, this line does not depend on~\(f\).
\end{lem}

\begin{proof}
Let~\(\sigma \) be in \(G_L\). From
Lemma~\ref{lemme-action-coeffs-m-alpha} and
Lemma~\ref{lemme-prop-elem-m-alpha} we see that
\[
m_{\mathbf{1}_{\CO_F^\times}\delta}(f)^\sigma =
    \delta\left(\frac{\chi_\CF(\sigma)}{\chi_\Gmul(\sigma)}\right)
    m_{\mathbf{1}_{\CO_F^\times}\delta}(f),
\]
and thus \(\sigma(a_n) =
    \delta\left(\frac{\chi_\CF(\sigma)}{\chi_\Gmul(\sigma)}\right)
a_n\)
for all~\(n\).

Ax-Sen-Tate's theorem (see e.g. \cite{Ax-thm-AST} or
\cite{Le-Borgne-AST}) says that \(\C_p^{G_L} = L\). Hence,
\[
\left\{\, z \in \C_p : \;
    \sigma(z) = \delta\left(\frac{\chi_\CF(\sigma)}{\chi_\Gmul(\sigma)}\right)
    z
    \;\forall\; \sigma \in G_L  \right\}
\]
is an \(L\)-vector subspace of \(\C_p\) with dimension \(0\)
or~\(1\), which proves the lemma.
\end{proof}

Since \(m_{\mathbf{1}_{\CO_F^\times}\delta} \circ
    m_{\mathbf{1}_{\CO_F^\times}\delta^{-1}}
    = \Res_{\CO_F^\times} = 1 - \varphi_q \circ \psi\) is not
null, there is a unique \(L\)-line in \(\C_p\) (which depends only
on~\(\delta\)) in which all the coefficients of the series
\(m_{\mathbf{1}_{\CO_F^\times}\delta}(f)\), for \(f \in \Robba_L\),
lie. Choose some non-zero \(a_\delta\) on this line.

As
\[
\varphi_q \circ \psi \circ
    m_{\mathbf{1}_{\CO_F^\times}\delta} =
m_{\mathbf{1}_{\pi \CO_F}
    \mathbf{1}_{\CO_F^\times}\delta} = 0
\]
and \(\varphi_q\) is injective,
\(m_{\mathbf{1}_{\CO_F^\times}\delta}(f)\) is in
\(\Robba_{\C_p}^{\psi=0}\).

\begin{lem}
\label{lemme-def-S-kappa} Define:
\[
M_\delta\colon\begin{array}[t]{rll}
\Robba_L^{\psi=0}&\longrightarrow&\Robba_L^{\psi=0}, \\
f&\longmapsto&%
a_\delta^{-1} m_{\mathbf{1}_{\CO_F^\times}\delta}(f).
\end{array}
\]
(These maps are defined up to homothety, with ratio in~\(L\),
because of the choice of constants~\(a_\delta\)). Then:
\begin{itemize}
\item \(M_1\) is a homothety (with ratio in~\(L^\times\))
of \(\Robba_L^{\psi=0}\);
\item \(M_{\delta_1} \circ M_{\delta_2} = M_{\delta_1 \delta_2}\),
up to homothety;
\item \(M_\delta\) is a bijection, and its inverse is \(M_{\delta^{-1}}\)
up to homothety;
\item for all \(\gamma \in \Gamma\), we have
\(\delta(\gamma) \gamma \circ M_\delta = M_\delta \circ \gamma\);
\item \((\Robba_L^+)^{\psi=0}\) is stable under~\(M_\delta\).
\end{itemize}
\end{lem}
\begin{proof}
This follows from Lemma \ref{lemme-prop-elem-m-alpha} and the fact
that \(\Im (\Res_{\CO_F^\times}) = \Ker (\Res_{\pi\CO_F})
    = \Robba_{\C_p}^{\psi=0}\).
\end{proof}

If $\delta$ is in $\SI_\an(L)$, we put
$\SR_L^-(\delta)=\SR_L(\delta)/\SR^+_L(\delta)$. Since
$\SR^+_L(\delta)$ is $\varphi_q,\psi, \Gamma$-stable,
$\SR_L^-(\delta)$ also has $\varphi_q$, $\psi$, $\Gamma$-actions.

\begin{lem} We have an exact sequence
$$ \xymatrix{ 0 \ar[r] & \SR^+_L(\delta)^{\psi=0} \ar[r] & \SR_L(\delta)^{\psi=0} \ar[r] & \SR^-_L(\delta)^{\psi=0} \ar[r] & 0. } $$
\end{lem}
\begin{proof} This follows from the snake lemma and the surjectivity
of the map $\psi:\SR^+_L(\delta)\rightarrow \SR^+_L(\delta)$.
\end{proof}

Observe that $\SR_L(\delta)^{\psi=0}=\SR_L^{\psi=0}\cdot e_\delta$
and $\SR^+_L(\delta)^{\psi=0}=(\SR^+_L)^{\psi=0}\cdot e_\delta$. As
$\psi$ commutes with $\Gamma$, $\SR_L(\delta)^{\psi=0}$,
$\SR^+_L(\delta)^{\psi=0}$ and $\SR^-_L(\delta)^{\psi=0}$ are all
$\Gamma$-invariant.

\begin{prop}\label{prop:key} Let $\delta_1$ and $\delta_2$ be two locally $F$-analytic
characters $F^\times\rightarrow L^\times$. Then as
\(L[\Gamma]\)-modules,
 \(\Robba_L(\delta_1)^{\psi=0}\) is isomorphic to
\(\Robba_L(\delta_2)^{\psi=0}\), $\SR^+_L(\delta_1)^{\psi=0}$ is
isomorphic to $\SR^+_L(\delta_2)^{\psi=0}$, and
$\SR_L^-(\delta_1)^{\psi=0}$ is isomorphic to
$\SR_L^-(\delta_2)^{\psi=0}$. \end{prop}
\begin{proof} All of the isomorphisms in question are induced by \(M_{\delta_1^{-1}
\delta_2}\).
\end{proof}

\begin{prop} \label{prop:partial-iso-0}
The map $\partial$ induces $\Gamma$-equivariant isomorphisms
$(\SR_L(\delta))^{\psi=0} \rightarrow (\SR_L(x\delta))^{\psi=0}$,
$(\SR_L^+(\delta))^{\psi=0} \rightarrow (\SR_L^+(x\delta))^{\psi=0}$
and $(\SR_L^-(\delta))^{\psi=0} \rightarrow
(\SR_L^-(x\delta))^{\psi=0}$.
\end{prop}
\begin{proof} We first show that the maps in question are bijective. For
this we only need to consider the case of $\delta=1$. As
$\Ker(\partial)=L$, $\partial$ is injective on $\SR_L^{\psi=0}$. For
any $z\in \SR_L^{\psi=0}$, $\Res(z)=\frac{q}{\pi}\Res(\psi(z))=0$.
Thus by Proposition \ref{prop:res-partial} there exists $z'\in
\SR_L$ such that $\partial z' =z$. As
$\partial(\psi(z'))=\frac{1}{\pi}\psi(\partial z' )=0$, $\psi(z)=c$
for some $c\in L$. Then $z'-c\in \SR_L^{\psi=0}$ and
$\partial(z'-c)=z$. This shows that the map $\SR_L^{\psi=0}
\rightarrow \SR_L ^{\psi=0}$ is bijective. It is clear that, for any
$z\in \SR_L^{\psi=0}$, $\partial z\in \SR^+_L$ if and only if $z\in
\SR_L^+$. Thus the restriction $\partial:
(\SR^+_L)^{\psi=0}\rightarrow (\SR^+_L)^{\psi=0}$ and the induced
map $\partial: (\SR^-_L)^{\psi=0}\rightarrow (\SR^-_L)^{\psi=0}$ are
also bijective.

That these isomorphisms are $\Gamma$-equivariant follows from Lemma
\ref{lemma-comm-1}.
\end{proof}

Put \begin{equation} \label{eq:S-delta} S_\delta :=
\SR_L^-(\delta)^{\Gamma=1, \psi=0 }.\end{equation} As before, let
$\nabla_\delta$ be the operator on $\SR^+_L$ or $\SR_L$ such that
$(\nabla_\delta a)e_\delta=\nabla (ae_\delta)$, i.e.
$\nabla_\delta=t_\CF\partial +w_\delta$. The set
$\SR^+_L(\delta)/\nabla_\delta\SR^+_L(\delta)$ also admits actions
of $\Gamma$, $\varphi_q$ and $\psi$.  Put
$$
T_\delta:=(\SR^+_L(\delta)/\nabla_\delta\SR^+_L(\delta))^{\Gamma=1,\psi=0}.$$
Both $S_\delta$ and $T_\delta$ are $L$-vector spaces and only depend
on $\delta|_{\CO_F ^\times}$.

\begin{lem}\label{lem:eq} $S_\delta= \SR_L^-(\delta)^{\psi=0, \nabla_\delta=0, \ \Gamma=1}$,
i.e. $S_\delta$ coincides with the set of $\Gamma$-invariant
solutions of $\nabla_\delta z=0$ in $\SR^-_L(\delta)^{\psi=0}$.
\end{lem}
\begin{proof} In fact, if $z\in \SR^-_L(\delta)^{
\Gamma=1}$, then $\nabla_\delta z=0$.
\end{proof}

\begin{cor} \label{prop:dim-1}
$\dim_L S_\delta = \dim_L S_{1}$ and $\dim_L T_\delta=\dim_L T_1$
for all $\delta\in\SI_\an(L)$.
\end{cor}
\begin{proof} This follows directly from Proposition \ref{prop:key}.
\end{proof}

\begin{cor}\label{prop:tran-S} The map $z\mapsto \partial^nz $ induces
isomorphisms $S_{\delta}\rightarrow S_{x^n \delta}$ and
$T_\delta\rightarrow T_{x^n \delta}$.
\end{cor}
\begin{proof} This follows directly from
Proposition \ref{prop:partial-iso-0}.
\end{proof}

We determine $\dim_LS_\delta$ and $\dim_L T_\delta$ below.

\begin{lem} \label{prop:inj}
The map $\nabla_\delta$ induces an injection $\bar{\nabla}_\delta:
S_\delta\rightarrow T_\delta$.
\end{lem}
\begin{proof}
By Proposition \ref{prop:key} we only need to consider the case of
$\delta=1$.

Let $z$ be an element of $S_1$. Let $\tilde{z}\in \SR_L^{\psi=0}$ be
a lifting of $z$. By Lemma \ref{lem:eq}, $\nabla \tilde{z} $ is in
$\SR^+_L $. We show that the image of $\nabla \tilde{z} $ in
$\SR^+_L /\nabla \SR^+_L $ belongs to $T_1$. Since
$\psi(\tilde{z})=0$, $\psi(\nabla  \tilde{z} )=\nabla
(\psi(\tilde{z}))=0$. For any $\gamma\in\Gamma$ there exists
$a_\gamma\in \SR^+_L$ such that $\gamma \tilde{z}
=\tilde{z}+a_\gamma$. Thus $\gamma(\nabla \tilde{z})=\nabla
\tilde{z}+\nabla  a_\gamma$. Hence the image of $\tilde{z}$ in
$\SR^+_L /\nabla \SR^+_L (\delta)$ is fixed by $\Gamma$. Furthermore
the image only depends on $z$. Indeed, if $\tilde{z}'\in
\SR_L^{\psi=0}$ is another lifting of $z$, then $\nabla
(\tilde{z}'-\tilde{z})$ is in $\nabla  \SR^+_L$. Therefore we obtain
a map $\bar{\nabla}: S_1\rightarrow T_1$.

We prove that $\bar{\nabla}$ is injective. Suppose that $z\in S_1$
satisfies $\bar{\nabla} z=0$. Let $\tilde{z}\in\SR_L^{\psi=0}$ be a
lifting of $z$. Since $\nabla\tilde{z}$ is in $\nabla \SR^+_L$,
there exists $y\in \SR^+_L$ such that $\nabla y=\nabla \tilde{z}$.
Thus $\nabla (\tilde{z}-y)=0$. Then $\tilde{z}-y$ is in $L$, which
implies that $\tilde{z}\in \SR^+_L$ or equivalently $z=0$.
\end{proof}

\begin{lem}\label{lem:dim} $\dim_L T_1=1$.
\end{lem}
\begin{proof}
Note that $T_1=(\SR^+_L/\SR^+_Lt_\CF)^{\Gamma=1,\psi=0}.$ As
$\SR^+_L$ is a Fr\'echet-Stein algebra, from the decomposition
(\ref{eq:t-decom}) of the ideal $(t_\CF)$ we obtain an isomorphism
\begin{equation} \label{eq:inc-iso}
\jmath: \SR^+_L/\SR^+_Lt_\CF \xrightarrow{\sim}
\SR^+_L/([\pi]_\CF(u_\CF))\times\prod_{n\geq 1} \SR^+_L/
(\varphi_q^{n}(Q)).
\end{equation}

The operator $\psi$ induces maps $\psi_0:
\SR^+_L/([\pi]_\CF(u_\CF))\rightarrow \SR^+_L/\SR^+_L u_\CF$ and
$\psi_n: \SR^+_L/ (\varphi_q^{n}(Q)) \rightarrow \SR^+_L/
(\varphi_q^{n-1}(Q))$, $n\geq 1$. Thus
$\jmath((\SR^+_L/\SR^+_Lt_\CF)^{\Gamma=1,\psi=0})$ is exactly the
subset
$$\{ (y_n)_{n\geq 0}: y_0\in (\SR^+_L/([\pi]_\CF(u_\CF)))^\Gamma, \psi_0(y_0)=0,
\ y_n\in (\SR^+_L/ (\varphi_q^{n}(Q)))^{\Gamma}, \psi_n(y_n)=0 \
\forall n\geq 1 \}  $$   of
$\SR^+_L/([\pi]_\CF(u_\CF))\times\prod_{n\geq 1} \SR^+_L/
(\varphi_q^{n}(Q))$.

If $n\geq 1$, then $\SR^+_F /\varphi_q^{n}(Q)$ is a finite extension
of $F$ and the action of $\Gamma$ factors through the whole Galois
group of this extension. Thus $(\SR^+_F /(\varphi_q^{n}(Q)))^\Gamma=
F $ and $(\SR^+_L/(\varphi_q^n(Q)))^\Gamma=L$. Since $\psi_n(a)=a$
for any $a\in L$, $(\SR^+_F /(\varphi_q^{n}(Q)))^\Gamma\cap
\ker(\psi_n)=0$ for any $n\geq 1$. Similarly
$(\SR^+_L/([\pi]_\CF(u_\CF)))^\Gamma=(\SR^+_L/(u_\CF))^\Gamma\times
(\SR^+_L/(Q))^\Gamma$ is $2$-dimensional over $L$. As $\psi_0(1)=1$
and the image of $\psi_0$, i.e. $\SR^+_L/\SR^+_L u_\CF$, is
$1$-dimensional over $L$, the kernel of
$\psi_0|_{(\SR^+_L/([\pi]_\CF(u_\CF)))^\Gamma}$ is of dimension $1$.
It follows that $T_1=(\SR^+_L/\SR^+_Lt_\CF)^{\Gamma=1,\psi=0}$ is of
dimension $1$.
\end{proof}

\begin{cor} \label{cor:dim} $\dim_LS_1=1$.
\end{cor}
\begin{proof} The map $\nabla$ injects $S_1$ into $T_1$ with image of
dimensional $1$.
\end{proof}

\begin{rem} \label{rem:constant-non-zero}
If $z\in T_1$ is non-zero, then any lifting $\tilde{z}\in \SR^+_L$
of $z$ is not in $u_\CF\SR^+_L$ or equivalently
$\tilde{z}|_{u_\CF=0}\neq 0$. We only need to verify this for the
special Lubin-Tate group. In this case,
$\SR^+_L/([\pi]_\CF(u_\CF))=\oplus_{i=0}^{q-1}Lu_\CF^i$. We have
$(\SR^+_L/([\pi]_\CF(u_\CF)))^\Gamma=L\oplus L u_\CF^{q-1}$. Indeed,
an element of $\SR^+_L/([\pi]_\CF(u_\CF))$ is fixed by $\Gamma$ if
and only if it is fixed by the operators $z\mapsto \sigma_\xi(z)$
with $\xi\in \mu_{q-1}$; but $\sigma_\xi(u_\CF)=[\xi]_\CF(u_\CF)=\xi
u_\CF$ and so $\sigma_\xi(u_\CF^i)=\xi^iu_\CF^i$ for any $i\in \BN$.
Then $\big(\SR^+_L/([\pi]_\CF(u_\CF))\big)^{\Gamma=1,\psi=0} = L
\cdot (u_\CF^{q-1}-(1-q)\pi/q )$.
\end{rem}

\begin{prop} \label{thm:dim}
For any $\delta\in \SI_\an(L)$, $\dim_LS_\delta=\dim_L T_\delta= 1$
and the map $\bar{\nabla}_\delta$ is an isomorphism.
\end{prop}
\begin{proof} This follows from Corollary \ref{prop:dim-1},
Lemma \ref{prop:inj}, Lemma \ref{lem:dim} and Corollary
\ref{cor:dim}.
\end{proof}

\section{Cohomology theories for $(\varphi_q,\Gamma)$-modules}
\label{sec:coh}

For a $(\varphi_q,\Gamma)$-module $D$ over $\SR_L$, the
$(\varphi_q,\Gamma)$-module structure induces an action of the
semi-group $G^+:=\varphi_q^{\BN}\times \Gamma$ on $D$. Following
\cite{Col-an} we define $H^\bullet(D)$ as the cohomology of the
semi-group $G^+$. Let $C^\bullet(G^+, D)$ be the complex $$
\xymatrix{ 0 \ar[r] & C^0(G^+,D) \ar[r]^{d_1} & C^1(G^+,D)
\ar[r]^{d_2} & \cdots, }
$$ where $C^0(G^+,D)=D$,  $C^n(G^+,D)$ for $n\geq 1$ is the set
of continuous functions from $(G^+)^n$ to $D$, and $d_{n+1}$ is the
differential
$$ d_{n+1} c(g_0,\cdots, g_n)= g_0\cdot c(g_1,\cdots, g_n) +\sum_{i=1}^{n-1} (-1)^{i+1} c(g_0, \cdots, g_ig_{i+1},\cdots g_n) + (-1)^{n+1} c(g_0,\cdots g_{n-1}).
$$ Then $H^i(D)=H^i(C^\bullet(G^+, D))$.

If $D_1$ and $D_2$ are two $(\varphi_q,\Gamma)$-modules over
$\SR_L$, we use $\Ext(D_1,D_2)$ to denote the set, in fact an
$L$-vector space, of extensions of $D_1$ by $D_2$ in the category of
$(\varphi_q,\Gamma)$-modules over $\SR_L$.

We construct a natural map $\Theta^D: \Ext(\SR_L, D)\rightarrow
H^1(D)$ for any $(\varphi_q,\Gamma)$-module $D$. Let $\widetilde{D}$
be an extension of $\SR_L$ by $D$. Let $e\in\widetilde{D}$ be a
lifting of $1\in \SR_L$. Then $g\mapsto g(e)-e$, $g\in G^+$, is a
$1$-cocycle, and induces an element of $H^1(D)$ independent of the
choice of $e$. Thus we obtain the desired map
$$\Theta^D: \Ext(\SR_L, D)\rightarrow H^1(D).$$

\begin{prop} For any $(\varphi_q,\Gamma)$-module $D$ over $\SR_L$, $\Theta^D$ is an isomorphism.
\end{prop}
\begin{proof}
Let $\widetilde{D}$ be an extension of $\SR_L$ by $D$ in the
category of $(\varphi_q,\Gamma)$-modules whose image under
$\Theta^D$ is zero. Let $e\in \widetilde{D}$ be a lifting of
$1\in\SR_L$. As the image of $g\mapsto g(e)-e$, $g\in G^+$, in
$H^1(D)$ is zero, there exists some $d\in D$ such that
$(g-1)e=(g-1)d$ for all $g\in G^+$. Then $g(e-d)=e-d$ for all $g\in
G^+$. Thus $\widetilde{D}=D\oplus \SR_L(e-d)$ as a
$(\varphi_q,\Gamma)$-module. This proves the injectivity of
$\Theta^D$. Next we prove the surjectivity of $\Theta^D$. Given a
$1$-cocycle $g\mapsto c(g)\in D$, correspondingly we can extend the
$(\varphi_q,\Gamma)$-module structure on $D$ to the $\SR_L$-module
$\widetilde{D}=D\oplus \SR_L e$ such that $\varphi_q(e)=e+
c(\varphi_q)$ and $\gamma(e)=e+c(\gamma)$ for $\gamma\in \Gamma$.
\end{proof}

If $D_1$ and $D_2$ are two $\CO_F $-analytic
$(\varphi_q,\Gamma)$-modules over $\SR_L$, we use
$\Ext_\an(D_1,D_2)$ to denote the  $L$-vector space of extensions of
$D_1$ by $D_2$ in the category of $\CO_F $-analytic
$(\varphi_q,\Gamma)$-modules over $\SR_L$. We will introduce another
cohomology theory $H^*_\an(-)$, wherein for any $\CO_F $-analytic
$(\varphi_q,\Gamma)$-module $D$ the first cohomology group
$H_\an^1(D)$ coincides with $\Ext_\an(\SR_L,D)$.

If $D$ is $\CO_F$-analytic, we consider the following complex
$$C^\bullet_{\varphi_q,\nabla}(D): \hskip 10pt
\xymatrix{ 0\ar[r] & D \ar[r]^{f_1} & D \oplus D \ar[r]^{f_2} & D
\ar[r] & 0 } ,
$$ where $f_1: D\rightarrow D \oplus D$ is the map $ m \mapsto ((\varphi_q-1)m,
\nabla m)$ and $f_2: D \oplus D\rightarrow D$ is $(m, n)\mapsto
\nabla m-(\varphi_q-1)n$. As $f_1$ and $f_2$ are
$\Gamma$-equivariant, $\Gamma$ acts on the cohomology groups
$H^i_{\varphi_q,\nabla}(D):=H^i(C^\bullet_{\varphi_q,\nabla}(D))$,
$i=0,1,2$. Put $ H_\an^i(D):=H^i_{\varphi_q,\nabla} ( D )^{\Gamma}.$

By a simple calculation we obtain
$$H^0(D)=H_\an^0(D)=D^{\varphi_q=1,\Gamma=1}. $$ Note that $D^{\varphi_q=1}$
is finite dimensional over $L$, and so is $H^0(D)$. If $D$ is
\'etale and if $V$ is the $L$-linear Galois representation of $G_F $
attached to $D$, then $$H^0(D)=H^0_\an(D)=H^0(G_F ,V)=V^{G_F }.$$

For our convenience we introduce some notations. Put
$Z^1_{\varphi_q,\nabla}(D):=\ker (f_2)$ and $B^1(D):=\im(f_1)$. For
any $(m_1, n_1)$ and $(m_2, n_2)$ in $Z^1_{\varphi_q,\nabla}(D)$, we
write $(m_1, n_1)\sim (m_2, n_2)$ if $(m_1-m_2, n_1-n_2)\in B^1(D)$.
Put
$$Z^1(D):=\{(m,n)\in Z^1_{\varphi_q,\nabla}(D): (m,n)\sim \gamma(m,n) \text{ for any }\gamma\in \Gamma
\}.$$  Then $H_\an^1(D)=Z^1(D)/B^1(D)$.

Let $\widetilde{D}$ be an $\CO_F $-analytic extension of $\SR_L$ by
$D$. Let $e\in\widetilde{D}$ be a lifting of $1\in \SR_L$. Then
$((\varphi_q-1)e, \nabla_{\widetilde{D}}e)$ belongs to $Z^1(D)$ and
induces an element of $H_\an^1(D)$ independent of the choice of $e$.
In this way we obtain a map
$$\Theta_{\an}^D: \Ext_\an(\SR_L, D)\rightarrow H_\an^1(D).$$

\begin{thm}\label{thm:an-iso} (= Theorem \ref{thm:intro-coh})
For any $\CO_F$-analytic $(\varphi_q,\Gamma)$-module $D$ over
$\SR_L$, $\Theta_\an^D$ is an isomorphism.
\end{thm} The proof below is due to the referee and much
simpler than the proof in our original version.
\begin{proof} First we show that $\Theta_\an^D$ is injective. Let $\widetilde{D}$ be an $\CO_F$-analytic extension of $\SR_L$ by $D$
whose image under $\Theta_\an^D$ is zero. Let $e\in \widetilde{D}$
be a lifting of $1\in \SR_L$. As the image of $((\varphi_q-1)e,
\nabla_{\widetilde{D}}e)$ in $H^1_{\varphi_q,\nabla}(D)$ is zero,
there exists some $d\in D$ such that $(\varphi_q-1)e=(\varphi_q-1)d$
and $\nabla_{\widetilde{D}}e=\nabla_{\widetilde{D}}d$. Then $e-d$ is
in $\widetilde{D}^{\varphi_q=1,\nabla=0}$. The $\Gamma$-action on
$\widetilde{D}^{\varphi_q=1,\nabla=0}$ is locally constant and thus
is semisimple. So $1\in \SR_L$ has a lifting $e'\in
\widetilde{D}^{\varphi_q=1,\nabla=0}$ fixed by $\Gamma$. This proves
the injectivity of $\Theta_\an^D$.

Next we prove the surjectivity of $\Theta_\an^D$.

Let $z$ be in $H^1_\an(D)$ and let $(x,y)$ represent $z$, so that
$\nabla x=(\varphi_q-1)y$. The invariance of $z$ by $\Gamma$ ensures
the existence of $y_\sigma\in D$ for each $\sigma\in \Gamma$ such
that $(\sigma-1)(x,y)=((\varphi_q-1)y_\sigma, \nabla y_\sigma)$. As
$y_\sigma$ is unique up to an element of $D^{\varphi_q=1,\nabla=0}$,
the $2$-cocycle $y_{\sigma, \tau}=y_{\sigma\tau}-\sigma y_\tau-
y_\sigma$ takes values in $D^{\varphi_q=1,\nabla=0}$. If $z=0$, then
there exists $a\in D$ such that $x=(\varphi_q-1)a$ and $y=\nabla a$.
We have $\nabla(y_\sigma- (\sigma-1)a)=0$. In other words, we can
write $y_\sigma=(\sigma-1)a+a_\sigma$ with $a_\sigma\in
D^{\varphi_q=1,\nabla=0}$. Then $y_{\sigma,\tau}=
a_{\sigma\tau}-\sigma a_\tau- a_\sigma$ and thus
$y_{\bullet,\bullet}$ is a coboundary. So we obtain a map
$H^1_\an(D)\rightarrow H^2(\Gamma, D^{\varphi_q=1, \nabla=0})$.

We will show that the image of $z$ by this map is zero. Fix a basis
$\{e_1,\cdots, e_d\}$ of $D$ over $\SR_L$. Let $r>0$ be sufficiently
small such that the matrices of $\varphi_q$ and $\sigma \in \Gamma$
relative to $\{e_i\}_{i=1}^d$ are all in $\GL_d(\SE_L^{]0,r]})$. Put
$D^{]0,r]}=\oplus_{i=1}^d \SE^{]0,r]}_L e_i$; if $s\in (0,r]$ put
$D^{[s,r]}=\oplus_{i=1}^d \SE^{[s,r]}_L e_i$. Then $D^{]0,r]}$ and
$D^{[s,r]}$ are stable by $\Gamma$. As the matrix of $\varphi_q$ is
invertible in $\RM_d(\SE_L^{]0,r]})$, $\{\varphi_q(e_i)\}_{i=1}^d$
is also a basis of $D^{]0,r]}$. Shrinking $r$ if necessarily we may
assume that $\varphi_q$ maps $D^{[s,r]}$ to $D^{[s/q,r/q]}$; we may
also suppose that $x$ and $y$ are in $D^{]0,r]}$, and that $t_\CF\in
\SE_L^{]0,r]}$. By the relation $\nabla=t_\CF
\partial$ on $\SE^{[s,r]}_L$, Lemma \ref{lemme-continuite-partial} and
the fact that $\nabla$ is a differential operator i.e. satisfies a
relation similar to (\ref{eq:diff}), we can show that the action of
$\Gamma$ induces a bounded infinitesimal action $\nabla$ on the
Banach space $D^{[s,r]}$. We leave this to the reader. Let us denote
$\ell(\sigma)=\log (\chi_\CF(\sigma))$. For $\sigma$ close enough to
$1$ (depending on $D$ and $s, r$) the series of operators
$$ E(\sigma) =\ell(\sigma)+ \frac{\ell(\sigma)^2}{2}\nabla +\frac{\ell(\sigma)^3}{3!}\nabla^2+\cdots
$$ converges on $D^{[s,r]}$ and
also on $D^{[s/q,r/q]}$. Note that, for $\sigma$ close enough to $1$
we have $\sigma=\exp(\ell(\sigma)\nabla)$ on $D^{[s/q,r/q]}$. Let
$\Gamma'$ be an open subgroup of $\Gamma$ such that for $\sigma\in
\Gamma'$ the above two facts hold. Then for $\sigma\in\Gamma'$ we
have
\begin{equation}\label{eq:frob}
(\varphi_q-1) (E(\sigma)y) = E(\sigma)(\varphi_q-1)y
=E(\sigma)\nabla x=\nabla E(\sigma)x=(\sigma-1)x.
\end{equation}  Note that $\varphi_q(E(\sigma)y)$ is
in $D^{[s/q,r/q]}$. So by (\ref{eq:frob}) we have $E(\sigma)y\in
D^{[s/q,r/q]}\cap D^{[s,r]}=D^{[s/q,r]}$ if $s$ is chosen such that
$s<r/q$. Doing this repeatedly we will obtain $E(\sigma)y\in
D^{]0,r]}$. Taking $y_\sigma=E(\sigma)y$ for $\sigma\in \Gamma'$ we
will have $y_{\sigma,\tau}=0$ for $\sigma,\tau\in \Gamma'$. In other
words, the restriction to $\Gamma'$ of the image of $z$ in
$H^2(\Gamma, D^{\varphi_q=1,\nabla=0})$ is $0$. Since
$\Gamma/\Gamma'$ is finite and $D^{\varphi_q=1,\nabla=0}$ is a
$\BQ$-vector space, the image of $z$ is itself $0$. So we can modify
$y_\sigma$ by an element of $D^{\varphi_q=1,\nabla=0}$ so that
$y_{\sigma,\tau}$ is identically $0$. But this means that
$(\sigma-1)y_\tau=(\tau-1)y_\sigma$, so the $1$-cocycle
$\varphi_q\mapsto x$, $\sigma\mapsto y_\sigma$ defines an element of
$H^1(D)$ hence also an extension of $\SR_L$ by $D$.

We will show that the resulting extension in fact belongs to
$\Ext^1_\an(\SR_L, D)$. As $\Gamma$ is locally constant on
$D^{\varphi_q=1,\nabla=0}$, shrinking $\Gamma'$ if necessary we may
assume that $\Gamma'$ acts trivially on $D^{\varphi_q=1,\nabla=0}$.
Then $\sigma\mapsto y_\sigma-E(\sigma)y $ is a continuous
homomorphism from $\Gamma'$ to $D^{\varphi_q=1,\nabla=0}$. Note that
any homomorphism from $\Gamma'$ to $D^{\varphi_q=1,\nabla=0}$ can be
extended to $\Gamma$. Thus $y_\sigma-E(\sigma)y=\lambda(\sigma)$ for
some $\lambda\in \Hom(\Gamma, D^{\varphi_q=1,\nabla=0})$ and all
$\sigma\in\Gamma'$. If $S$ is a set of representatives of
$\Gamma/\Gamma'$ in $\Gamma$, the map
$T_S=\frac{1}{|\Gamma:\Gamma'|}\sum_{\sigma\in S}\sigma$ is the
identity on $H^1_\an(D)$ and a projection from
$D^{\varphi_q=1,\nabla=0}$ to $H^0(D)$; moreover it commutes with
$\varphi_q$, $\nabla$ and $\Gamma$. This means that we can apply
$T_S$ to $(x,y)$ and $y_\sigma$; then we have
$y_\sigma-E(\sigma)y=\lambda(\sigma)$ for some $\lambda\in
\Hom(\Gamma, H^0(D))$ and all $\sigma\in \Gamma'$. As $\sigma\mapsto
E(\sigma)y$ is analytic, the extension in question is
$\CO_F$-analytic.
\end{proof}

As above, let $\Hom(\Gamma, H^0(D))$ be the set of homomorphisms of
groups from $\Gamma$ to $H^0(D)$. A homomorphism $h:
\Gamma\rightarrow H^0(D)$ is said to be {\it locally analytic} if
$h(\exp(a\beta))=a h(\exp(\beta))$ for all $a\in \CO_F$ and
$\beta\in \Lie\Gamma$. Let $\Hom_\an(\Gamma, H^0(D))$ be the subset
of $\Hom(\Gamma, H^0(D))$ consisting of locally analytic
homomorphisms.

Note that we have natural injections $$ \Hom_\an(\Gamma,
H^0(D))\rightarrow \Ext^1_\an(\SR_L, D) \hskip 10pt  \text{ and }
\hskip 10pt \Hom(\Gamma, H^0(D))\rightarrow \Ext^1(\SR_L, D).$$

\begin{thm} \label{thm:exact-sq}
Assume that $D$ is an $\CO_F$-analytic $(\varphi_q,\Gamma)$-module over $\SR_L$. Then we have an exact sequence
$$ \xymatrix{ 0\ar[r] & \Hom_\an(\Gamma, H^0(D)) \ar[r] & \Hom(\Gamma, H^0(D)) \oplus \Ext^1_\an(\SR_L,D) \ar[r] & \Ext^1(\SR_L, D) \ar[r] & 0. }
$$
\end{thm}

For the proof we introduce an auxiliary cohomology theory. Let
$\gamma$ be an element of $\Gamma$ of infinite order, i.e.
$\log(\chi_\CF(\gamma))\neq 0$. We consider the complex
$$C^\bullet_{\varphi_q,\gamma}(D):  \hskip 10pt
\xymatrix{ 0\ar[r] & D \ar[r]^{g_1} & D \oplus D \ar[r]^{g_2} & D
\ar[r] & 0 },$$ where $g_1: D\rightarrow D \oplus D$ is the map $ m
\mapsto ((\varphi_q-1)m, (\gamma-1)m)$ and $g_2: D \oplus
D\rightarrow D$ is $(m,n)\mapsto (\gamma-1)m-(\varphi_q-1)n$. As
$g_1$ and $g_2$ are $\Gamma$-equivariant, $\Gamma$ acts on
$H^i_{\varphi_q,\gamma}(D):=H^i(C^\bullet_{\varphi_q,\gamma}(D))$,
$i=0,1,2$. Put $
H^i_{\an,\gamma}(D):=H^i_{\varphi_q,\gamma}(D)^{\Gamma}.$ A simple
calculation shows that $H^0_{\an,\gamma}(D)=H^0_{\an}(D)$.

For any $\gamma\in\Gamma$ we use $\overline{\langle \gamma \rangle}$
to denote the closed subgroup of $\Gamma$ topologically generated by
$\gamma$. If $\gamma$ is of infinite order and if $D$ is an
$\SR_L$-module together with a semilinear
$\overline{\langle\gamma\rangle}$-action, let $\nabla_\gamma$ be the
operator on $D$ that can be written as
$\lim\limits_{\overrightarrow{\;\;\gamma'\;}}\frac{\log(\gamma')}{\log(\chi_\CF(\gamma'))}$
formally, where $\gamma'$ runs through all elements of
$\overline{\langle \gamma \rangle}$ with $\log \chi_\CF(\gamma')\neq
0$. (For a precise definition we only need to imitate the definition
of $\nabla$.)

Let $\widetilde{D}$ be an $\CO_F $-analytic extension of $\SR_L$ by
$D$. Let $e\in\widetilde{D}$ be a lifting of $1\in \SR_L$. Then
$((\varphi_q-1)e, (\gamma-1)e)$ induces an element of
$H^1_{\an,\gamma}(D)$ independent of the choice of $e$. This yields
a map $\Theta_{\an,\gamma}^D: \Ext_\an(\SR_L,D)\rightarrow
H^1_{\an,\gamma}(D).$ Given an element of $H^1_{\an,\gamma}(D)$, we
can attach to it an extension $\widetilde{D}$ of $\SR_L$ by $D$ in
the category of free $\SR_L$-modules of finite rank together with
semilinear actions of $\varphi_q$ and
$\overline{\langle\gamma\rangle}$. Let $e\in\widetilde{D}$ be a
lifting of $1\in \SR_L$. Then $\big((\varphi_q-1)e, \nabla_\gamma e
\big)$ belongs to $Z^1(D)$ and induces an element of $H_{\an}^1(D)$
independent of the choice of $e$. This gives a map $
\Upsilon_{\an,\gamma}^D: H^1_{\an,\gamma}(D)\rightarrow
H_{\an}^1(D).$ Observe that
$\Upsilon_{\an,\gamma}^D\circ\Theta_{\an,\gamma}^D = \Theta_\an^D$.
By an argument similar to the proof of the injectivity of
$\Theta_\an^D$, we can show that both $\Theta_{\an,\gamma}^D$ and
$\Upsilon_{\an,\gamma}^D$ are injective. Hence it follows from
Theorem \ref{thm:an-iso} that $\Theta_{\an,\gamma}^D$ and
$\Upsilon_{\an,\gamma}^D$ are isomorphisms.

If $c$ is a $1$-cocycle representing an element $z$ of $H^1(D)$,
then $(c(\varphi_q), c(\gamma))$ induces an element in
$H^1_{\an,\gamma}(D)$ which only depends on $z$. This yields a map
$\Upsilon_\gamma^D: H^1(D)\rightarrow H^1_{\an, \gamma}(D)$. Hence,
$\Theta_{\an,\gamma}^D: \Ext_\an(\SR_L, D)\rightarrow
H^1_{\an,\gamma}(D)$ extends to a map $\Ext(\SR_L,D)\rightarrow
H^1_{\an,\gamma}(D)$, which will also be denoted by
$\Theta_{\an,\gamma}^D$. We have the following commutative diagram
\begin{equation}\label{eq:comm-HH} \xymatrix{ \Ext(\SR_L,D)
\ar[r]^{\Theta^D}_{\sim} \ar[rd]^{\Theta^D_{\an,\gamma}} &
H^1(D)\ar[d]^{\Upsilon_\gamma^D}
\\ \Ext_\an(\SR_L, D) \ar[r]^{\sim}_{\Theta^D_{\an,\gamma}} \ar@{^(->}[u] &
H^1_{\an,\gamma}(D). }\end{equation} The composition
$(\Theta^D_{\an,\gamma^{-1}})^{-1}\circ\Upsilon_\gamma^D\circ
\Theta^D$ is a projection from $\Ext(\SR_L,D)$ to
$\Ext_\an(\SR_L,D)$, which depends on $\gamma$.

\vskip 10pt

\noindent{\it Proof of Theorem \ref{thm:exact-sq}.} The only
nontrivial thing to be proved is the surjectivity of $\Hom(\Gamma,
H^0(D)) \oplus \Ext^1_\an(\SR_L,D) \rightarrow \Ext^1(\SR_L, D)$.
Let $\widetilde{D}$ be in $\Ext^1(\SR_L, D)$. Without loss of
generality we may assume that the image of $\widetilde{D}$ by the
projection
$(\Theta^D_{\an,\gamma^{-1}})^{-1}\circ\Upsilon_\gamma^D\circ
\Theta^D$ is zero. Let $e\in \widetilde{D}$ be a lifting of $1\in
\SR_L$. Then let $c$ be the $1$-cocycle defined by $\varphi_q\mapsto
(\varphi_q-1)e$, $\sigma\mapsto (\sigma-1)e$ for $\sigma\in \Gamma$,
so that $\bar{c}$, the class of $c$ in $H^1(D)$, corresponds to
$\widetilde{D}$. So the image of $\bar{c}$ by the map
$\Upsilon^D_\gamma$ is zero. This means that there exists $d\in D$
such that $(\varphi_q-1)d=c(\varphi_q)$ and $(\gamma-1)d=c(\gamma)$.
Replacing $e$ by $e-d$, we may assume that
$c(\varphi_q)=c(\gamma)=0$. Then for any $\sigma\in \Gamma$, we have
$(\varphi_q-1)c(\sigma)=(\sigma-1)c(\varphi_q)=0$ and
$(\gamma-1)c(\sigma)=(\sigma-1)c(\gamma)=0$. This means that
$c(\sigma)\in D^{\varphi_q=1,\gamma=1}$.  Note that $M:=
D^{\varphi_q=1,\gamma=1}$ is of finite rank over $L$. We write
$M=H^0(D)\oplus\oplus_j M_j$ as a $\Gamma$-module, where each of
$M_j$ is an irreducible $\Gamma$-module. Write $c=c' + \sum_j c_j$
by this decomposition. Observe that $c'$ and $c_j$ are all
$1$-cocycles. As $M_j$ is irreducible and the $\Gamma$-action on
$M_j$ is nontrivial, there exists some $\gamma_j\in \Gamma$ such
that $\gamma_j-1$ is invertible on $M_j$. Then there exists $m_j\in
M_j$ such that $c_j(\gamma_j)=(\gamma_j-1)m_j$. A simple calculation
shows that $c_j(\sigma)=(\sigma-1)m_j$ for all $\sigma\in \Gamma$.
Replacing $e$ by $e-\sum_j m_j$, we may assume that $c=c'$. Then
$c(\varphi_q)=0$ and $c|_{\Gamma}$ is a homomorphism from $\Gamma$
to $H^0(D)$. \qed

\begin{cor} \label{thm:comp}  (=Theorem \ref{thm:ext}) $\Ext_\an(\SR_L, D)$ is of codimension $([F:\BQ_p]-1)\dim_L
H^0(D)$ in $\Ext(\SR_L, D)$. In particular, if $H^0(D)=0$, then
$\Ext_\an(\SR_L,D) = \Ext(\SR_L, D)$; in other words, all extensions
of $\SR_L$ by $D$ are $\CO_F$-analytic.
\end{cor}
\begin{proof} This follows from Theorem \ref{thm:exact-sq} and
the equalities $\dim_L \Hom(\Gamma, H^0(D))=[F:\BQ_p]\dim_L H^0(D)$
and $\dim_L \Hom_\an(\Gamma, H^0(D))=\dim_L H^0(D)$.
\end{proof}

\section{Computation of $H_\an^1(\delta)$ and $H^1(\delta)$} \label{sec:comp}

In the case of $F=\BQ_p$, Colmez \cite{tri} computed $H^1$ for not
necessarily \'etale $(\varphi,\Gamma)$-modules of rank $1$ over the
Robba ring. In this case, Liu \cite{Liu} computed $H^2$ for this
kind of $(\varphi,\Gamma)$-modules, and used it and Colmez's result
to build analogues, for not necessarily \'etale
$(\varphi,\Gamma)$-modules over the Robba ring, of the
Euler-Poincar\'e characteristic formula and Tate local duality.
Later, Chenevier \cite{Chenevier} obtained the Euler-Poincar\'e
characteristic formula for families of trianguline
$(\varphi,\Gamma)$-modules and some related results.

In this section we compute $H_\an^1(\delta)=H_\an^1(\SR_L(\delta))$
(for $\delta\in \SI_\an(L)$) and $H^1(\delta)=H^1(\SR_L(\delta))$
(for $\delta\in \SI(L)$) following Colmez's approach. In Sections
\ref{ss:H0} and \ref{ss:H1} we assume that $\delta$ is in $\SI(L)$,
and in Sections \ref{ss:H1an}, \ref{ss:H1an-par} and \ref{ss:iota-k}
we assume that $\delta$ is in $\SI_\an(L)$.

\subsection{Preliminary lemmas} \label{ss:pre}

\begin{lem} \label{lem:no1} \label{lem:no5}
\begin{enumerate}
\item \label{it:no1-a}
If $\alpha\in L^\times$ is not of the form $\pi^{-i}$, $i\in
\BN$, then $\alpha\varphi_q-1: \SR^+_L\rightarrow \SR^+_L$ is an
isomorphism.
\item \label{it:no1-b}
If $\alpha=\pi^{-i}$ with $i\in \BN$, then the kernel of
$\alpha\varphi_q-1: \SR^+_L\rightarrow \SR^+_L$ is $L\cdot t_\CF^i$,
and $a\in \SR^+_L$ is in the image of $\alpha\varphi_q-1$ if and
only if $\partial^i a|_{u_\CF=0}=0$. Further, $\alpha \varphi_q-1$
is bijective on the subset $\{a\in \SR^+_L : \partial^i
a|_{u_\CF=0}=0\}$.
\end{enumerate}
\end{lem}
\begin{proof} The argument is similar to the proof of \cite[Lemma
A.1]{tri}.  If $k> -v_\pi(\alpha)$, then $-\sum_{n=0}^{+\infty}
(\alpha \varphi_q)^n  $ is the continuous inverse of
$\alpha\varphi_q-1$ on $u_\CF^k\SR^+_L$. The assertions follows from
the fact that $ \SR^+_L = \oplus_{i=0}^{k-1} L \cdot t_\CF^i \oplus
 u_\CF^k\SR^+_L $ and the formula $\varphi_q(t_\CF^i)=\pi^it_\CF^i$. We
 just need to remark that $\partial^i a|_{u_\CF=0}=0$ if and only if $a$ is in
$\oplus_{j=0}^{i-1} L t_\CF^j \oplus u_\CF^{i+1}\SR_L^+$.
\end{proof}

\begin{lem} \label{lem:no2} If $\alpha\in L$ satisfies $v_\pi(\alpha)<1-v_\pi(q)$,
then for any $b\in \SE^\dagger_L$ there exists $c\in \SE^\dagger_L$
such that $b'=b-(\alpha\varphi_q-1)c$ is in
$(\SE^\dagger_L)^{\psi=0}$.
\end{lem}
\begin{proof} By Proposition \ref{prop:psi-con} (\ref{it:psi-bon}), $c=\sum_{k=1}^{+\infty}\alpha^{-k}\psi^k(b)$
is convergent in $\SE^\dagger_L$. It is easy to check that $\alpha
c- \psi(c)=\psi(b)$, which proves the lemma.
\end{proof}

\begin{cor} \label{cor:no1} If $\alpha\in L$ satisfies $v_\pi(\alpha)<1-v_\pi(q)$, then for any $b\in \SR_L$ there
exists $c\in \SR_L$ such that $b'=b-(\alpha \varphi_q-1)c$ is in
$(\SE_L^\dagger)^{\psi=0}$.
\end{cor}
\begin{proof} Let $k$ be an integer  $> -v_\pi(\alpha)$. By Lemma
\ref{lem:no1}, there exists $c_1\in \SR_L$ such that
$b-(\alpha\varphi_q-1)c_1$ is of the form $\sum_{i<k}a_iu_\CF^i$ and
thus is in $\SE_L^\dagger$. Then we apply Lemma \ref{lem:no2}.
\end{proof}

\begin{lem} \label{lem:no3}
If $\alpha\in L$ satisfies $v_\pi(\alpha)<1-v_\pi(q)$, and if $z\in
\SR_L$ satisfies $\psi(z)-\alpha z\in \SR_L^+$, then $z\in \SR_L^+$.
\end{lem}
\begin{proof} Write $z$ in the form $\sum_{k\in
\BZ}a_k  u_\CF^k$ and put $y=\sum_{k\leq -1}a_k  u_\CF^k\in
\SE_L^\dagger$. If $y\neq 0$, multiplying $z$ by a scalar in $L$ we
may suppose that $\inf_{k\leq -1}v_p(a_k)=0$. Then
$$ y-\alpha^{-1}\psi(y) =\alpha^{-1} ( \alpha z -\psi(z) )
+ \sum_{k\geq 0} a_k (\alpha^{-1}\psi(u_\CF^k)-u_\CF^k)$$ belongs to
$ \CO_{\SE_L^\dagger}\cap \SR_L^+ = \CO_L[[u_\CF]]$.  But this is a
contradiction since $ y-\alpha^{-1}\psi(y) \equiv y \mod \pi$. Hence
$y=0$.
\end{proof}

\begin{cor} \label{cor:no2}
If $\alpha\in L$ satisfies $v_\pi(\alpha)<1-v_\pi(q)$, and if $z \in
\SR_L$ satisfies $(\alpha \varphi_q-1)z\in \SR_L^{\psi=0}$, then $z$
is in $\SR_L^+$.
\end{cor}
\begin{proof} We have $\psi(z)-\alpha z = \psi ( z - \alpha \varphi_q (z)
)=0$. Then we apply Lemma \ref{lem:no3}.
\end{proof}

\subsection{Computation of $H^0(\delta)$} \label{ss:H0}

Recall that, if $\delta\in \SI_\an(L)$, then
$H^0_\an(\delta)=H^0(\delta)$.

\begin{prop}\label{prop:H0} Let $\delta$ be in $ \SI(L)$.
\begin{enumerate}
\item If $\delta$ is not of the form $x^{-i}$ with $i\in\BN$, then
$H^0(\delta)=0$.
\item If $i\in\BN$, then $H^0(x^{-i})=Lt_\CF^i$.
\end{enumerate}
\end{prop}
\begin{proof} Observe that $\SR^-_L(\delta)^{\varphi_q=1}=(\SR^-_L)^{\delta(\pi)\varphi_q=1}\cdot e_\delta=0$,
where $\SR_L^-(\delta)=\SR_L(\delta)/\SR^+_L(\delta)$. Thus
$\SR_L(\delta)^{\varphi_q=1,\Gamma=1}=\SR^+_L(\delta)^{\varphi_q=1,\Gamma=1}$.
If $\delta(\pi)$ is not of the form $\pi^{-i}$ with $i\in\BN$, by
Lemma \ref{lem:no1} (\ref{it:no1-a}) we have
$\SR^+_L(\delta)^{\varphi_q=1}=0$ and so
$\SR^+_L(\delta)^{\varphi_q=1,\Gamma=1}=0$. If
$\delta(\pi)=\pi^{-i}$, then
$$\SR^+_L(\delta)^{\varphi_q=1,\Gamma=1}=(Lt_\CF^i \cdot
e_\delta)^{\Gamma=1}=\left\{ \begin{array}{ll} Lt_\CF^i \cdot
e_\delta & \text{ if }\delta=x^{-i}, \\ 0 & \text{ otherwise,}
\end{array} \right.$$ as desired.
\end{proof}

\begin{cor}\label{cor:H0} If $\delta_1$ and $\delta_2$ are two different
characters in $\SI(L)$, then $\SR_L(\delta_1)$ is not isomorphic to
$\SR_L(\delta_2)$.
\end{cor}
\begin{proof} We only need to show that
$\SR_L(\delta_1\delta_2^{-1})$ is not isomorphic to $\SR_L$. By
Proposition \ref{prop:H0}, $\SR_L(\delta_1\delta_2^{-1})$ is not
generated by $H^0(\delta_1\delta_2^{-1})$, but $\SR_L$ is generated
by $H^0(1)$. Thus $\SR_L(\delta_1\delta_2^{-1})$ is not isomorphic
to $\SR_L$.
\end{proof}

\subsection{Computation of $H_{\an}^1(\delta)$ for $\delta\in
\SI_\an(L)$ with $v_\pi(\delta(\pi))<1-v_\pi(q)$} \label{ss:H1an}

Until the end of Section \ref{sec:comp} we will write
$\SR_L(\delta)$ by $\SR_L$ with the twisted
$(\varphi_q,\Gamma)$-action given by
$$ \varphi_{q;\delta}(x) =\delta(\pi)\varphi_q(x), \hskip 10pt \sigma_{a;\delta}(x)=\delta(a)\sigma_a(x).
$$ Recall that $\nabla_\delta=t_\CF \partial +w_\delta$. Write $\delta(\sigma_a)=\delta(a)$.

\begin{lem}\label{lem:no4}
Suppose that $\delta\in\SI_\an(L)$ satisfies
$v_\pi(\delta(\pi))<1-v_\pi(q)$. For any $(a,b)\in
Z^1_{\varphi_q,\nabla}(\delta)$, there exists $(m, n)\in
Z^1_{\varphi_q,\nabla}(\delta)$ with $m\in (\SE_L^\dagger)^{\psi=0}$
and $n \in \SR_L^+$ such that $(a,b)\sim (m,n)$.
\end{lem}
\begin{proof} As $v_\pi(\delta(\pi))<1-v_\pi(q)$, by Corollary
\ref{cor:no1} there exists $c\in \SR_L$ such that
$m=a-(\delta(\pi)\varphi_q-1)c$ is in $(\SE_L^\dagger)^{\psi=0}$.
Put $n=b-\nabla_\delta c $. Then $(m,n)$ is in
$Z^1_{\varphi_q,\nabla}(\delta)$ and $(m,n)\sim (a,b)$. As $
(\delta(\pi)\varphi_q-1)n = \nabla_\delta m = t_\CF \partial m +
w_\delta m $ is in $\SR_L^{\psi=0}$, by Corollary \ref{cor:no2}, $n$
is in $\SR_L^+$.
\end{proof}

\begin{lem} \label{lem:pre}
Suppose that $v_\pi(\delta(\pi))<1-v_\pi(q)$ and $\delta$ is not of
the form $x^{-i}$. Let $(m, n)$ be in $
Z_{\varphi_q,\nabla}^1(\delta)$ with $m\in
(\SE^{\dagger}_L)^{\psi=0}$ and $n\in \SR_L^+$. Then $(m,n)$ is in
$B^1(\delta)$ if and only if

$\bullet$ $m\in(\SE_L^+)^{\psi=0}$ when $\delta(\pi)$ is not of the
form $\pi^{-i}$, $i\in \BN$;

$\bullet$ $m\in(\SE_L^+)^{\psi=0}$ and $\partial^i m|_{u_\CF=0}=0$
when $\delta(\pi)=\pi^{-i}$ and $w_\delta\neq -i$ for some
$i\in\BN$.

$\bullet$ $m\in(\SE_L^+)^{\psi=0}$ and $\partial^i
m|_{u_\CF=0}=\partial^i n|_{u_\CF=0}=0$ when $\delta(\pi)=\pi^{-i}$
and $w_\delta=-i$ for some $i\in\BN$.
\end{lem}

\begin{proof}
We only prove the assertion for the case that $\delta(\pi)=\pi^{-i}$
and $w_\delta\neq -i$ for some $i\in\BN$. The arguments for the
other two cases are similar.

If $(m, n)$ is in $B^1(\delta)$, then there exists $z\in \SR_L$ such
that $(\delta(\pi)\varphi_q-1)z=m$ and $\nabla_\delta z=n$. Since
$m$ is in $\SR_L^{\psi=0}$, by Corollary \ref{cor:no2} we have $z\in
\SR_L^+$. It follows that $m$ is in $\SR^+_L\cap
\SE_L^\dagger=\SE_L^+$. By Lemma \ref{lem:no1} (\ref{it:no1-b}), we
have $\partial^i m|_{u_\CF=0}=0$.

Now we assume that $m$ is in $\SE_L^+$ and $\partial^i
m|_{u_\CF=0}=0$. By Lemma \ref{lem:no1} (\ref{it:no1-b}), there
exists $z\in \SR_L^+ $ with $\partial^i z|_{u_\CF=0}=0$ such that
$(\delta(\pi)\varphi_q-1)z=m$. Then
$(\delta(\pi)\varphi_q-1)(\nabla_\delta z-n)= \nabla_\delta
(\delta(\pi)\varphi_q-1) z - (\delta(\pi)\varphi_q-1) n =
\nabla_\delta m -  (\delta(\pi)\varphi_q-1) n  = 0$. Again by Lemma
\ref{lem:no1} (\ref{it:no1-b}), we have $ \nabla_\delta z-n = c \:
t_\CF^i $ for some $ c\in L $. Put $
z'=z-\frac{c\:t^i_\CF}{w_\delta+i} $. Then $
(\delta(\pi)\varphi_q-1)z'=m $ and $ \nabla_\delta z'=n $. Hence $
(m,n) $ is in $ B^1(\delta) $.
\end{proof}

Recall that $S_\delta=\SR_L^-(\delta)^{\Gamma=1, \psi=0}$.

\begin{prop} \label{prop:H1-case1}
Suppose that $v_\pi(\delta(\pi))<1-v_\pi(q)$.
\begin{enumerate}
\item\label{it:dim-1} If $\delta$ is not of the form $x^{-i}$, then $H_{\an}^1(\delta)$ is
isomorphic to the $L$-vector space $S_\delta$ and is
$1$-dimensional.
\item\label{it:dim-2} If $\delta=x^{-i}$, then $H_{\an}^1(\delta)$ is $2$-dimensional
over $L$ and is generated by the images of $(t_\CF^i,0)$ and
$(0,t_\CF^i)$.
\end{enumerate}
\end{prop}
\begin{proof}
For (\ref{it:dim-1}) we only consider the case that
$\delta(\pi)=\pi^{-i}$ and $w_\delta=-i$ for some $i\in \BN$. The
arguments for the other cases are similar. As $\delta\neq x^{-i}$,
there exists an element $\gamma_1\in\Gamma$ of infinite order such
that $\delta(\gamma_1)\neq \chi_\CF(\gamma_1)^{-i}$.

We give two useful facts: for any $z\in \SR^+_L$, $\partial^i
z|_{u_\CF=0}=0$ if and only if
$\partial^i(\delta(\gamma_1)\gamma_1-1)z|_{u_\CF=0}=0$; if
$\partial^i z|_{u_\CF=0}=0$, then $\partial^i
(\delta(\gamma)\gamma-1) z|_{u_\CF=0}=0$ for any $\gamma\in\Gamma$.
Both of these two facts follow from Lemma \ref{lem:no1}
(\ref{it:no1-b}). We will use them freely below.

Let $(m, n)$ be in $ Z^1(\delta)$ with $m\in
(\SE^\dagger_L)^{\psi=0}$ and $n\in \SR^+_L$. For any $\gamma\in
\Gamma$, since $\gamma(m,n)-(m,n)\in B^1(\delta)$, by Lemma
\ref{lem:pre}, $(\delta(\gamma)\gamma-1)m$ is in $\SR^+_L$ , i.e.
the image of $m$ in $\SR_L^-(\delta)$ belongs to $S_\delta$.

We will show that, for any $\bar{m}\in S_\delta$, there exists a
lifting $m\in (\SE^\dagger_L)^{\psi=0}$ of $\bar{m}$ such that
$\partial^i(\delta(\gamma)\gamma-1)m|_{u_\CF=0}=0$ for all
$\gamma\in \Gamma$. Let $m'\in (\SE^\dagger_L)^{\psi=0}$ be an
arbitrary lifting of $\bar{m}$. Assume that
$\partial^i(\delta(\gamma_1)\gamma_1-1)m'|_{u_\CF=0}=c $. Put
$m=m'-\frac{1}{i!}\frac{c\:
t^i_\CF}{\delta(\gamma_1)\chi_\CF(\gamma_1)^i-1}$. Then
$\partial^i(\delta(\gamma_1)\gamma_1-1)m|_{u_\CF=0}=0$ and thus
$\partial^i\nabla_\delta m|_{u_\CF=0}=0$. Hence, by Lemma
\ref{lem:no1} (\ref{it:no1-b}) there exists $n\in \SR^+_L$ with
$\partial^i n|_{u_\CF=0}=0$ such that
$(\delta(\pi)\varphi_q-1)n=\nabla_\delta m$. This means that
$(m,n)\in Z^1_{\varphi_q,\nabla}(\delta)$. For any
$\gamma\in\Gamma$, as
$\partial^i(\delta(\gamma_1)\gamma_1-1)(\delta(\gamma)\gamma-1)m|_{u_\CF=0}=\partial^i(\delta(\gamma)\gamma-1)(\delta(\gamma_1)\gamma_1-1)m|_{u_\CF=0}=0$,
we have $\partial^i(\delta(\gamma)\gamma-1)m|_{u_\CF=0}=0$. In a
word, for any $\gamma\in \Gamma$, $(\delta(\gamma)\gamma-1)m$ is in
$\SR^+_L$ and
$\partial^i(\delta(\gamma)\gamma-1)m|_{u_\CF=0}=\partial^i(\delta(\gamma)\gamma-1)n|_{u_\CF=0}=0$.
This means that $\gamma(m,n)-(m,n)$ is in $B^1(\delta)$ for any
$\gamma\in \Gamma$. In other words, $(m,n)$ is in $Z^1(\delta)$.

Now let $(m_1, n_1)$ and $(m_2,n_2)$ be two elements of
$Z^1(\delta)$ with $m_1, m_2\in (\SE_L^\dagger)^{\psi=0}$ and
$n_1,n_2\in \SR^+_L$.  By Lemma \ref{lem:pre}, % we have
$$ \partial^i(\delta(\gamma_1)\gamma_1-1)m_1|_{u_\CF=0}= \partial^i(\delta(\gamma_1)\gamma_1-1)m_2|_{u_\CF=0}
=\partial^i(\delta(\gamma_1)\gamma_1-1)n_1|_{u_\CF=0}=\partial^i(\delta(\gamma_1)\gamma_1-1)n_2|_{u_\CF=0}=0.
$$ Suppose that the image of $m_1$ in $S_\delta$
coincides with that of $m_2$, which implies that $m_1-m_2\in
\SE^+_L$. From
$$\partial^i(\delta(\gamma_1)\gamma_1-1)(m_1-m_2)|_{u_\CF=0}=
\partial^i(\delta(\gamma_1)\gamma_1-1)(n_1-n_2)|_{u_\CF=0}=0$$ we obtain $\partial^i(m_1-m_2)|_{u_\CF=0}=
\partial^i(n_1-n_2)|_{u_\CF=0}=0$. This means that
$(m_1,n_1)\sim (m_2, n_2)$.

Combining all of the above discussions, we obtain an isomorphism
$S_\delta\xrightarrow{\sim} H_{\an}^1(\delta)$. Then by Proposition
\ref{thm:dim}, $\dim_L H_{\an}^1(\delta)=\dim_L S_\delta=1$.

Next we prove (\ref{it:dim-2}). Again let $(m, n)$ be in $
Z^1(\delta)$ with $m\in (\SE^\dagger_L)^{\psi=0}$ and $n\in
\SR^+_L$. Then the image of $m$ in $\SR_L^-(\delta)$, denoted by
$\bar{m}$, is in $S_\delta$. We show that $m$ in fact belongs to
$(\SR^+_L)^{\psi=0}$, i.e. $\bar{m}=0$. By Corollary
\ref{prop:tran-S}, $\partial^i: S_\delta\rightarrow S_1$ is an
isomorphism. So we only need to prove that the image of $\partial^i
m$ in $S_1$ is zero. By Remark \ref{rem:constant-non-zero}, it
suffices to show that $\nabla\partial^i m |_{u_\CF=0}=0$. But
$\nabla\partial^i m=\partial^i\nabla_\delta m$. Since $\nabla_\delta
m=(\delta(\pi)\varphi_q-1)n$, by Lemma \ref{lem:no5}
(\ref{it:no1-b}) we have $\partial^i\nabla_\delta m |_{u_\CF=0}= 0$.

Write $m=at^i_\CF+m'$ with $a\in L$ and $m'\in \SR^+_L$ satisfying
$\partial^i m'|_{u_\CF=0}=0$. By Lemma \ref{lem:no5}
(\ref{it:no1-b}) there exists $z\in \SR^+_L$ such that
$(\delta(\pi)\varphi_q-1)z=m'$. Then $(m,n)\sim (at_\CF^i,
n-\nabla_\delta z)$. Thus we may suppose that $m=a t_\CF^i$. Then
$(\delta(\pi)\varphi_q-1)n=\nabla_\delta(at^i_\CF)=0$. So, by Lemma
\ref{lem:no5} (\ref{it:no1-b}), we have $n=bt_\CF^i$ for some $b\in
L$. Suppose $(at_\CF^i, bt_\CF^i)$ is in $B^1(\delta)$. Then there
exists $z\in \SR_L$ such that $(\delta(\pi)\varphi_q-1)z=at_\CF^i$
and $\nabla_\delta z=bt_\CF^i$. So $\psi(z)-\delta(\pi)z =
\psi((1-\delta(\pi)\varphi_q)z) = \psi( -a t_\CF^i ) \in \SR_L^+$.
By Lemma \ref{lem:no3} we get $z\in \SR_L^+$. By Lemma \ref{lem:no5}
(\ref{it:no1-b}) again we have $a=0$ and $z\in Lt_\CF^i$. Then
$bt^i_\CF=\nabla_\delta z=0$.
\end{proof}

\subsection{$\partial: H_{\varphi_q,\nabla}^1(x^{-1}\delta)\rightarrow
H_{\varphi_q,\nabla}^1(\delta)$ and $\partial:
H_{\an}^1(x^{-1}\delta)\rightarrow H_{\an}^1(\delta)$}
\label{ss:H1nabla-par} \label{ss:H1an-par}

Observe that, if $(m,n)$ is in
$Z^1_{\varphi_q,\nabla}(x^{-1}\delta)$ (resp. $B^1(x^{-1}\delta)$),
then $(\partial m,\partial n)$ is in
$Z^1_{\varphi_q,\nabla}(\delta)$ (resp. $B^1(\delta)$). Thus we have
a map $\partial: H_{\varphi_q,\nabla}^1(x^{-1}\delta)\rightarrow
H_{\varphi_q,\nabla}^1(\delta)$. Further, the map is
$\Gamma$-equivariant and it induces a map $\partial:
H_{\an}^1(x^{-1}\delta)\rightarrow H_{\an}^1(\delta)$.

Put $\bar{Z}_{\varphi_q,\nabla}^1(\delta):=\{(m,n)\in
Z^1_{\varphi_q,\nabla}(\delta): \Res(m)=\Res(n)=0\}$ and
$\bar{B}^1(\delta):=\{(m,n)\in B^1(\delta): \Res(m)=\Res(n)=0\}$.
Then $\bar{H}^1_{\varphi_q,\nabla}(\delta):=
\bar{Z}_{\varphi_q,\nabla}^1(\delta)/\bar{B}_{\varphi_q,\nabla}^1(\delta)$
is a subspace of $H^1_{\varphi_q,\nabla}(\delta)$.

\begin{lem}\label{lem:bar-nonbar} If $\delta(\pi)\neq \pi/q$ or $w_\delta\neq 1$, then for
any $(m,n)\in Z^1_{\varphi_q,\nabla}(\delta)$, there exists
$(m_1,n_1)\in \bar{Z}_{\varphi_q,\nabla}^1(\delta)$ such that
$(m,n)\sim (m_1,n_1)$, and so
$H^1_{\varphi_q,\nabla}(\delta)=\bar{H}_{\varphi_q,\nabla}^1(\delta)$.
\end{lem}
\begin{proof} Let $(m,n)$ be in $ Z^1_{\varphi_q,\nabla}(\delta)$. Then
$\nabla_\delta m = (\delta(\pi)\varphi_q-1)n$. If $\delta(\pi)\neq
\frac{\pi}{q}$, by Proposition \ref{prop-tate-trace} and the
definition of $\Res$ we have
$$ \Res\Big( m-(\delta(\pi)\varphi_q-1)\big(\Res(m)\frac{(\frac{\rmd t_\CF}{\rmd
u_\CF})^{-1}}{(\delta(\pi)\frac{q}{\pi}-1)u_\CF}\big) \Big) = 0 . $$
Replacing $(m,n)$ by
$$\Big(m-(\delta(\pi)\varphi_q-1)\big(\Res(m)\frac{(\frac{\rmd t_\CF}{\rmd
u_\CF})^{-1}}{(\delta(\pi)\frac{q}{\pi}-1)u_\CF}\big), n-
\nabla_\delta \big(\Res(m)\frac{(\frac{\rmd t_\CF}{\rmd
u_\CF})^{-1}}{(\delta(\pi)\frac{q}{\pi}-1)u_\CF}\big)\Big),$$ we may
assume that $\Res(m)=0$. Then
$$(\frac{q}{\pi}\delta(\pi)-1)\Res(n)=\Res((\delta(\pi)\varphi_q-1)n)=
\Res(\nabla_\delta m)=\Res(\partial (t_\CF
m)+(w_\delta-1)m)=(w_\delta-1)\Res(m) =0,$$ and so $\Res(n)=0$.

The argument for the case of $w_\delta\neq 1$ is similar.
\end{proof}

As $\Res\circ \partial=0$, the map $\partial:
H_{\varphi_q,\nabla}^1(x^{-1}\delta)\rightarrow
H_{\varphi_q,\nabla}^1(\delta)$ factors through $\partial:
H_{\varphi_q,\nabla}^1(x^{-1}\delta)\rightarrow
\bar{H}_{\varphi_q,\nabla}^1(\delta)$.

\begin{lem} \label{lem:partial-1}
\begin{enumerate}
\item\label{it:sur-partial-a} If $\delta(\pi)\neq \pi$ or $w_\delta\neq 1$, then $\partial:
H_{\varphi_q,\nabla}^1(x^{-1}\delta)\rightarrow
\bar{H}_{\varphi_q,\nabla}^1(\delta)$ is surjective.
\item If $\delta(\pi)=\pi$ and $w_\delta=1$, then we have an exact
sequence of $\Gamma$-modules
\[ \xymatrix{
H_{\varphi_q,\nabla}^1(x^{-1}\delta)\ar[r]^{\partial} &
\bar{H}_{\varphi_q,\nabla}^1(\delta) \ar[r] & L(x^{-1}\delta) \ar[r]
& 0. }\]
\end{enumerate}
\end{lem}
\begin{proof} Let $(m,n)$ be in $\bar{Z}^1_{\varphi_q,\nabla}(\delta)$. Then there exist $m'$ and $n'$ such that $\partial m'=m$ and $\partial n'=n$.
Then $\nabla_{x^{-1}\delta} m'-(\pi^{-1}\delta(\pi)\varphi_q-1)n'=c$
is in $L$. If $\delta(\pi)\neq \pi$, we replace $n'$ by
$n'+\frac{c}{\pi^{-1}\delta(\pi) -1}$. If $w_\delta\neq 1$, we
replace $m'$ by $m'-\frac{c}{w_\delta-1}$. Then $(m',n')$ is in
$Z^1_{\varphi_q,\nabla}(x^{-1}\delta)$. This proves
(\ref{it:sur-partial-a}). When $\delta(\pi)=\pi$ and $w_\delta=1$,
$\nabla m'-(\varphi_q-1)n'$ does not depend on the choice of $m'$
and $n'$. This induces a map
$\bar{H}_{\varphi_q,\nabla}^1(\delta)\rightarrow L$ whose kernel is
exactly $\partial H^1_{\varphi_q,\nabla}(x^{-1}\delta)$. We show
that $\bar{H}_{\varphi_q,\nabla}^1(\delta)\rightarrow L$ is
surjective. Put $m'=\log\frac{\varphi_q(u_\CF)}{u_\CF^q}$. A simple
calculation shows that
$$\nabla m'= ( \frac{t_\CF \cdot [\pi]'_\CF(u_\CF)}{[\pi]_\CF(u_\CF)} -q \frac{ t_\CF}{u_\CF} ) \partial u_\CF \equiv (1-q) \mod u_\CF\SR^+_L.$$
Thus by Lemma \ref{lem:no1} (\ref{it:no1-b}) there exists $n'\in
u_\CF\SR^+_L$ such that $(\varphi_q-1)n'=\nabla m'-(1-q)$. Put
$m=\partial m'$ and $n=\partial n'$. Then $(m,n)$ is in
$\bar{Z}^1_{\varphi_q,\nabla}(\delta)$ whose image in $L$ is
nonzero. The $\Gamma$-action on
$\bar{H}_{\varphi_q,\nabla}^1(\delta)$ induces an action on $L$.
From
$$ (\delta(a)\sigma_a (m), \delta(a)\sigma_a(n))= (\partial(a^{-1}\delta(a)\sigma_a (m')),\partial(a^{-1}\delta(a)\sigma_a (n')))$$
and $$ \nabla(a^{-1}\delta(a)\sigma_a
(m'))-(\varphi_q-1)(a^{-1}\delta(a)\sigma_a (n')) = a^{-1}\delta(a)
\sigma_a(\nabla m'- (\varphi_q-1)n') \equiv a^{-1}\delta(a)(1-q)
\mod u_\CF \SR^+_L
$$ we see that
the induced action comes from the character $x^{-1}\delta$.
\end{proof}

\begin{sublem} \label{sublem:partial-1}
Let $a, b$ be in $ L$. If $(a,b)$ is in
$Z^1_{\varphi_q,\nabla}(x^{-1}\delta)$ but not in
$B^1(x^{-1}\delta)$, then $\delta(\pi)=\pi$ and $w_\delta=1$.
\end{sublem}
\begin{proof}  If $\delta(\pi)\neq \pi$, then $(a,b)\sim (0, b-\frac{\nabla_{x^{-1}\delta}}{\pi^{-1}\delta(\pi)-1}
a)$. So
$$(\pi^{-1}\delta(\pi)-1)(b-\frac{\nabla_{x^{-1}\delta}}{\pi^{-1}\delta(\pi)-1}
a)=(\pi^{-1}\delta(\pi)\varphi_q-1)(b-\frac{\nabla_{x^{-1}\delta}}{\pi^{-1}\delta(\pi)-1}
a)=0.$$ As $\delta(\pi)\neq \pi$, we have
$b-\frac{\nabla_{x^{-1}\delta}}{\pi^{-1}\delta(\pi)-1} a=0$.
Similarly, if $w_\delta\neq 1$, then $(a,b)\in
Z^1_{\varphi_q,\nabla}(x^{-1}\delta)$ if and only if $(a,b)\sim
(0,0)$.
\end{proof}

Recall that $\delta_\unr$ is the character of $F^\times$ such that
$\delta_\unr(\pi)=q^{-1}$ and $\delta_{\unr}|_{\CO_F^\times}=1$.

\begin{sublem} \label{lem:H1-exam}
$(\frac{1}{q}\log\frac{\varphi_q(u_\CF)}{u_\CF^q}, \frac{t_\CF
\partial u_\CF}{u_\CF})$ induces a nonzero element of
$H_{\an}^1(\delta_\unr)$.
\end{sublem}
\begin{proof} Write $(m,n)=(\frac{1}{q}\log\frac{\varphi_q(u_\CF)}{u_\CF^q}, \frac{t_\CF \partial
u_\CF}{u_\CF})$. Note that $m= (\delta_\unr(\pi)\varphi_q-1)\log
u_\CF $ and $n= \nabla \log u_\CF $. Thus $(m,n)$ is in
$Z^1_{\varphi_q,\nabla}(\delta_\unr)$. For any $\gamma\in \Gamma$ we
have $\gamma(m,n)\sim (m,n)$. Indeed,
$\gamma(m,n)-(m,n)=((\delta_\unr(\pi)\varphi_q-1)\log
\frac{\gamma(u_\CF)}{u_\CF} , \nabla \log
\frac{\gamma(u_\CF)}{u_\CF})$. So $(m,n)$ is in $Z^1(\delta_\unr)$.
We show that $(m,n)$ is not in $B^1(\delta_\unr)$. Otherwise there
exists $z\in \SR_L$ such that $m=(\delta_\unr(\pi)\varphi_q-1)z$ and
$n=\nabla z$. This will implies that $\nabla(\log u_\CF -z)=0$ or
equivalently $\log u_\CF -z$ is in $L$, a contradiction.
\end{proof}

\begin{cor}\label{cor-of-sublem} If $\delta(\pi)=\pi/q$ and $w_\delta=1$, then $(\frac{1}{q}\log\frac{\varphi_q(u_\CF)}{u_\CF^q}, \frac{t_\CF
\partial u_\CF}{u_\CF})$ is in $Z^1_{\varphi_q,\nabla}(x^{-1}\delta)$ but not in
$B^1(x^{-1}\delta)$.
\end{cor}

\begin{lem}\label{lem:partial-2}
\begin{enumerate}
\item If $\delta(\pi)\neq \pi, \pi/q$ or if $w_\delta\neq 1$, then
$\partial: H^1_{\varphi_q,\nabla}(x^{-1}\delta) \rightarrow
\bar{H}^1_{\varphi_q,\nabla}(\delta) $ is injective.
\item If $\delta(\pi)=\pi$ and $w_\delta=1$, then we have an exact
sequence of $\Gamma$-modules
$$ \xymatrix{ 0 \ar[r] & L(x^{-1}\delta) \oplus L(x^{-1}\delta) \ar[r] & H^1_{\varphi_q,\nabla}(x^{-1}\delta) \ar[r]^{\partial} & \bar{H}^1_{\varphi_q,\nabla}(\delta)
. }$$
\item If If $\delta(\pi)=\pi/q$ and $w_\delta=1$, then we have an exact
sequence of $\Gamma$-modules
$$ \xymatrix{ 0 \ar[r] & L(x^{-1}\delta) \ar[r] & H^1_{\varphi_q,\nabla}(x^{-1}\delta)
\ar[r]^{\partial} & \bar{H}^1_{\varphi_q,\nabla}(\delta)
. }$$
\end{enumerate}
\end{lem}
\begin{proof} Let $
(m,n)$ be in $ Z^1_{\varphi_q,\nabla}(x^{-1}\delta) $, and suppose
that $(\partial m,
\partial n)\in \bar{B}^1(\delta)$. Let $z$ be an element of $\SR_L$ such that
$(\delta(\pi)\varphi_q-1)z=\partial m$ and $\nabla_\delta z=\partial
n$. If $\Res(z)=0$, then there exists $z'\in \SR_L$ such that
$\partial z'=z$. Then $m-(\delta(\pi)\pi^{-1}\varphi_q-1)z'$ and
$n-\nabla_{x^{-1}\delta}z'$ are in $\{(a,b): a,b\in L\}$, i.e.
$(m,n)$ is in $B^1(x^{-1}\delta)\oplus L(0,1)\oplus L(1,0)$.

If either $\delta(\pi)\neq \frac{\pi}{q}$ or $w_\delta\neq 1$, we
always have $\Res(z)=0$. Indeed, this follows from
$$(\delta(\pi)\frac{q}{\pi}-1)\Res(z)=\Res( (\delta(\pi)\varphi_q-1) z ) =
\Res(\partial m)=0 $$ and
$$ (w_\delta-1)  \Res(z) = \Res(\partial (t_\CF z)+(w_\delta-1)z) = \Res(\nabla_\delta
z)=\Res(\partial n) =0. $$ In the case of
$\delta(\pi)=\frac{\pi}{q}$ and $w_\delta=1$, if $z\in
L\frac{\partial u_\CF}{u_\CF}$, then $(m,n)$ is in $L(0,1) \oplus
L(1,0) \oplus L( \frac{1}{q}\log \frac{\varphi_q (u_\CF)}{u_\CF^q},
\frac{t_\CF\partial u_\CF}{u_\CF} )$. Now our lemma follows from
Sublemma \ref{sublem:partial-1} and Corollary \ref{cor-of-sublem}.
\end{proof}

\begin{prop}\label{prop:partial-iso-1}
\begin{enumerate}
\item\label{it:partial-a} If $\delta(\pi)\neq \pi,\pi/q$ or if $w_\delta\neq 1$, then $\partial: H^1_{\varphi_q,\nabla}(x^{-1}\delta)\rightarrow
H^1_{\varphi_q,\nabla}(\delta)$ is an isomorphism of
$\Gamma$-modules.
\item\label{it:partial-b} If $\delta(\pi)=\pi$ and $w_{\delta}=1$, then we have an exact
sequence of $\Gamma$-modules
\[ \xymatrix{ 0 \ar[r] & L(x^{-1}\delta)\oplus L(x^{-1}\delta) \ar[r] & H^1_{\varphi_q,\nabla} (x^{-1}\delta) \ar[r]^{\partial} & H^1_{\varphi_q,\nabla}(\delta) \ar[r]
& L(x^{-1}\delta) \ar[r] & 0. } \]
\item\label{it:partial-c} If $\delta(\pi)=\pi/q$ and $w_{\delta}=1$, then we have an exact
sequence of $\Gamma$-modules
\[ \xymatrix{ 0 \ar[r] & L(x^{-1}\delta) \ar[r] & H^1_{\varphi_q,\nabla} (x^{-1}\delta) \ar[r]^{\partial} & H^1_{\varphi_q,\nabla}(\delta)
\ar[r] & L(x^{-1}\delta) \oplus L(x^{-1}\delta) \ar[r] & 0. }
\]
\end{enumerate}
\end{prop}
\begin{proof} Assertions (\ref{it:partial-a}) and
(\ref{it:partial-b}) follow from Lemma \ref{lem:bar-nonbar}, Lemma
\ref{lem:partial-1} and Lemma \ref{lem:partial-2}. Based on these
lemmas, for (\ref{it:partial-c}) we only need to show that, we have
an exact sequence of $\Gamma$-modules
$$\xymatrix{ 0 \ar[r] & \bar{H}^1_{\varphi_q,\nabla}(\delta)\ar[r]
& H^1_{\varphi_q,\nabla}(\delta)\ar[r]^{\hskip -18pt \Res} &
L(x^{-1}\delta)\oplus L(x^{-1}\delta) \ar[r] & 0, }$$ where $\Res$
is induced by $(m,n)\mapsto (\Res (m), \Res(n))$ which is
$\Gamma$-equivariant by Proposition \ref{prop-tate-trace}. Here we
prove this under the assumption that $q$ is not a power of $\pi$. We
will see in Section \ref{ss:iota-k} that it also holds without this
assumption. Put $m_1=1/u_\CF $. Then $ \nabla_\delta m_1 = t_\CF
\partial (1/u_\CF) + 1/u_\CF =\partial (t_\CF/u_\CF) $ is in
$\SR^+_L$. As $q$ is not a power of $\pi$, the map
$\frac{\pi}{q}\varphi_q-1:\SR^+_L\rightarrow \SR^+_L$ is an
isomorphism. Let $n_1$ be the unique solution of
$(\frac{\pi}{q}\varphi_q-1)n_1=t_\CF\partial m_1 +m_1$ in $\SR^+_L$.
Then $c_1=(m_1,n_1)$ is in $Z^1_{\varphi_q,\nabla}(\delta)$ and
$\Res(m_1,n_1)=(1,0)\neq 0$. For any $\ell\in\BN$ we choose a root
$\xi_\ell$ of $Q_\ell=\varphi_q^{\ell-1}(Q)$. For any $f(u_\CF)\in
\SR^+_L$, the value of $f$ at $\xi_\ell$ is an element $f(\xi_\ell)$
in $L\otimes_F F_\ell$. By (\ref{eq:inc-iso}) there exists an
element $z\in \SR^+_L$ whose value at $\xi_\ell$ is $1\otimes\log
\xi_\ell$. Put $m_2=t_\CF^{-1} (q^{-1}\varphi_q-1) ( \log u_\CF - z
)$ and $n_2=\partial (\log u_\CF - z)$. Then $(m_2,n_2)$ is in
$Z^1_{\varphi_q,\nabla}(\delta)$ and $\Res(n_2)=1$.
\end{proof}

\begin{prop} \label{prop:partial-iso-2}
\begin{enumerate}
\item \label{it:partial-iso} If $\delta\neq x, \ x\delta_\unr$, then $\partial: H_{\an}^1(x^{-1}\delta)\rightarrow
H_{\an}^1(\delta)$ is an isomorphism.
\item \label{it:par-zero-1} If $\delta=x$, then $\partial: H_{\an}^1(x^{-1}\delta)\rightarrow
H_{\an}^1(\delta)$ is zero, and $\dim_LH_{\an}^1(\delta)=1$.
\item \label{it:par-zero-2} If $\delta=x\delta_\unr$, then
$\partial: H_{\an}^1(x^{-1}\delta)\rightarrow H_{\an}^1(\delta)$ is
zero, and $\dim_LH_{\an}^1(\delta)=2$.
\end{enumerate}
\end{prop}
\begin{proof}
We apply Proposition \ref{prop:partial-iso-1}. There is nothing to
prove for the case that $\delta(\pi)\neq \pi, \pi/q$ or
$w_\delta\neq 1$. Combining the assertions in this case and
Proposition \ref{prop:H1-case1} we obtain that $\dim_L
H^1_\an(\delta_\unr)=1$. This fact is useful below.

Next we consider the case of $\delta(\pi)=\pi/q$ and $w_{\delta}=1$.
The argument for the case of $\delta(\pi)=\pi$ and $w_\delta=1$ is
similar.

Let $M$ be the image of $\partial: H^1_{\varphi_q,\nabla}
(x^{-1}\delta) \rightarrow H^1_{\varphi_q,\nabla}(x)$. Then we have
two short exact sequences of $\Gamma$-modules
\[ \xymatrix{ 0 \ar[r] & L(x^{-1}\delta) \ar[r] & H^1_{\varphi_q,\nabla} (x^{-1}\delta)
\ar[r]^{\hskip 15pt \partial} & M \ar[r] & 0 } \] and
 \[\xymatrix{ 0 \ar[r] & M  \ar[r] & H^1_{\varphi_q,\nabla}(\delta) \ar[r]
& L(x^{-1}\delta) \oplus L(x^{-1}\delta) \ar[r] & 0 . } \] We will
show that, taking $\Gamma$-invariants yields two exact sequences
\[ \xymatrix{ 0 \ar[r] &
L(x^{-1}\delta)^\Gamma \ar[r] & H^1_{\an} (x^{-1}\delta)
\ar[r]^{\partial} & M^\Gamma \ar[r] & 0}
\] and
 \[\xymatrix{ 0 \ar[r] & M^\Gamma  \ar[r] & H^1_{\an}(\delta) \ar[r]
& L(x^{-1}\delta)^\Gamma \oplus L(x^{-1}\delta)^\Gamma \ar[r] & 0 .
}
\] If we have that the $\Gamma$-actions on $H^1_{\varphi_q,\nabla} (x^{-1}\delta)$ and
$H^1_{\varphi_q,\nabla}(\delta)$ are semisimple, then there is
nothing to prove. However we will avoid this by an alternative
argument. It suffices to prove the surjectivity of
$H^1_{\varphi_q,\nabla} (x^{-1}\delta)^\Gamma \rightarrow M^\Gamma$
and that of $H^1_{\varphi_q,\nabla}(\delta)^\Gamma \rightarrow
L(x^{-1}\delta)^\Gamma \oplus L(x^{-1}\delta)^\Gamma$. The latter
follows from the proof of Proposition \ref{prop:partial-iso-1}. In
fact, if $\delta=x\delta_\unr$, then $(m_1, n_1)$ and $(m_2,n_2)$
constructed there are in $Z^1(\delta)$. Now let $c$ be any element
of $M^\Gamma$, then the preimage $\partial^{-1} (L c)$ is two
dimensional over $L$ and $\Gamma$-invariant. From the definition of
$H^1_{\varphi_q,\nabla}$, we obtain that the induced $\nabla$-action
on $\partial^{-1} (L c)$ is zero and thus $\partial^{-1} (L c)$ is a
semisimple $\Gamma$-module, as wanted.

If $\delta=x\delta_\unr$, then $\dim_L L(x^{-1}\delta)^\Gamma=\dim_L
H^1_\an (x^{-1}\delta)=1$, and so $M^\Gamma=0$. Thus $\partial:
H^1_\an(x^{-1}\delta)\rightarrow H^1_\an(\delta)$ is zero and
$\dim_L H^1_\an(\delta)=2$. If $\delta\neq x\delta_\unr$, then
$\partial: H^1_\an(x^{-1}\delta)\rightarrow H^1_\an(\delta)$ is an
isomorphism since both $H^1_\an(x^{-1}\delta)\rightarrow M^\Gamma$
and $M^\Gamma\rightarrow H^1_\an(\delta)$ are isomorphisms.
\end{proof}

\subsection{Dimension of $
H^1(\delta)$ for $\delta\in \SI(L)$} \label{ss:H1}

\begin{thm}(= Theorem \ref{thm:intro-dim})\label{thm:dim-H1} Let $\delta$ be in $\SI_\an(L)$.
\begin{enumerate}
\item\label{it:dim-a} If $\delta$ is not of the form $x^{-i}$ with $i\in \BN$, or the
form $x^i\delta_\unr$ with $i\in\BZ_+$, then $H_{\an}^1(\delta)$ and
$H^1(\delta)$ are $1$-dimensional over $L$.
\item \label{it:dim-b} If  $\delta=x^i\delta_\unr$ with $i\in\BZ_+$,
then $H^1_\an(\delta)$ and $H^1(\delta)$ are $2$-dimensional over
$L$.
\item If
$\delta=x^{-i}$ with $i\in \BN$, then $H^1_\an(\delta)$ is
$2$-dimensional over $L$ and $H^1(\delta)$ is $(d+1)$-dimensional
over $L$, where $d=[F:\BQ_p]$.
\end{enumerate}
\end{thm}
\begin{proof} The assertions for $H^1_\an(\delta)$ follow from Proposition \ref{prop:H1-case1} and Proposition
\ref{prop:partial-iso-2}. By Proposition \ref{prop:H0} we have
$$\dim_L \SR_L(\delta)^{\varphi_q=1,\Gamma=1}=\left\{
\begin{array}{ll} 1 & \text{ if } \delta=x^{-i} \text{ with } i\in \BN,
\\
0 & \text{ otherwise.}
\end{array} \right.$$ So the assertions for $H^1(\delta)$ come
from the assertions for $H^1_\an(\delta)$ and Corollary
\ref{thm:comp}.
\end{proof}

When $\delta=x^{-i}$ with $i\in \BN$, $H^1_\an(\delta)$ is generated
by the classes of $(t_\CF^i,0)$ and $(0, t_\CF^i)$. Let $\rho_i$
$(i=1,\cdots, d)$ be a basis of $\Hom(\Gamma, Lt_\CF^i)$. Then the
class of the $1$-cocycle $c_0$ with $c_0(\varphi_q)=t_\CF^i$ and
$c_0|\Gamma=0$, and the classes of $1$-cocycles $c_i$ with
$c_i(\varphi_q)=0$ and $c_i|\Gamma=\rho_i$ ($i=1,\cdots, d$), form a
basis of $H^1(\delta)$.

\begin{thm}\label{thm:not-anal} (=Theorem \ref{thm:intro-nonanal}) If $\delta\in \SI(L)$ is not locally $F$-analytic, then
$H^1(\delta)=0$.
\end{thm}
\begin{proof} As the maps \(\gamma-1\), $\gamma\in \Gamma$, are null on \(H^1(\delta)\), by definition of
\(H^1\), so are the maps \(\rd\Gamma_{\SR_L(\delta)}(\beta)\),
\(\beta \in \Lie\Gamma\), and the differences \(\beta^{-1}
\rd\Gamma_{\SR_L(\delta)}(\beta) - \beta'^{-1}
\rd\Gamma_{\SR_L(\delta)}(\beta')\). Note that \(\beta^{-1}
\rd\Gamma_{\SR_L(\delta)}(\beta) - \beta'^{-1}
\rd\Gamma_{\SR_L(\delta)}(\beta')\) are \(\SR_L\)-linear on
$\SR_L(\delta)$. So \(\beta^{-1} \rd\Gamma_{\SR_L(\delta)}(\beta) -
\beta'^{-1} \rd\Gamma_{\SR_L(\delta)}(\beta')\) are multiplications
by scalars in $L$, since \(\beta^{-1}
\rd\Gamma_{\SR_L(\delta)}(\beta) e_\delta- \beta'^{-1}
\rd\Gamma_{\SR_L(\delta)}(\beta') e_\delta\) is in $Le_\delta$. If
the intersection of their kernels is null, then the cohomology
\(H^1(\delta)\) vanishes. Thus, either the intersection of their
kernels is \(0\) and so the cohomology vanishes, or they are all
null and \(\delta\) is of form \(x \mapsto x^w\) for \(x\) close
to~\(1\) with \(w = \frac{\log\delta(\beta)}{\log\beta}\) for
\(\beta\) close to \(1\) (i.e. $\delta$ is locally $F$-analytic).
\end{proof}

\begin{rem}\label{rem:non-over} Suppose that $[F:\BQ_p]\geq 2$.  Let $\delta\neq 1$ be a character of
$F^\times$ with $\delta(\pi)\in \CO_L^\times$, and let $L(\delta)$
be the $L$-representation of $G_F$ induced by $\delta$. Suppose that
$\delta\neq x^2\delta_\unr$ when $[F:\BQ_p]=2$. Combining Theorem
\ref{thm:dim-H1} and the Euler-Poincar\'e characteristic formula
\cite{Tate} we obtain that, there exist Galois representations in
$\Ext(L,L(\delta))$ that are not overconvergent. Theorem
\ref{thm:not-anal} tells us that, if further $\delta$ is not locally
analytic, then there is no nontrivial overconvergent extension of
$L$ by $L(\delta)$.
\end{rem}

\subsection{The maps $\iota_k : H^1(\delta)\rightarrow
H^1(x^{-k}\delta)$ and $\iota_{k,\an}: H^1_\an(\delta)\rightarrow
H^1_\an(x^{-k}\delta)$} \label{ss:iota-k}

Let $k$ be a positive integer.

\begin{prop} \label{prop:iota-pre} Let $\delta$ be in $\SI_\an(L)$.
\begin{enumerate}
\item\label{it-lem-inv-1} If $w_\delta\notin \{ 1-k,\cdots, 0 \}$, then
$H_{\an}^0(\SR_L(\delta)/t^k_\CF\SR_L(\delta))=0$.
\item\label{it-lem-inv-2} If $w_\delta\in\{1-k,\cdots, 0\}$, then
$H_{\an}^0(\SR_L(\delta)/t^k_\CF\SR_L(\delta))$ is a $1$-dimensional
$L$-vector space.
\end{enumerate}
\end{prop}
\begin{proof}
We have $\SR_L^+/t^k_\CF\SR^+_L =  \SR^+_L/(u_\CF^k) \times
\prod_{n=1}^\infty \SR^+_L / (\varphi_q^{n-1}(Q))^k $. As
$\Gamma$-modules, $\SR_L^+/(u_\CF^k) = \oplus_{i=0}^{k-1}L t^i_\CF $
and $\SR^+_L / (\varphi_q^{n}(Q))^k =
 \bigoplus_{i=0}^{k-1}(L\otimes_F  F_{n}) t_\CF^i $. Thus as a
$\Gamma$-module $\SR_L^+/t^k_\CF\SR^+_L$ is isomorphic to
$\bigoplus_{i=0}^{k-1}(\SR^+_L/\SR^+_L t_\CF) \otimes_L L t_\CF^i$.
Note that the natural map $\SR_L^+/\SR^+_L t^k_\CF \rightarrow
\SR_L/\SR_L t^k_\CF$ is surjective. Furthermore, two sequences
$(y_n)_{n\geq 0}$ and $(z_n)_{n\geq 0}$ in $\SR^+_L/\SR^+_L u_\CF^k
\times \prod_{n=1}^\infty \SR^+_L / (\varphi_q^{n-1}(Q))^k $ have
the same image in $\SR_L/\SR_L t^k_\CF$, if and only if there exists
$N>0$ such that $y_n=z_n$ when $n\geq N$.

Since the action of $\Gamma$ on $(\SR_L^+/t_\CF\SR^+_L)t_\CF^i$
twisted by the character $x^{-i}$ is smooth, (\ref{it-lem-inv-1})
follows.

For (\ref{it-lem-inv-2}) we only need to consider the case of
$w_\delta=0$ and $k=1$. The operator $\varphi_q$ induces an
injection $\SR^+_L/(\varphi_q^n(Q))\rightarrow
\SR^+_L/(\varphi_q^{n+1}(Q))$ which is denoted by $\varphi_{q,n}$.
The action of $\varphi_q$ on $\SR_L/\SR_L t_\CF$ is given by
$\varphi_q(y_n)_n=(\varphi_{q,n}(y_n))_{n+1}$. For any $n\geq 0$,
the $\Gamma$-action on $L\otimes_F F_n$ factors through
$\Gamma/\Gamma_n$, and the resulting $\Gamma/\Gamma_n$-module
$L\otimes_F F_n$ is isomorphic to the regular one. Thus for any
discrete character $\delta$ of $\Gamma$, $\dim_L (L\otimes_F
F_n)^{\Gamma=\delta^{-1}}=1$ when $n$ is sufficiently large. Then
from the fact that $\varphi_{q,n}$ $(n\geq 1)$ are injective, we
obtain $\dim_L \big(\SR_L /t_\CF\SR_L \big)^{\Gamma=\delta^{-1},
\varphi_q=\delta(\pi)^{-1}}=1$.
\end{proof}

\begin{cor} Let $\delta$ be in $\SI_\an(L)$.
\begin{enumerate}
\item  If $w_\delta\notin \{ 1,\cdots, k \}$, then
$H_{\an}^0(t^{-k}_\CF\SR_L(\delta)/\SR_L(\delta))=0$.
\item  If $w_\delta\in\{1,\cdots, k\}$, then
$H_{\an}^0(t^{-k}_\CF\SR_L(\delta)/\SR_L(\delta))$ is a
$1$-dimensional $L$-vector space.
\end{enumerate}
\end{cor}

Note that $\SR_L(x^{-k}\delta)$ is canonically isomorphic to
$t_\CF^{-k}\SR_L(\delta)$. When $k\geq 1$, the inclusion
$\SR_L(\delta)\hookrightarrow t_\CF^{-k}\SR_L(\delta)$ induces maps
$\iota_{k,\an}: H_{\an}^1(\delta)\rightarrow
H_{\an}^1(x^{-k}\delta)$ and $\iota_k : H^1(\delta)\rightarrow
H^1(x^{-k}\delta)$. If $\gamma\in\Gamma$ is of infinite order, then
we have the following commutative diagram
\begin{equation}\label{eq:comm-iota}
\xymatrix{ H^1(\delta) \ar[r]^{\iota_k }
\ar[d]^{\Upsilon^\delta_{\an,\gamma}\circ\Upsilon^\delta_\gamma} &
H^1(x^{-k}\delta)
\ar[d]^{\Upsilon^{x^{-k}\delta}_{\an,\gamma}\circ\Upsilon^{x^{-k}\delta}_\gamma } \\
H^1_\an(\delta) \ar[r]^{\iota_{k,\an}} & H_\an^1(x^{-k}\delta). }
\end{equation}

\begin{lem} We have the following exact sequence
\begin{equation} \label{eq:exact-sq}
0\rightarrow H_{\an}^0(\delta) \rightarrow
H_{\an}^0(x^{-k}\delta)\rightarrow
H_{\an}^0(t_\CF^{-k}\SR_L(\delta)/\SR_L(\delta))\rightarrow
H_{\an}^1(\delta) \xrightarrow{\iota_{k,\an}}
H_{\an}^1(x^{-k}\delta) .
\end{equation}
\end{lem}
\begin{proof} From
the short exact sequence $0\rightarrow \SR_L(\delta)\rightarrow
\SR_L(x^{-k}\delta) \rightarrow \SR_L(x^{-k}\delta)/\SR_L(\delta)
\rightarrow 0$ we deduce an exact sequence \begin{equation}
\label{eq:exact-sq-pre} 0\rightarrow H_{\varphi_q,\nabla}^0(\delta)
\rightarrow H_{\varphi_q,\nabla}^0(x^{-k}\delta)\rightarrow
H_{\varphi_q,\nabla}^0(t_\CF^{-k}\SR_L(\delta)/\SR_L(\delta))\rightarrow
H_{\varphi_q,\nabla}^1(\delta) \rightarrow
H_{\varphi_q,\nabla}^1(x^{-k}\delta) . \end{equation} Being finite
dimensional $H_{\varphi_q,\nabla}^0(\delta) $ and
$H_{\varphi_q,\nabla}^0(x^{-k}\delta)$ are semisimple
$\Gamma$-modules; since $t_\CF^{-k}\SR_L(\delta)/\SR_L(\delta)$ is a
semisimple $\Gamma$-module, so is
$H_{\varphi_q,\nabla}^0(t_\CF^{-k}\SR_L(\delta)/\SR_L(\delta))$.
Hence, taking $\Gamma$-invariants of each term in
(\ref{eq:exact-sq-pre}), we obtain the desired exact sequence.
\end{proof}

\begin{prop}\label{prop:iota}
Let $\delta$ be in $\SI_\an(L)$, $k\in\BZ_+$. If
$w_\delta\notin\{1,\cdots, k\}$, then $\iota_{k,\an}$ and $\iota_k $
are isomorphisms.
\end{prop}
\begin{proof} We only prove the assertion for $\iota_{k,\an}$. The
proof of the assertion for $\iota_k $ is similar. By Theorem
\ref{thm:dim-H1}, $\dim_L H_{\an}^1(\delta)=\dim_L
H_{\an}^1(x^{-k}\delta)$ when $w_\delta\notin\{1,\cdots, k\}$.
Combining (\ref{eq:exact-sq}) with the facts that
$H_{\an}^0(t_\CF^{-k}\SR_L(\delta)/\SR_L(\delta))=0$ and that
$\dim_L H_{\an}^1(\delta) = \dim_L H_{\an}^1(x^{-k}\delta)$, we
obtain the assertion.
\end{proof}

We assign to any nonzero $c\in H_{\an}^1(\delta)$ an $\SL$-invariant
in $\RP^1(L)=L\cup \{\infty\}$. In the case of $\delta=x^{-k}$ with
$k\in \BN$, put $\SL((at^k_\CF, bt^k_\CF))=a/b$. If
$\delta=x\delta_\unr$, then any $c\in H_{\an}^1(\delta)$ can be
written as $$c=t_\CF^{-1}( (q^{-1}\varphi_q-1) (\lambda G(1,1) + \mu
(\log u_\CF -z )) , t_\CF\partial (\lambda G(1,1) + \mu (\log u_\CF
-z )))$$ with $\lambda,\mu\in L$.  Here $G(1,1)$ is an element of
$\SR_L$ which induces a basis of $(\SR_L/\SR_Lt_\CF)^\Gamma$ and
whose value at $\xi_n$ is $1\otimes 1\in L\otimes_F  F_n$ when $n$
is large enough; $z$ is an element of $\SR_L$ whose value at $\xi_n$
is $1\otimes\log (\xi_n)\in L\otimes_F  F_n$ for any $n$. We put
$\SL(c)=-\frac{e_F (q-1) }{q }\cdot \frac{\lambda}{\mu}$. In the
case of $\delta=x^k\delta_\unr$ with $k\geq 2$, for any $c\in
H_{\an}^1(x^k\delta_\unr)$, put $\SL(c)=\SL(\iota_{k-1}(c))$. In the
case that $\delta$ is not of the form $x^{-k}$ with $k\in \BN$ or
the form $x^k\delta_\unr$ with $k\in \BZ_+$, we put $\SL(c)=\infty$.

\begin{prop} \label{prop:iota-k} Let $\delta$ be in $\SI_\an(L)$, $k\in\BZ_+$.
\begin{enumerate} % \item\label{it:iota-k-no-1}
\item\label{it:iota-k-no-2} If $w_\delta\in\{1,\cdots, k\}$ and if $\delta\neq
x^{w_\delta},$ $x^{w_\delta}\delta_\unr$, then $\iota_{k,\an}$ and
$\iota_k $ are zero.
\item\label{it:iota-k-no-3} If $\delta=x^{w_\delta}\delta_\unr$ with $1\leq w_\delta\leq
k$, then $\iota_{k,\an}$ and $\iota_k $ are surjective, and the
kernel of $\iota_{k,\an}$ is the $1$-dimensional subspace  $\{ c\in
H_{\an}^1(\delta): c=0 \text{ or } \SL(c)=\infty\}$.
\item\label{it:iota-k-no-4} If $\delta=x^{w_\delta}$ with $1\leq w_\delta \leq k$, then
$\iota_{k,\an}$ and $\iota_k $ are injective, and the image of
$\iota_{k,\an}$ is $\{ c\in H_{\an}^1(x^{-k}\delta): c=0 \text{ or }
\SL(c)=\infty\}$.
\end{enumerate}
\end{prop}
\begin{proof} We will
use the exact sequence (\ref{eq:exact-sq}) frequently without
mentioning it.

First we prove (\ref{it:iota-k-no-2}). From the fact that
$\dim_LH_{\an}^0(t_\CF^{-k}\SR_L(\delta)/\SR_L(\delta)) = \dim_L
H_{\an}^1(\delta) =1 $ and $H_{\an}^0(x^{-k}\delta)=0$, we obtain
the assertion for $\iota_{k,\an}$. The assertion for $\iota_k $
follows from this and the commutative diagram (\ref{eq:comm-iota})
where the two vertical maps are isomorphisms.

Next we prove (\ref{it:iota-k-no-3}). From the fact that
$$H_{\an}^0(x^{-k}\delta)=0, \hskip 5pt
\dim_LH_{\an}^0(t_\CF^{-k}\SR_L(\delta)/\SR_L(\delta)) =1, \hskip
5pt \dim_L H_{\an}^1(\delta) =2 \ \ \text{and}  \ \dim_L
H_{\an}^1(x^{-k}\delta)=1,$$ we obtain the surjectivity of
$\iota_{k,\an}$. The surjectivity of $\iota_k $ follows from this
and the commutative diagram (\ref{eq:comm-iota}) where the two
vertical maps are isomorphisms. We show that, if $ c \in
H_{\an}^1(\delta) $ satisfies $\SL(c)=\infty$, then
$\iota_{k,\an}(c)=0$. As $\SL(\iota_{w_\delta-1,\an}(c))=\infty$ and
$\iota_{k,\an}=\iota_{k+1-w_\delta,\an}\iota_{w_\delta-1,\an}$, we
reduce to the case of $\delta=x\delta_\unr$. In this case, $c=
t_\CF^{-1}\lambda((q^{-1}\varphi_q-1)  G(1,1) , \nabla
 G(1,1))$ with $\lambda\in L$. Thus $\iota_{1,\an}(c)=\lambda( (q^{-1}\varphi_q-1) G(1,1) ,
\nabla G(1,1)) \sim (0,0)$. Hence $\iota_{k,\an}(c)=0$ for any
integer $k\geq 1$.

Finally we prove (\ref{it:iota-k-no-4}). From the fact that
$$H_{\an}^0(\delta)=0 \ \ \text{and}  \ \dim_LH_{\an}^0(x^{-k}\delta)=\dim_LH_{\an}^0(t_\CF^{-k}\SR_L(\delta)/\SR_L(\delta))
= 1,$$ we obtain the injectivity of $\iota_{k,\an}$. The injectivity
of $\iota_k $ follows from this and the commutative diagram
(\ref{eq:comm-iota}) where the vertical map
$\Upsilon^\delta_{\an,\gamma}\circ\Upsilon^\delta_\gamma$ is an
isomorphism. For the second assertion, let $(m,n)$ be in
$Z^1(x^{w_\delta})$. Then
$\iota_{w_\delta-1}(m,n)=(t^{w_\delta-1}_\CF m, t^{w_\delta-1}_\CF
n)\in Z^1(x)$. In other words, $\partial(t^{w_\delta}_\CF
m)=\nabla_x(t^{w_\delta-1}_\CF
m)=(\pi\varphi_q-1)(t^{w_\delta-1}_\CF n)$. Thus
$\Res(t^{w_\delta-1}_\CF n)=0$ and there exists $z\in \SR_L$ such
that $\partial z =t^{w_\delta-1}_\CF n$ or equivalently $\nabla z=
t^{w_\delta}_\CF n$. It follows that
$\nabla_{x^{w_\delta-k}}(t^{k-w_\delta}_\CF
z)=\big(\nabla+(w_\delta-k)\big)(t^{k-w_\delta}_\CF
z)=t_\CF^{k-w_\delta}\nabla z = t^k_\CF n$. Thus
$\iota_{k,\an}(m,n)=(t^k_\CF m, t^k_\CF n)\sim \Big(t^k_\CF m -
(\pi^{w_\delta-k}\varphi_q-1)(t^{k-w_\delta}_\CF z), 0 \Big)$. So we
have $\iota_{k,\an}(m,n)=(at^{k-w_\delta}_\CF, 0)$. If
$\iota_{k,\an}(m,n)\neq 0$ or equivalently $a\neq 0$, then
$\SL(\iota_{k,\an}(m,n))=\infty$.
\end{proof}

% Let $n$ be the unique solution of $(\frac{\pi}{q}\varphi_q-1)n=t_\CF\partial 1/u +1/u$ in $\SR^+_L$.
% Then $(1/u, n)$ is in $Z^1(x\delta_\unr)$ with $\SL(1/u,n)=\infty$.
% Thus $\partial^{w_\delta-1}(1/u,n)$ is a base of $\{ c\in H_{\an}^1(\delta): c=0 \text{ or } \SL(c)=\infty\}$. We have
% $\iota_{k,\an}(\partial^{w_\delta-1}(1/u,n))=(t^k_\CF \partial^{w_\delta-1}(1/u), t^k_\CF \partial^{w_\delta-1} n).$
% It is in $B^1(x^{-k}\delta)$ since $t^k_\CF \partial^{w_\delta-1}(1/u)$ and $t^k_\CF \partial^{w_\delta-1} n$ are in $\SR^+_L$.

\section{Triangulable % $\CO_F $-analytic
$(\varphi_q,\Gamma)$-modules of rank $2$} \label{sec:tri}

In his paper \cite{tri}, Colmez classified $2$-dimensional
trianguline representations of the Galois group $G_{\BQ_p}$. Later
Nakamura \cite{Nakamura} classified $2$-dimensional trianguline
representations of the Galois group of a $p$-adic local field that
is finite over $\BQ_p$, generalizing Colmez's work.

In this section we classify triangulable $\CO_F$-analytic
$(\varphi_q,\Gamma)$-modules of rank $2$ following Colmez's method
\cite{tri}. First we recall the definition.

\begin{defn} A $(\varphi_q,\Gamma)$-module over $\SR_L$ is called
{\it triangulable} if there exists a filtration of $D$ consisting of
$(\varphi_q,\Gamma)$-submodules $0=D_0\subset D_1\subset \cdots
\subset D_d=D$ such that $D_i/D_{i-1}$ is free of rank $1$ over
$\SR_L$.
\end{defn}

Note that, if $D$ is $\CO_F $-analytic, then so is $D_i/D_{i-1}$ for
any $i$.

% \subsection{Triangulable $\CO_F $-analytic $(\varphi_q,\Gamma)$-modules of rank $2$}

If $\delta_1,\delta_2\in\SI_\an(L)$, then $\Ext(\SR_L(\delta_2),
\SR_L(\delta_1))$ is isomorphic to $\Ext(\SR_L,
\SR_L(\delta_1\delta_2^{-1}))$, or $H^1(\delta_1\delta_2^{-1})$. The
isomorphism only depends on the choices of $e_{\delta_1}$,
$e_{\delta_2}$ and $e_{\delta_1\delta_2^{-1}}$. Thus it is unique up
to a nonzero multiple and induces an isomorphism from
$\Proj(\Ext(\SR_L(\delta_2), \SR_L(\delta_1)))$ to
$\Proj(H^1(\delta_1\delta_2^{-1}))$ independent of the choices of
$e_{\delta_1}$, $e_{\delta_2}$ and $e_{\delta_1\delta_2^{-1}}$.
Similarly there is a natural isomorphism from
$\Proj(\Ext_\an(\SR_L(\delta_2), \SR_L(\delta_1)))$ to
$\Proj(H_\an^1(\delta_1\delta_2^{-1}))$. Hence the set of
triangulable (resp. triangulable and $\CO_F$-analytic)
$(\varphi_q,\Gamma)$-modules $D$ of rank $2$ satisfying the
following two properties is classified by
$\Proj(H^1(\delta_1\delta_2^{-1}))$ (resp.
$\Proj(H_\an^1(\delta_1\delta_2^{-1}))$):

$\bullet$  $\SR_L(\delta_1)$ is a saturated
$(\varphi_q,\Gamma)$-submodule of $D$ and $\SR_L(\delta_2)$ is the
quotient module,

$\bullet$ $D$ is not isomorphic to $\SR_L(\delta_1)\oplus
\SR_L(\delta_2)$.

% Let $\SS$ be the parameter space defined in the introduction. Recall

Let $\SS^\an=\SS^\an(L)$ be the analytic variety obtained by blowing
up $(\delta_1,\delta_2)\in\SI_\an(L)\times \SI_\an(L)$ along the
subvarieties $\delta_1\delta_2^{-1}=x^i\delta_\unr$ for $i\in
\BZ_{+}$ and the subvarieties $\delta_1\delta_2^{-1}=x^{-i}$ for
$i\in \BN$. The fiber over the point $(\delta_1,\delta_2)$ is
isomorphic to $\Proj(H_{\an}^1(\delta_1\delta_2^{-1}))$. Similarly
let $\SS=\SS(L)$ be the analytic variety over $\SI_\an(L)\times
\SI_\an(L)$ whose fiber over $(\delta_1,\delta_2)$ is isomorphic to
$\Proj(H^1(\delta_1\delta_2^{-1}))$. The inclusions
$\Ext_\an(\SR_L(\delta_1),\SR_L(\delta_2))\hookrightarrow
\Ext(\SR_L(\delta_1),\SR_L(\delta_2))$ for
$\delta_1,\delta_2\in\SI_\an(L)$ induce a
natural injective map $\SS^\an \hookrightarrow \SS$. %
We write points of $\SS$ (resp. $\SS^\an$) in the form $(\delta_1,
\delta_2, c)$ with $c\in \Proj(H^1(\delta_1\delta_2^{-1}))$ (resp.
$c\in \Proj(H_{\an}^1(\delta_1\delta_2^{-1}))$). If
$(\delta_1,\delta_2,c)\in \SS$ is in the image of $\SS_\an$, for our
convenience we use $c_\an$ to denote the element in
$\Proj(H_{\an}^1(\delta_1\delta_2^{-1}))$ corresponding to $c$. For
$(\delta_1,\delta_2,c)\in \SS^\an$, since the $\SL$-invariant
induces an inclusion
$\Proj(H_{\an}^1(\delta_1\delta_2^{-1}))\hookrightarrow \RP^1(L)$,
we also use $(\delta_1,\delta_2, \SL(c))$ to denote
$(\delta_1,\delta_2, c)$.

If $s\in \SS$, we assign to $s$ the invariant $w(s)\in L$ by $
w(s)=w_{\delta_1}-w_{\delta_2}. $ Let $\SS_+$ be the subset of $\SS$
consisting of elements $s\in \SS$ with
$$v_\pi(\delta_1(\pi))+v_\pi(\delta_2(\pi))=0, \hskip 10pt
v_\pi(\delta_1(\pi))\geq 0.$$
% Here $v_L$ is the valuation such that $v_L(\pi)=1$.
If $s\in \SS_+$, we assign to $s$ the invariant $u(s)\in \BQ_+$ by
$$ u(s) =v_\pi(\delta_1(\pi)) =-v_\pi(\delta_2(\pi)) .
$$
Put $\SS_0=\{s\in \SS_+ \ |\ u(s)=0\}$ and $\SS_*=\{s\in\SS_+ \ | \
u(s)>0\}$. Then $\SS_+$ is the disjoint union of $\SS_0$ and
$\SS_*$. For $?\in \{+, 0,
*\}$ we put $\SS^\an_?=\SS^\an\cap \SS_?$. We decompose the set $\SS^\an_?$ as $\SS^\an_?=\SS_?^\ng\coprod
\SS_?^\cris\coprod \SS_?^\st \coprod \SS_?^\ord \coprod \SS_?^\ncl$,
where {\allowdisplaybreaks
\begin{eqnarray*} \SS_?^\ng &=& \{ s\in\SS_? \ |\ w(s) \text{ is not
an integer }\geq 1 \},
\\
\SS_?^\cris &=& \{ s\in \SS_? \ |\ w(s) \text{ is an integer }\geq 1
, u(s)<w(s), \SL=\infty \},
\\
\SS_?^\st &=& \{ s\in \SS_? \ |\ w(s) \text{ is an integer }\geq 1 ,
u(s)<w(s), \SL\neq \infty \}
\\
\SS_?^\ord &=& \{ s\in \SS_? \ |\ w(s) \text{ is an integer }\geq 1
, u(s)=w(s) \},
\\
\SS_?^\ncl &=& \{ s\in \SS_? \ |\ w(s) \text{ is an integer }\geq 1
, u(s)>w(s) \} .
\end{eqnarray*} }
\vskip -15pt \noindent Note that $\SS_0^\ord$ and $\SS_0^\ncl$ are
empty.

Let $D$ be an extension of $\SR_L(\delta_2)$ by $\SR_L(\delta_1)$.
For any $k\in \BN$, the preimage of $t_\CF^k \SR_L(\delta_2)$ is a
$(\varphi_q,\Gamma)$-submodule of $D$, which is denoted by $D'$.
Then $D'$ is an extension of $\SR_L(x^k\delta_2)$ by
$\SR_L(\delta_1)$. If $D$ is $\CO_F $-analytic, then so is $D'$.

\begin{lem}\label{lem:split}
\begin{enumerate}
\item The class of $D'$ in $H^1(\delta_1\delta_2^{-1}x^{-k})$
coincides with $\iota_k(c)$ up to a nonzero multiple, where $c$ is
the class of $D$ in $H^1(\delta_1\delta_2^{-1})$.
\item If $D$ is $\CO_F $-analytic, the class of $D'$ in $H_{\an}^1(\delta_1\delta_2^{-1}x^{-k})$
coincides with $\iota_{k,\an}(c)$ up to a nonzero multiple, where
$c$ is the class of $D$ in $H_{\an}^1(\delta_1\delta_2^{-1})$.
\end{enumerate}
\end{lem}
\begin{proof} We only prove (b). The proof of (a) is similar. Let $e$ be a basis of $\SR_L(\delta_2)$ such that
$\varphi_q(e)=\delta_2(\pi)e$ and $\sigma_a e=\delta_2(a)e$. Let
$\tilde{e}$ be a lifting of $e$ in $D$. The class of $D$, or the
same, $c$, coincides with the class of $\Big(
(\delta_2(\pi)^{-1}\varphi_q-1)\tilde{e},
(\nabla-w_{\delta_2})\tilde{e} \Big)$ up to a nonzero multiple.
Similarly, up to a nonzero multiple, the class of $D'$ coincides
with the class of
$$\Big(
(\pi^{-k}\delta_2(\pi)^{-1}\varphi_q-1)(t_\CF^k\tilde{e}),
(\nabla-w_{\delta_2}-k)(t_\CF^k\tilde{e})
\Big)=\big(t_\CF^k(\delta_2(\pi)^{-1}\varphi_q-1)\tilde{e},
t_\CF^k(\nabla-w_{\delta_2})\tilde{e} \Big)$$ which is exactly
$\iota_{k,\an}(c)$.
\end{proof}

\begin{prop}\label{prop:equiv}
Put $D=D(s)$ with $ s = ( \delta_1, \delta_2, c) \in \SS $. Then the
following two conditions are equivalent:
\begin{enumerate}
\item \label{it:equi-no1}
$D(s)$ has a $(\varphi_q,\Gamma)$-submodule $M$ of rank $1$ such
that $M\cap \SR_L(\delta_1)=0$;
\item \label{it:equi-no2} $s$ is in $\SS^\an$ and satisfies $w(s)\in \BZ_+$,
$\delta_1\delta_2^{-1}\neq x^{w(s)}$ and $\SL(c_\an)=\infty$.
\end{enumerate}
Among all such $M$ there exists a unique one, $M_\sat$, that is
saturated; $M_\sat$ is isomorphic to $\SR_L(x^{w(s)}\delta_2)$. For
any $M$ that satisfies Condition (\ref{it:equi-no1}), there exists
some $i\in \BN$ such that $M=t_\CF^i M_\sat$.
\end{prop}
\begin{proof}
Assume that $D(s)$ satisfies Condition (\ref{it:equi-no1}). Since
the intersection of $M$ and $\SR_L(\delta_1)$ is zero, the image of
$M$ in $\SR_L(\delta_2)$ is a nonzero $(\varphi_q,\Gamma)$-submodule
of $\SR_L(\delta_2)$, and so must be of the form $t_\CF^k
\SR_L(\delta_2)$ with $k\in \BN$. Since $D(s)$ does not split, we
have $k\geq 1$. The preimage of $t_\CF^k \SR_L(\delta_2)$ in $D$ is
exactly $M\oplus \SR_L(\delta_1)$. Since $M\oplus \SR_L(\delta_1)$
splits, by Lemma \ref{lem:split} we have $\iota_{k}(c)=0$. By
Proposition \ref{prop:iota-k} this happens only if $w(s)\in \{
1,\cdots, k\}$ and $\delta_1\delta_2^{-1}\neq x^{w(s)}$. Note that,
when $w(s)\in \{ 1,\cdots, k\}$ and $\delta_1\delta_2^{-1}\neq
x^{w(s)}$, $D(s)$ is automatically $\CO_F $-analytic.
% So we may write $s$ in the form $(\delta_1,\delta_2,c_\an)$ with $c_\an\in H^1_\an(\delta_1\delta_2^{-1})$.
Again by Proposition \ref{prop:iota-k} we obtain
$\SL(c_\an)=\infty$. This proves
(\ref{it:equi-no1})$\Leftrightarrow$(\ref{it:equi-no2}).

If (\ref{it:equi-no1}) holds, then the preimage of
$t_\CF^{w(s)}\SR_L(\delta_2)$ splits as $\SR_L(\delta_1)\oplus M_0$,
where $M_0$ is isomorphic to $\SR_L(x^{w(s)}\delta_2)$. We show that
$M_0$ is saturated. Note that $M_0$ is not included in $t_\CF D(s)$.
Otherwise, the preimage of $t_\CF^{w(s)-1}\SR_L(\delta_2)$ will
split, which contradicts Proposition \ref{prop:iota-k}. Let $e_1$
(resp. $e_2$, $e$) be a basis of $\SR_L(\delta_1)$ (resp.
$\SR_L(\delta_2)$, $M_0$) such that $Le_1$ (resp. $L e_2$, $Le$) is
stable under $\varphi_q$ and $\Gamma$. Let $\tilde{e}_2$ be a
lifting of $e_2$. Write $e=ae_1+b\tilde{e}_2$. Then $a\notin t_\CF
\SR_L$ and $b\in t_\CF^{w(s)}\SR_L$. Observe that the ideal $I$
generated by $a$ and $t_\CF^{w(s)}$ satisfies $\varphi_q(I)=I$ and
$\gamma(I)=I$ for all $\gamma\in\Gamma$. Thus by Lemma
\ref{lem:Berger}, $I=\SR_L$. It follows that $M_0$ is saturated. If
$M$ is another $(\varphi_q,\Gamma)$-submodule of $D(s)$ such that
$M\cap \SR_L(\delta_1)=0$, then the image of $M$ in
$\SR_L(\delta_2)$ is $t_\CF^k\SR_L(\delta_2)$ for some integer
$k\geq w(s)$. Then $M\subset \SR_L(\delta_1)\oplus M_0$. Since
$\delta_1\neq \delta_2x^{w(s)}$, $\SR_L(\delta_1)$ has no nonzero
$(\varphi_q,\Gamma)$-submodule isomorphic to $\SR_L(x^k\delta_2)$.
It follows that $M\subset M_0$ and thus $M=t_\CF^{k-w(s)}M_0$.
\end{proof}

\begin{cor}\label{cor:sat} Let $s=(\delta_1,\delta_2,c)$ be in $\SS$. If $s$ is in $\SS^\an$ and satisfies $w(s)\in \BZ_+$,
$\delta_1\delta_2^{-1}\neq x^{w(s)}$ and $\SL(c_\an)=\infty$, then
$D(s)$ has exactly two saturated $(\varphi_q,\Gamma)$-submodules of
$D(s)$ of rank $1$, one being $\SR_L(\delta_1)$ and the other
isomorphic to $\SR_L(x^{w(s)}\delta_2)$. Otherwise, $D(s)$ has
exactly one saturated $(\varphi_q,\Gamma)$-submodule of rank $1$
which is $\SR_L(\delta_1)$.
\end{cor}

\begin{cor}\label{cor:classify} Let $s=(\delta_1,\delta_2, c)$ and
$s'=(\delta'_1,\delta'_2, c')$ be in $\SS(L)$.
\begin{enumerate}
\item\label{it:cor-iso-no1} If $\delta_1=\delta'_1$, then $D(s)\cong D(s')$ if and only if
$s=s'$.
\item\label{it:cor-iso-no2} If $\delta_1\neq \delta'_1$, then $D(s)\cong D(s')$ if and
only if $s$ and $s'$ are in $\SS^\an$ and satisfy $w(s)\in \BZ_+$,
$\delta'_1=x^{w(s)}\delta_2$, $\delta'_2=x^{-w(s)}\delta_1$ and
$\SL(c_\an)=\SL(c'_\an)=\infty$.
\end{enumerate}
\end{cor}
\begin{proof} Assertion (\ref{it:cor-iso-no1}) is clear. We prove (\ref{it:cor-iso-no2}). Since $D(s)\cong D(s')$, there
exists a $(\varphi_q,\Gamma)$-submodule $M$ of $D(s)$ such that
$M\cong \SR_L(\delta'_1)$ and $D(s)/M\cong \SR_L(\delta'_2)$. Since
both $\SR_L(\delta_1)$ and $M$ are saturated
$(\varphi_q,\Gamma)$-submodules of $D$, $\SR_L(\delta_1)\cap M=0$.
By Proposition \ref{prop:equiv} we have $w(s)\in \BZ_+$,
$\delta_1\delta_2^{-1}\neq x^{w(s)}$, $\SL(c_\an)=\infty$ and
$\delta'_1=x^{w(s)}\delta_2$. Similarly,
$\delta_1=x^{w(s')}\delta'_2$. As
$\delta_1\delta_2=\delta'_1\delta'_2$, we have $w(s)=w(s')$.
\end{proof}

\begin{prop}\label{prop:slope} Let $s=(\delta_1,\delta_2, c)$ be in $\SS$. Then $D(s)$
is of slope zero if and only if $s\in \SS_+ - \SS_+^\ncl$; $D(s)$ is
of slope zero and the Galois representation attached to $D(s)$ is
irreducible if and only if $s$ is in $\SS_\ast -(\SS_\ast^\ord\cup
\SS_\ast^\ncl)$; $D(s)$ is of slope zero and $\CO_F$-analytic if and
only if $s\in \SS^\an_+ - \SS_+^\ncl$.
\end{prop}
\begin{proof} By Kedlaya's slope filtration
theorem, $D(s)$ is of slope zero if and only if
$v_\pi(\delta_1(\pi)\delta_2(\pi))=0$ and $D(s)$ has no
$(\varphi_q,\Gamma)$-submodule of rank $1$ that is of slope $<0$. In
particular, if $D(s)$ is of slope zero, then
$v_\pi(\delta_1(\pi))\geq 0$ and thus $s\in \SS_+$. Hence we only
need to consider the case of $s\in \SS_+$. Assume that $D(s)$ has a
$(\varphi_q,\Gamma)$-submodule of rank $1$, say $M$, that is of
slope $<0$. Then the intersection of $M$ and $\SR_L(\delta_1)$ is
zero. By Proposition \ref{prop:equiv} we may suppose that $M$ is
saturated. By Corollary \ref{cor:sat}, this happens if and only if
$s$ is in $\SS^\an$ and satisfies $w(s)\in\BZ_+$,
$\delta_1\delta_2^{-1}\neq x^{w(s)}$, $\SL(c_\an)=\infty$ and
$w(s)<u(s)$. Note that $\delta_1\delta_2^{-1}\neq x^{w(s)}$ and
$\SL(c_\an)=\infty$ automatically hold when $0<w(s)<u(s)$. The first
assertion follows. Similarly, $D(s)$ has a saturated
$(\varphi_q,\Gamma)$-submodule of rank $1$ that is of slope zero, if
and only if $u(s)=0$ or $u(s)=w(s)$. By Proposition
\ref{prop:faithful} (\ref{it:over-2}) and Remark \ref{rem:1-over},
we know that the Galois representation attached to an \'etale
$(\varphi_q,\Gamma)$-module $D$ over $\SR_L$ of rank $2$ is
irreducible if and only if $D$ has no \'etale
$(\varphi_q,\Gamma)$-submodule of rank $1$. This shows the second
assertion. The third assertion follows from the first one.
\end{proof}

\vskip 5pt

\noindent{\it Proof of Theorem \ref{thm:cl}.} Assertion
(\ref{it:cl-1}) follows from Proposition \ref{prop:slope}, and
(\ref{it:cl-2}) follows from Corollary \ref{cor:classify}. \qed

\vskip 10pt

\begin{rem} Let $s\neq s'$ be as in Theorem \ref{thm:cl}
(\ref{it:cl-2}). Then $s\in \SS_*^\cris$ if and only if $s'\in
\SS_*^\cris$; $s\in \SS_+^\ord$ if and only if $s'\in \SS_0^\cris$.
\end{rem}

\begin{rem} By an argument similar to that in \cite{tri} one can show that, if $s$ is in $\SS_+^\cris$ (resp. $\SS_+^\ord$,
$\SS_+^\st$), then $D(s)$ comes from a crystalline (resp. ordinary,
semistable but non-crystalline) $L$-representation twisted by a
character.
\end{rem}

{\bf Acknowledgement}: Both authors thank Professor P. Colmez and
Professor L. Berger for helpful advices on revising the original
version of this paper. The second author thanks the hospitality and
stimulating environment provided by Beijing International Center for
Mathematical Research, where a part of this research was carried
out. The second author also thanks Professor Q. Tian and Professor
C. Zhao for encouragements. The second author is supported by the
NSFC grant 11101150 and the doctoral fund for new teachers
20110076120002.

\end{document}